\documentclass[11pt,a4paper,reqno]{amsart}

 \usepackage{cases,hyperref,tikz} \usepackage{xcolor,appendix,geometry}
\usepackage{mathrsfs}
\usepackage{epsfig}
\usepackage{fp,ifthen}
\usepackage{float}
\usepackage{graphicx}
\usepackage{enumitem}
\usepackage{esint}
\usepackage{accents}

\usepackage{mathtools}
\usepackage{cite}
\usepackage{amsmath,amssymb,amsthm,esint,empheq,etoolbox}

\title[A measure associated to $GBD$]{A matrix-valued measure associated \\to  the derivatives of a function \\ of generalised bounded deformation}
\author{Gianni Dal Maso}
\address{SISSA, via Bonomea 265, Trieste, Italy}
\email{dalmaso@sissa.it}
\author{Davide Donati}
\address{SISSA, via Bonomea 265, Trieste, Italy}
\email{ddonati@sissa.it}
\numberwithin{equation}{section}
\theoremstyle{plain} 
\newtheorem{theorem}{Theorem}[section] \newtheorem{lemma}[theorem]{Lemma} 
\newtheorem{proposition}[theorem]{Proposition}
\newtheorem{corollary}[theorem]{Corollary}
\theoremstyle{definition} 
\newtheorem{definition}[theorem]{Definition}
 \newtheorem{remark}[theorem]{Remark}

\newcommand{\Eom}{\mathfrak{E}}
\newcommand\e{\varepsilon}
\newcommand{\N}{\mathbb{N}}
\newcommand{\Z}{\mathbb{Z}}
\newcommand\R{\mathbb{R}}
\newcommand{\B}{\mathcal{B}}
\newcommand{\E}{\mathcal{E}}
\newcommand\Rd{\mathbb{R}^d}
\newcommand\Sd{\mathbb{S}^{d-1}}

\newcommand{\Rdsym}{\mathbb{R}^{d\times d}_{\rm sym}}
\newcommand{\hd}{\mathcal{H}^{d-1}}
\newcommand\Hd[1]{\mathcal{H}^{d-1}(#1)}
\newcommand{\huno}{\mathcal{H}^1}
\newcommand{\Lb}{\mathcal{L}}
\newcommand{\Ld}{\mathcal{L}^d}
\newcommand{\hzero}{\mathcal{H}^0}
\newcommand{\mres}
{\mathbin{\vrule height 1.6ex depth 0pt width
0.13ex\vrule height 0.13ex depth 0pt width 1.3ex}}
\mathchardef\emptyset="001F

\theoremstyle{remark}
\begingroup
\endgroup

\begin{document}

\thanks{Preprint SISSA 04/2025/MATE}
\begin{abstract} 
\smallskip
We associate to every function $u\in GBD(\Omega)$ a  measure $\mu_u$ with values in the space of  symmetric matrices, which generalises the distributional symmetric gradient $Eu$ defined for  functions of bounded deformation. We show that this measure $\mu_u$ admits a decomposition as the sum of three mutually singular matrix-valued measures $\mu^a_u$, $\mu^c_u$, and $\mu^j_u$, the absolutely continuous part, the Cantor part, and the jump part, as in the case of $BD(\Omega)$ functions. We then characterise the space $GSBD(\Omega)$, originally defined only by slicing, as the space of functions $u\in GBD(\Omega)$ such that $\mu^c_u=0$. 

\vspace{0.5 cm}


\end{abstract}

\maketitle

\section{Introduction}
Given a bounded open set $\Omega\subset \Rd$, with $d\geq 1$, the spaces $GBD(\Omega)$ of {\it  functions of generalised bounded deformation} and  $GSBD(\Omega)$ of {\it  special functions of generalised bounded deformation} were introduced in \cite{DalMasoJems} to provide a functional framework for variational problems related to  Griffith's energy in fracture mechanics (see  \cite{BourFrancMari,FrancfortMarigo}). The main feature of these spaces is that they avoid the unnatural $L^\infty$ {\it a priori} bounds typically required for 
compactness in the space $SBD(\Omega)$ (see \cite{BellettiniCoscia}). 

Thanks to the very weak requirements appearing in the definitions of $GBD(\Omega)$ and $GSBD(\Omega)$, it was shown in \cite{DalMasoJems} that compactness in these spaces is achieved under very mild assumptions (see also \cite{AlmiTassoComp,ChambolleCrismaleCompInd}).
For the space $GSBD(\Omega)$ compactness results under even weaker conditions have later been obtained
by Friedrich and Solombrino \cite{FriedrichSolombrino} in the planar case, and Chambolle and Crismale in the general case \cite{ChambolleCrismComp}. These results are similar to those available in the more restrictive setting of functions of bounded variation (see \cite{AmbrosioExistence,CompactnessFried,gsbvpconcentrationcompactness}). These advancements allow one to solve, in a weak sense,  minimisation problems concerning  Griffith's functional, thereby justifying the use of $GSBD(\Omega)$ for {\it brittle} models in fracture mechanics. Many more applications of the space $GSBD(\Omega)$ were also considered in the recent literature, see, for instance, \cite{AlmiDavoliFried,ChambolleContiIurlano,FriedrichStrong,ContiFocardiIUrlGriff,WeakFormulationTasso,CrismFried,AlmiTassoBrittle,DeGiorgiconj,CrismaleFriedrichSeutter}.

The study of {\it cohesive }models for fracture mechanics in the {\it anti-plane} case carried out in \cite{DalToaJConv,DalToaARMA} suggests that $GBD(\Omega)$ should be the appropriate space for the study of minimisation problems connected with these models, when the anti-plane 
hypothesis is dropped. This requires to extend to $GBD(\Omega)$ the structure theorems proved in \cite{AmbrCosciaDalM} for the space $BD(\Omega)$ of {\it functions of bounded deformation} (see \cite{AmbrCosciaDalM, Temam1,Temam2, Suquet}),  in analogy with what was done in \cite{DalToaNoDea} in the setting of functions of bounded variation.

The analysis of the {\it fine properties} of  functions in $GBD(\Omega)$ carried out in \cite{DalMasoJems} reveales that many of the key structural features of the space $BD(\Omega)$ of {\it functions of bounded deformation} (see \cite{AmbrCosciaDalM, Temam1,Temam2, Suquet}) naturally extend to $GBD(\Omega)$, albeit with suitable modifications to account for the weaker regularity of  these functions.  
The fine properties of $BD(\Omega)$ were thoroughly examined in \cite{AmbrCosciaDalM}, where the authors show for a function $u\colon\Omega\to\Rd$ in $BD(\Omega)$ that the following conditions hold:
\begin{enumerate}
    \item[(a)] $u$ admits an {\it approximate gradient} $\nabla u \in L^1(\Omega;\R^{d\times d})$, where $\R^{d\times d}$ is the space of $d\times d$ matrices with real entries, and $\E u:=(\nabla u+\nabla u^T)/2$ defines an {\it approximate symmetric gradient} of $u$ (see also \cite{PiotrHaj}); moreover, if $\hd$ is the $(d-1)$-dimensional Hausdorff measure, for every $\xi\in\Sd:=\{\xi\in\Rd:|\xi|=1\}$ and for $\hd$-a.e\ $y\in \Pi^\xi:=\{y\in\Rd:y\cdot\xi=0\}$,  the one-dimensional scalar function  $t\mapsto u^\xi_y(t):= u(y+t\xi)\cdot\xi$  has bounded variation and
     \begin{equation}\label{slicing intro}
       \hspace{2 cm} \E u(y+t\xi)\xi\cdot\xi=\nabla u^\xi_y(t)\quad \text{  for $\Lb^1$-a.e.\ $t\in\{s\in\R:y+s\xi\in\Omega\}$},
    \end{equation}
   where $\Lb^1$ is the one-dimensional Lebesgue measure and $\nabla u^\xi_y$ denotes the absolutely continuous part of the distributional derivative $Du^\xi_y$ of $u^\xi_y$;
    \item [(b)] the jump set $J_u$ (see Definition \ref{def Jump set}) is $(\hd,d-1)$-rectifiable (see \eqref{def rectifiable} and also \cite{DelNin}), with measure theoretical unit normal $\nu_u$; in addition, for every $\xi\in\Sd$ and for $\hd$-a.e.\ $y\in \Pi^\xi$, setting  $[u]:=u^+-u^-$, where $u^+$ and $u^-$ are the unilateral traces of $u$ on $J_u$, and $(J^\xi_u)^\xi_y:=\{t\in\R:x=y+t\xi\in J_u \text{ and }[u](x)\cdot\xi\neq 0\}$, we have 
    \begin{gather*}
        J_{u^\xi_y}=(J^\xi_u)^\xi_y \,\,\,\text{ and }\,\,\,[u](y+t\xi)\cdot\xi=[u^\xi_y](t) \,\,\,\text{   for every }t\in(J^\xi_u)^\xi_y;
    \end{gather*}
    \item [(c)] the distributional symmetric gradient $Eu:=(Du+Du^T)/2$, which by definition is a bounded Radon measure taking values in the space  $\Rdsym$ of $d\times d$ symmetric matrices, can be decomposed as the sum of three mutually singular measures:
    \begin{equation}\label{decomposition Eu}
        Eu=(\E u) \Lb^d+E^cu+([u]\odot\nu_u)\hd\mres J_u,
    \end{equation}
    where $\Lb^d$ is the Lebesgue measure in $\Rd$, $E^cu$ is called the {\it Cantor part} of $Eu$, $\odot$ is the symmetric tensor product, and $\hd\mres J_u$ is the restriction of $\hd$ to $J_u$; the Cantor part $E^cu$ is singular with respect to $\Ld$ and vanishes on all Borel sets that are $\sigma$-finite with respect to $\hd$.
\end{enumerate}

In $GBD(\Omega)$ property (c) would not make sense, since, in general, the symmetrised gradient $Eu$ cannot be defined in the sense of distributions. In particular, it is not clear what is the analogue of $E^cu$ for a function $u\in GBD(\Omega)$. Understanding how to generalise this term is crucial for possible applications to cohesive fracture mechanics, as shown in the corresponding problems for the anti-plane case (see \cite{DalToaJConv,DalToaARMA}). 

However, for $u\in GBD(\Omega)$ it is shown  in \cite{DalMasoJems}  that property (b) still holds and also that $u$ admits an {\it approximate symmetric gradient $\E u$} that enjoys the slicing property \eqref{slicing intro}. It is still an open question whether {\it every} $GBD(\Omega)$ function admits an approximate gradient.

In this paper we extend the analysis of the fine properties of functions in $GBD(\Omega)$ by introducing two {\it matrix-valued} measures $\mu_u$ and $\mu^c_u$, which are closely related to the measures $Eu$ and $E^cu$ when $u\in BD(\Omega)$.  To describe this result, we fix some notation.  Given  $R>0$,  let $\tau_R\colon\R\to\R$ be  the $1$-Lipschitz function defined by 
\begin{equation}\notag
    \tau_R(s):=\begin{cases}-\frac{R}{2}&\text{ if }s\leq -\frac{R}{2},\\
        s& \text{if }\frac{-R}{2}\leq s\leq \frac{R}{2},\\
        \frac{R}{2}&\text{ if }s\geq \frac{R}{2}.
        \end{cases}
\end{equation}
It follows from the definition of $GBD(\Omega)$ (see Definition \ref{def:GBD} and Remark \ref{Lipschitz}) that for every $\xi\in\Sd$ and $R>0$ the distributional derivative of $\tau_R(u\cdot\xi)$ in direction $\xi$, denoted by $D_\xi(\tau_R(u\cdot\xi))$, is a scalar-valued bounded Radon measure on $\Omega$. 

The main result of our paper is that (see Corollary \ref{main corollary}) 
    for every $u\in GBD(\Omega)$ and $r>0$, there exists a bounded Radon measure $\mu_{u,r}$ on $\Omega$  with values in $\Rdsym$ such that for every $\xi\in\Sd$ we have
\begin{equation}\label{intro theorem}
\mu_{u,r}(B)\xi\cdot\xi=\lim_{R\to +\infty}D_\xi(\tau_R(u\cdot\xi))(B)
    \end{equation}
   for every Borel set $B\subset \Omega$  such that $B\cap J^r_u=\emptyset$, where  $J^r_u:=\{x\in J_u:|[u](x)|\geq r\}$.
   
 If $u\in BD(\Omega)$ we can see that $\mu_{u,r}(B)=(Eu)(B)$ for every Borel set $B\subset \Omega$ with $B\cap J^r_u=\emptyset$ (see Remark \ref{derivative BDsecond}).
In the general case $u\in GBD(\Omega)$, the measure $\mu_{u,r}$ is not always the symmetrised distributional gradient $Eu$, and its connection with the distributional derivatives of $u$ is given only by \eqref{intro theorem}, which takes into account the directional derivatives of  suitable truncations of the scalar components of $u$. However, the measure $\mu_{u,r}$ enjoys many of the formal properties of $Eu$ and in particular an analogue of property (c) above holds for $\mu_{u,r}$.

More precisely, in Corollary \ref{main corollary}, we prove that for every $u\in GBD(\Omega)$ the measure $\mu_{u,r}$ can be decomposed as the sum of three mutually singular measures:
\begin{equation}\label{decomposition mcu} 
\mu_{u,r}=\mu_u^a+\mu_u^c+\mu_{u,r}^j,
\end{equation}
where for every Borel set $B\subset \Omega$
\begin{gather*}
    \mu^a_u(B)=\int_B\E u\,{\rm d}x,\qquad
    \mu^j_{u,r}(B)=\int_{(J_u\setminus J^r_u)\cap B}([u]\odot \nu_u)\,{\rm d}\hd,
\end{gather*}
and $\mu^c_u$ is a singular measure (with respect to the $d$-dimensional Lebesgue measure) with values in $\Rdsym$ that vanishes on all  $\sigma$-finite Borel sets with respect to $\hd$ (see Proposition \ref{prop sigma finito}). We remark that both $\mu^a_u$ and $\mu^c_u$ do not depend on $r$.

A slicing property  of the measure $\mu_u^c$,  obtained in Proposition \ref{representation}, allows us to characterise the space $GSBD(\Omega)$, introduced in  \cite[Definition 4.2]{DalMasoJems} using slicing arguments, as the set of functions  $u\in GBD(\Omega)$  such that $\mu^c_u=0$, in analogy with what happens for $SBD(\Omega)$ (see \cite[Definition 4.6]{AmbrCosciaDalM}).  Combining this result with the recent characterisation of the space $GBD(\Omega)$, proved by Chambolle and Crismale  in\cite{ChambolleCrismaleSad}, we obtain, in Theorem \ref{chambolleCrismale}, an analogous characterisation for the space $GSBD(\Omega)$. 

We now give a brief sketch of how we prove the existence of a measure $\mu_u=\mu_{u,1}$ which satisfies \eqref{intro theorem} with $r=1$.
Straightforward arguments show that the limit in the right-hand side of \eqref{intro theorem}  exists for every $\xi\in\Rd\setminus \{0\}$ (see Proposition \ref{truncation}) and  that this limit is equal to 
\begin{equation}\label{def sigma intro}
\sigma_u^\xi(B):=|\xi|\int_{\Pi^\xi}Du^\xi_y((B\setminus J^1_{u})^\xi_y)\,{\rm d}\Hd{y},
\end{equation}
where for every $E\subset \Rd$  and for every $y\in\Pi^\xi$ we set  $E^\xi_y:=\{t\in\R: y+t\xi\in E\}$.
 By the definition of $ GBD(\Omega)$ (see Definition \ref{def:GBD}) the expression above defines  a bounded Radon measure for every $\xi\in\Rd$.
To conclude, one needs to show that for every Borel set $B\subset \Omega$ there exists $\mu_u(B)\in\Rdsym$ such that 
\begin{equation}\label{eq intro}
\mu_u(B)\xi\cdot\xi=\sigma^\xi_u(B)
\end{equation}
 for every $\xi\in\Rd\setminus \{0\}$. To prove this fact, we will show that for every Borel set $B\subset \Omega$ the function $\xi\mapsto\sigma^\xi(B)$ is  2-homogeneous, lower bounded, and satisfies the  parallelogram identity, i.e.,
\begin{equation}\label{parallelogram identity}
    \sigma_u^{\xi+\eta}(B)+\sigma_u^{\xi-\eta}(B)=2\sigma_u^{\xi}(B)+2\sigma_u^{\eta}(B)
\end{equation}
for every $\xi,\eta\in\Rd$. Indeed, these three conditions imply (see Proposition \ref{character quadratic}) the existence of a symmetric matrix $\mu_u(B)$ such that \eqref{eq intro} holds. It is then easy to check that $B\mapsto \mu_u(B)$ is a bounded Radon measure. 
Since the 2-homogeneity and the lower boundedness are easily obtained, to conclude  we only need to show that \eqref{parallelogram identity} holds. 

This is done first in dimension $d=2$ by means of a discretisation argument. We fix two linearly independent vectors $\xi,\eta\in\R^2$ and  assume that $B$ is a parallelogram  with sides parallel to $\xi$ and $\eta$. For every $\zeta\in\{\xi,\eta,\xi+\eta,\xi-\eta\}$ we approximate the integral $\sigma^\zeta_u(B)$ given by \eqref{def sigma intro} by means of {\it Riemann sums} corresponding to a well-chosen grid of points $y_j$ of $\Pi^\zeta$ and write each term  $Du^\zeta_{y_j}((B\setminus J_u^1)^\zeta_{y_j})=Du^\zeta_{y_j}(B^\zeta_{y_j}\setminus (J_u^1)^\zeta_{y_j})$ as a sum over $i$ of the numbers $Du^\zeta_{y_j}(I^j_i\setminus (J_u^1)^\zeta_{y_j})$, where $I^j_i=[a_i^j,a_{i+1}^j)$ are well-chosen disjoint small intervals, whose union is the interval $B^\zeta_{y_j}$.
The points  $y_j$ and $a^j_i$ can be chosen by  projecting onto the straight lines $\Pi^\zeta$ and $\{y_j+t\zeta:t\in\R\}$ the points $x_{i,j}$ of a two-dimensional grid, constructed using discrete linear combinations of $\xi$ and $\eta$, and translated by a small vector $\omega$ to be chosen carefully. We observe that if $I^j_i\cap (J^1_u)^{\zeta}_{y_j}=\emptyset$, then \begin{equation*}
  Du^\zeta_{y_j}(I^j_i\setminus (J_u^1)^\zeta_{y_j})=u(y_j+a_{i+1}^j\zeta)\cdot\zeta-u(y_j+a_{i}^j\zeta)\cdot\zeta.
\end{equation*}This leads to an approximation of $\sigma^\zeta_u(B)$ based on the difference of the values of $u$ on neighbouring grid-points $x_{i,j}.$ Hence, if $I^j_i\cap (J^1_u)^{\zeta}_{y_j}=\emptyset$, writing carefully this approximation for every $\zeta\in\{\xi,\eta,\xi+\eta,\xi-\eta\}$ we  obtain a discrete version of the parallelogram identity \eqref{parallelogram identity}. To conclude, we have to show that in this approximation we can neglect all terms that correspond to pairs $i,j$ such that  $I^j_i\cap (J^1_u)^{\zeta}_{y_j}\neq \emptyset$.
This step constitutes the main difficulty of the proof and will be the content of Section \ref{section conclusion}. This result is obtained by regarding the sum of the contribution of these ill-behaved indices as sort of {\it Riemann sum} of arbitrarily small integrals. 

Once the planar case $d=2$ is settled, the general case $d>2$ can be obtained by a Fubini-type argument, considering a sort of two dimensional slicing in the spirit of \cite{AmbrCosciaDalM}. In our case, this is based on the properties of the restrictions of $GBD$ functions to two dimensional slices proved in \cite{DalMasoJems} (see Theorem \ref{theorem restriction} below).  

\medskip

The paper is structured as follows. In Section \ref{section preliminar} we introduce the basic notions and the necessary tools we will use throughout the paper, while in Section \ref{section Riemann} we present some technical results concerning approximation of Lebesgue integrals by means of Riemann sums. Then, we introduce in Section \ref{section auxiliary} the measures $\sigma^\xi_u$, which will be the main focus of the rest of the paper, and prove several properties of these measures. Section \ref{section main} is devoted to the proof of the quadraticity of the function $\xi\mapsto\sigma^\xi_u(B)$ in the planar case $d=2$. In Section \ref{section conclusion} we complete this proof by means of some technical arguments. This result is then extended to every dimension in Section \ref{section d larger}. In Section \ref{cantor section}, we prove the decomposition \eqref{decomposition mcu} and deduce from it several consequences. The Appendix is devoted to proving the measurability of several auxiliary functions appearing in the arguments of Section \ref{section conclusion}.  

\section{Notation and preliminary results}\label{section preliminar}
In this section we fix the notation and lay down the basic tools used in this paper. 
$\Omega$ is a bounded open set of $\Rd$ with $d\geq 1$. The scalar product in $\mathbb{R}^d$ is denoted by $\cdot$ , while the Euclidean norm of $\mathbb{R}^d$ is denoted by $|\,\,|$.
For every $\rho>0$ and $x\in\Rd$ the open ball of centre $x$ and radius $\rho$ is denoted by $B_\rho(x)$.
The unit sphere of $\Rd$ is denoted by $\mathbb{S}^{d-1}:=\{\xi\in\Rd: |\xi|=1\}$.  The vector space $\R^{d\times d}$ is identified with the space of  $d\times d$ matrices.  Given  $A\in\R^{d\times d}$, its $ij$-th component is denoted by $A_{ij}$. 
For $A\in\R^{d\times d}$ and $\xi\in\Rd$, $A\xi\in\Rd$ is defined via the standard rules of matrix multiplication.  The symbol $\Rdsym$ denotes the space of  all $d\times d$ symmetric matrices, that is, the space of those matrices $A\in\R^{d\times d}$ such that $A=A^T$, where $A^T$ is  the transpose of $A$. 
We recall that all matrices $A\in\Rdsym $ satisfy the polarisation identity, i.e.,
\begin{equation}\notag \label{polarisation identity}
     A\xi\cdot \eta=\frac{1}{4}\big(A(\xi+\eta)\cdot(\xi+\eta)-A(\xi-\eta)\cdot(\xi-\eta)\big)
\end{equation}
for every $\xi,\eta\in \Rd$. 

We recall the definition of quadratic function.
\begin{definition}\label{def quadratic}
    A function $f\colon\Rd\to \R$ is quadratic if there exists a matrix $A\in\Rdsym$ such that $f(\xi)=A\xi\cdot\xi$ for every $\xi\in\Rd$.\end{definition}
    We recall the following characterisation of quadratic functions. 
\begin{proposition}\label{character quadratic}
    A function $f\colon\Rd\to \R$ is quadratic if and only if the following conditions are satisfied:
   \begin{enumerate}
       \item [(a)]{\it 2-homogeneity: }$f(t\xi)=t^2f(\xi)$ for every $\xi\in\R^d$ and every $t\in\R$;
       \item[(b)] {\it parallelogram identity:}  for every $\xi,\eta\in\R^d$ we have \begin{equation*}
           f(\xi+\eta)+f(\xi-\eta)=2f(\xi)+2f(\eta);
       \end{equation*} 
       \item[(c)]  {\it lower bound: }there exists a constant $c>0$ such that 
    \begin{equation*}
           f(\xi)\geq -c|\xi|^2 \quad \text{for every $\xi\in\Rd$.}
       \end{equation*}
   \end{enumerate}
\end{proposition}
\begin{proof}
 Assume that (a)-(c) hold. By applying \cite[Proposition 11.9]{DalBook} to the function $g(\xi)=f(\xi)+c|\xi|^2$ we obtain a matrix $B\in\Rdsym$ such that $g(\xi)=B\xi\cdot\xi$ for every $\xi\in\Rd$. Setting $A:=B-cI$, where $I$  is the identity matrix, we obtain $f(\xi)=A\xi\cdot\xi$ for every $\xi\in\Rd$.

 The converse implication is trivial.
\end{proof}

Given two vectors $\xi,\eta\in\Rd$, the symmetric tensor product $\xi\odot\eta\in\Rdsym$  is defined by $(\xi\odot\eta)_{ij}:=\tfrac{1}{2}(\xi_i\eta_j+\xi_j\eta_i)$.
   Given two distinct points $x_1,x_2\in\Rd$ we set 
\begin{equation*}
[x_1,x_2]:=\{tx_1+(1-t)x_2: t\in [0,1]\}.
\end{equation*}
The notation naturally extends to $[x_1,x_2)$, $(x_1,x_2]$, and $(x_1,x_2)$  replacing $[0,1]$ by $[0,1)$, $(0,1]$, and $(0,1)$, respectively. Given a  set $A\subset \Omega$, we say that $A$ is relatively compact in $\Omega$ and write $A\subset \subset \Omega$ if there exists a compact set $K\subset \Omega$ such that $A\subset K$.  Given $k\in\N$ and $E\subset \R^k$ the characteristic function of $E$  is the function $\chi_{E}\colon\R^k\to \{0,1\}$ such that  $\chi_E(x)=1$ if $x\in E$ and $0$ otherwise.  

 Given a finite dimensional real normed vector space $X$, $\mathcal{M}_b(\Omega;X)$ is  the space of all bounded Radon measures with values in $X$; the indication of $X$ is omitted if $X=\R$. The symbol  $\mathcal{M}_{b}^+(\Omega)$ denotes the space of all positive bounded Radon measures.
 Given $\mu \in\mathcal{M}_b(\Omega;X)$ and $\lambda\in \mathcal{M}^+_b(\Omega)$, ${\rm d}\mu/{\rm d}\lambda$ is the Radon-Nikod\'ym derivative of $\mu$ with respect to $\lambda$.  Given a measure $\mu\in\mathcal{M}_b(\Omega;X)$, the total variation  measure of $\mu$ with respect to the norm $|\,\,|$ on $X$, is the Borel measure defined for Borel set $B\subset \Omega$ by 
 \begin{equation}\notag 
    |\mu|(B):=\sup\sum_{i\in I}|\mu(B_i)|,
 \end{equation}
 where the {\it supremum} is taken over all finite sets $I\subset \N$ and all Borel partitions $(B_i)_{i\in I}$ of $B$. Given a measure $\lambda\in \mathcal{M}_b^+(\Omega)$ and a Borel measurable function $f\colon\Omega\to X$, the symbol $f\lambda$ denotes the $X$-valued measure defined for every Borel set $B\subset \Omega$ by 
 \begin{equation*}
     f\mu(B):=\int_{B}f\,{\rm d}\lambda.
 \end{equation*}

 The $k$-dimensional Lebesgue measure and the  $k$-dimensional Hausdorff measure are denoted by $\Lb^k$  and $\mathcal{H}^k$, respectively.  Given a measure $\mu$ on $\Omega$ and a $\mu$-measurable set $B\subset \Omega$, we introduce the measure $\mu\mres B$ defined by $(\mu\mres B) (E):=\mu(B\cap E)$ for every Borel set  $E\subset \Omega$. 
 
 We say that $E\subset \R^d$ is $(\hd,d-1)$-{\it rectifiable} if there exist a collection of compact sets $(K_i)_{i\in\N}$, a collection of $(d-1)$-dimensional $C^1$ manifolds $M_i\subset \Rd$, with  $K_i\subset M_i$ and $\hd(M_i)<+\infty$ for every $i\in\N$, and a set $N_0$ with $\hd(N_0)=0$ such that 
\begin{equation}\label{def rectifiable}
    E=N_0\cup\big(\bigcup_{i\in\N}K_i\big).
\end{equation}
We refer the reader to \cite[Chapter 2]{AmbrosioFuscoPallara} and to \cite[Chapter 3]{Federer} for the properties of these sets.

\noindent{\textbf{Approximate limits.}} 
    Let $E$ be a Lebesgue measurable subset of  $\R^d$ and let $x\in \Rd$ be such that
    \begin{equation}\label{positive density}
        \limsup_{\rho \to 0}\frac{\Lb^d(E\cap B_\rho(x))}{\rho^d}>0.
    \end{equation}
    We say that an $\Lb^d$-measurable function $u\colon E\to \R^m$  has approximate limit $\widetilde{u}(x)\in\R^m$ at $x$, in symbols \begin{equation}\label{aplim}
\text{ap}\!\lim_{\substack{y\to x\\y\in E}}u(y)=\widetilde{u}(x),
    \end{equation}
    if for every $\e>0$ we have
    \begin{equation*}
         \lim_{\rho\to 0}\frac{\Lb^d(\{y\in E: |u(y)-\widetilde{u}(x)|>\e\}\cap B_\rho(x))}{\rho^d}=0. 
    \end{equation*}
  Throughout the paper the symbol $\widetilde{u}(x)$ is always used to denote the approximate limit of $u$ considered in \eqref{aplim}. By \eqref{positive density} the vector $\widetilde{u}(x)$ is uniquely defined. The set $S_u$ is defined as the complement of the points where the approximate limit exists. It is well-known that for an $\Lb^d$-measurable function $u\colon \R^d\to \R^m$ it holds $\Lb^d(S_u)=0$.

\medskip

\noindent{\textbf{Jump set.}}
We now give the definition of jump set of a measurable function.

    \begin{definition} \label{def Jump set} Let $U$ be an open set of $\R^d$ and $u\colon U\to \R^m$ an $\Lb^d$-measurable function. The   \textit{jump set} $J_u$ is the set of all points $x\in U$ such that there exists  $(u^+(x),u^-(x),\nu_u(x))\in \R^m\times \R^m\times \mathbb{S}^{d-1}$, with $u^+(x)\neq u^-(x)$, such that
\begin{gather}\nonumber\label{def u+}
       \text{ap}\!\!\!\!\!\!\!\!\!\lim_{\substack{y\to x\\y\in H^{\pm}(x)\cap U}}u(y)=u^\pm(x),
     \end{gather}
     where  $H^{\pm}(x):=\{y\in\Rd:\pm(y-x)\cdot\nu_u(x)>0\}$.
    The triple $(u^+(x),u^-(x),\nu_u(x))$ is uniquely defined up to swapping $u^+(x)$ and $u^-(x)$  and changing the sign of $\nu_u(x)$. Given $x\in J_u$, we set $[u](x):=u^{+}(x)-u^{-}(x)$.
    For $r>0$ we also introduce  $J^r_u$ as the set defined by
    \begin{equation}\label{def Jr}
        J^r_u:=\{x\in J_u: |[u](x)|\geq r\}.
    \end{equation}
\end{definition}

\medskip

\noindent {\textbf{Slicing}.} For every $\xi\in\Rd\setminus\{0\}$, $y\in\Rd$, and $A\subset \Omega$ we define 
\begin{equation*}
    A^{\xi}_y:=\{t\in\R:y+t\xi\in A\}.
\end{equation*}
Given a function $u\colon A\to \Rd$, we define the slice in direction $\xi$ of the $\xi$-component of $u$ as the function $u^\xi_y\colon \R\to \Rd$   defined  by
\begin{equation}\notag \label{slice in Rn}
    u^{\xi}_y(t):=\begin{cases}
        u(y+t\xi)\cdot\xi &\text{if }t\in A^\xi_y,\\
        0&\text{otherwise}.
    \end{cases}
\end{equation} 
$\Pi^\xi$ denotes the hyperplane orthogonal to $\xi$ and passing through 0, that is,  
\begin{equation}\notag 
    \Pi^\xi:=\{y\in\Rd: y\cdot \xi=0\}.
\end{equation}
The projection map onto $\Pi^\xi$ is denoted by $\pi^\xi\colon \Rd\to \Pi^\xi$.
We shall use the following estimate.
\begin{lemma}\label{h0 slices}
Let $\xi\in\R^d\setminus\{0\}$ and let $B\subset\R^d$ be a Borel set. Then the function
$y\mapsto\hzero(B^\xi_y)$
is $\hd$-measurable on $\Pi^\xi$ and
\begin{equation}\label{int hzero Bxiy less}
\int_{\Pi^\xi}\hzero(B^\xi_y)\,{\rm d}\Hd{y}\le \Hd{B}.
\end{equation}
\end{lemma}

\begin{proof}For every $k\in\N$ and $i\in\Z$ we set
$
B_{k,i}:=\{x\in B: i/2^k\le x\cdot\xi <(i+1)/2^k\}
$
and we consider the function $f_k\colon\Pi^\xi\to[0,+\infty]$ defined by
\begin{equation}\label{fk in lemma}
f_k(y):=\vphantom{\int_j}\smash{\sum_{i\in\Z}}\chi_{\pi^\xi(B_{k,i})}(y),
\end{equation}
where $\chi_{\pi^\xi(B_{k,i})}$ is the characteristic function of the projection of $B_{k,i}$ onto $\Pi^\xi$.
We observe that $f_k\le f_{k+1}$ for every $k\in\N$.

It follows from the definition of $B^\xi_y$ that 
\begin{equation}\label{ho lim of fk}
\hzero(B^\xi_y)=\lim_{k\to+\infty}f_k(y)\quad\hbox{for every }y\in\Pi^\xi.
\end{equation}
By the Projection Theorem (see, e.g., \cite[Proposition 8.4.4]{Cohn}) for every $k\in\N$ and $i\in\Z$ the set
$\pi^\xi(B_{k,i})$ is  $\hd$-measurable on $\Pi^\xi$. By \eqref{fk in lemma} and \eqref{ho lim of fk}
this implies that the same is true for the function
$y\mapsto\hzero(B^\xi_y)$. By the Monotone Convergence Theorem we obtain from \eqref{ho lim of fk} that
\begin{equation}\label{int hzero Bxiy equal}
\int_{\Pi^\xi}\hzero(B^\xi_y)\,{\rm d}\Hd{y}=\lim_{k\to+\infty}\int_{\Pi^\xi}f_k(y)\,{\rm d}\Hd{y}.
\end{equation}
By \eqref{fk in lemma} we have
$$
\int_{\Pi^\xi}f_k(y)\,{\rm d}\Hd{y}=\sum_{i\in\Z}\Hd{\pi^\xi(B_{k,i})}\le \sum_{i\in\Z}\Hd{B_{k,i}}=\Hd{B},
$$
which, together with \eqref{int hzero Bxiy equal}, gives \eqref{int hzero Bxiy less}.
\end{proof}
\medskip

\noindent {\textbf{Functions of generalised bounded deformation.}} 
We recall the definition of the space of \textit{functions of generalised  bounded deformation}, introduced in \cite[Definition 4.1]{DalMasoJems}. This definition uses the collection $\mathcal{T}$  of regular truncation functions defined by
\begin{equation*}
\mathcal{T}:=\{\tau\in C^1(\R):-1/2\leq\tau\leq 1/2 \,\text{ and }\,0 \leq\tau'\leq 1\}.
\end{equation*}
\begin{definition}\label{def:GBD}The space $GBD(\Omega)$ is the space of all $\Lb^d$-measurable functions $u\colon \Omega\to\Rd$ for which there exists a measure $\lambda\in \mathcal{M}^+_b(\Omega)$ such that the following equivalent conditions are satisfied for every $\xi\in\Sd$:
\begin{enumerate}
    \item[(a)] for every $\tau\in\mathcal{T}$ we have $D_\xi(\tau(u\cdot\xi))\in \mathcal{M}_b(\Omega)$ and 
    \begin{equation}\label{def a gbd}|D_\xi(\tau(u\cdot\xi))|(B)\leq \lambda(B)\quad \text{for every Borel set } B\subset \Omega;\end{equation}
    \item[(b)] for $\mathcal{H}^{d-1}$-a.e.\ $y\in\Pi^\xi$ we have $u^\xi_y\in BV_{\rm loc}(\Omega^\xi_y)$ and 
    \begin{equation}\label{eq:def GBD}\int_{\Pi^\xi}\big(|Du^\xi_y|(B^\xi_y\setminus J^1_{u^{\xi}_y})+\hzero(B^\xi_y\cap J^{1}_{u^\xi_y})\big){\rm d}\Hd{y}\leq \lambda(B)\,\,\, \text{for every Borel set }B\subset\Omega
.\end{equation}
\end{enumerate}

\begin{remark}\label{remar: finiteness}
    If $u\in GBD(\Omega)$ it follows from (b) of Definition \ref{def:GBD} that for every $\xi\in\Sd$ there exists a Borel set $N_\xi\subset \Pi^\xi$, with $\hd(N_\xi)=0$, such that  $u^\xi_y\in BV_{\rm loc}(\Omega^\xi_y)$ and $|Du^\xi_y|(\Omega^\xi_y)<+\infty$  for every  $y\in\Pi^\xi\setminus N_\xi$. In particular, if $\Omega^\xi_y$ is an interval and $y\in\Pi^\xi\setminus N_\xi$, then $u^\xi_y\in BV(\Omega^\xi_y)$.
    \end{remark}
    \begin{remark}\label{1d gbd}
    The previous remark implies that, if $d=1$, then $GBD(\Omega):=\{u\in BV_{\rm loc}(\Omega): |Du|(\Omega)<+\infty\}$.
\end{remark}
\begin{remark}\label{Lipschitz}
    Condition (a) of Definition \ref{def:GBD} can be strengthened by requiring that   \eqref{def a gbd} holds also for every $\tau\in\mathcal{T}_{\text{Lip}}$, where
 \begin{equation*}
\mathcal{T}_{\rm Lip}:=\{\tau\in \text{Lip}(\R):-1/2\leq\tau\leq 1/2,\,\, 0 \leq\tau'\leq 1,\,\, \tau^\prime \text{ has compact support}\}.
\end{equation*}
Indeed, every function $\tau\in \mathcal{T}_{\rm Lip}$ can be approximated uniformly on $\R$ by  a sequence $(\tau_n)_n\subset\mathcal{T}$, so that $\tau_n(u\cdot\xi)\to \tau(u\cdot\xi)$ in $L^1(\Omega)$. Since for every $U\subset \Omega$ open, the function $v\mapsto |D_\xi v|(U)$ is lower semicontinuous with respect to the $L^1(\Omega)$ convergence, by \eqref{def a gbd} we obtain that $D_\xi(\tau(u\cdot\xi))\in \mathcal{M}_b(\Omega)$ and that
\begin{equation*}
    |D_\xi(\tau(u\cdot\xi))|(U)\leq \liminf_{n\to+\infty} |D_\xi(\tau_n(u\cdot\xi))|(U)\leq \lambda (U)
\end{equation*}
for every $U\subset \Omega$ open. Inequality \eqref{def a gbd} for a general Borel set $B\subset \Omega$ follows by approximation with open sets.
\end{remark}

\begin{remark}\label{scaling remark} Inequality \eqref{eq:def GBD} can be extended to non-unitary vectors. 
Elementary arguments show that for every $t\neq0$, $\xi\in\Rd\setminus\{0\}$, $y\in\Pi^\xi$, and $A\subset \R^d$ 
 \begin{equation}
        \label{eq:rescaled intervals}A^\xi_y=tA^{t\xi}_y.
    \end{equation}
Since $ u_y^{t\xi}(s)=t u_y^{\xi}(st)$ for every $s\in\R$, we have 
\begin{gather}\notag\label{DhatuS}
D u_y^{t\xi}(S)=|t|D u_y^{\xi}(tS)\quad\hbox{and}\quad  |D u_y^{t\xi}|(S)=|t||D u_y^{\xi}|(tS),
\\
t J_{ u_y^{t\xi}}=J_{ u_y^\xi} \quad\hbox{and}\quad t J_{ u_y^{t\xi}}^{|t|}=J_{ u_y^\xi}^1
\label{Jhatutxi}
\end{gather}
for every Borel set $S\subset\Omega_y^{t\xi}$.
In particular, if $B\subset\Omega$ is a Borel set, taking $S=B_y^{t\xi}\setminus J^{|t|}_{u^{t\xi}_y}$ and using \eqref{eq:rescaled intervals}  and \eqref{Jhatutxi} we get
\begin{equation}\label{DhatuAytxi}
  |D u_y^{t\xi}|(B_y^{t\xi}\setminus J^{|t|}_{u^{t\xi}_y})=|t||D u_y^{\xi}|(B_y^\xi\setminus J^1_{u^\xi_y}),
\end{equation}
while, taking $S=B^{t\xi}_y$ and using \eqref{eq:rescaled intervals}
\begin{equation}\label{2.13 bis}
    D u_y^{t\xi}(B_y^{t\xi})=|t|D u_y^{\xi}(B_y^\xi),
\end{equation}
which will be used later in the proof of Proposition \ref{prop: 2 homogeneity}.
Applying \eqref{Jhatutxi} and \eqref{DhatuAytxi} with $t=|\xi|$ and $\xi$ replaced by $\xi/|\xi|$, from \eqref{eq:def GBD} we obtain
    \begin{gather}\label{def:GBD xiquadro}
\int_{\Pi^\xi}|\xi||Du^\xi_y|(B^\xi_y\setminus J^{|\xi|}_{u^\xi_y}){\rm d}\Hd{y}\leq |\xi|^2\lambda(B),\\\notag
\label{disuguaglianza sul saltino}
\int_{\Pi^\xi}\hzero(B^\xi_y\cap J^{|\xi|}_{u^\xi_y}){\rm d}\Hd{y}\leq \lambda (B),
\end{gather}
    for every $\xi\in\Rd\setminus \{0\}$ and every Borel set $B\subset\Omega$. 
\end{remark}

 \begin{remark}
\label{re:RemarkAlmiTasso}  It follows from \cite[Remark 3.6]{AlmiTassoComp} that, if $u\in GBD(\Omega)$, then $\Hd{J^1_u}<+\infty$. More precisely, it is  shown that, if $\lambda\in \mathcal{M}^+_b(\Omega)$ satisfies \eqref{eq:def GBD}, then for every Borel set $B\subset \Omega$  we have the inequality
\begin{equation}\label{eq:bound on J1}
    \Hd{J^1_u\cap B}\leq 4d\lambda(B).
\end{equation}
 Since by Theorem \ref{fine propr of GBD} below for $\hd$-a.e.\ $y\in \Pi^\xi$ we have
\begin{equation*}J^{|\xi|}_{u^\xi_y}\subset (J^1_u)^\xi_y,\end{equation*}
it follows from \eqref{int hzero Bxiy less}, \eqref{def:GBD xiquadro}, and   \eqref{eq:bound on J1} 
that 
\begin{gather}\label{def:GBD con J1}
\int_{\Pi^\xi}|\xi||Du^\xi_y|((B\setminus J^1_u)^\xi_y){\rm d}\Hd{y}\leq |\xi|^2\lambda(B),\\\notag
\int_{\Pi^\xi}\hzero((B\cap J^1_u)^\xi_y){\rm d}\Hd{y}\leq 4d\lambda(B).
    \end{gather}
    Given $0<r\leq 1$, we can consider  the function $v=u/r$, which by \cite[Remark 4.6]{DalMasoJems} satisfies \eqref{eq:def GBD} with $\lambda/r$. Applying \eqref{eq:bound on J1} to this function, we obtain 
    \begin{equation}\label{bound on Jr}
         \Hd{J^{r}_u\cap B}\leq \frac{4d}{r}\lambda(B)
    \end{equation}
    for every Borel set $B\subset\Omega$. In particular, this implies that $[u]\in L^1_{\rm weak}(J_u,\hd)$. Moreover, since $J_u=\bigcup_{0< r\leq 1}J^{r}_u$, we have that $J_u$ is $\sigma$-finite with respect to $\hd$.
\end{remark}

 \begin{definition}\label{def:lambda xi}
     For every $u\in GBD(\Omega)$ and $\xi\in\Rd\setminus \{0\}$ we introduce the bounded Radon measure $\lambda^\xi_u$ defined for every $\xi\in\Rd\setminus\{0\}$ and every  Borel set $B\subset\Omega$ by
 \begin{equation}
\label{eq:defmuxi} \lambda^\xi_u(B):=
\displaystyle \int_{\Pi^\xi}\big(|\xi||Du^\xi_y|(B^\xi_y\setminus J^{|\xi|}_{u^\xi_y})+\hzero(B^\xi_y\cap J^{|\xi|}_{u^\xi_y})\big) {\rm d}\Hd{y}.
 \end{equation}
 \end{definition}

 Given $u\in GBD(\Omega)$, let $\lambda_u$ be the smallest measure $\lambda$ for which (a) and (b) of Definition \ref{def:GBD} hold true.
 It can be shown (see \cite[Proposition 4.17]{DalMasoJems}) that 
 \begin{equation}\label{def lambda u}
     \lambda_u(B)= \sup\sum_{i=1}^k\lambda^{\xi_i}_u(B_i)\quad \text{ for every Borel set  }B\subset \Omega,
 \end{equation}
 where the  supremum is taken over all $k\in\N$, all families $(\xi_i)_{i=1}^k $ of vectors of $\Sd$, and all Borel partitions $(B_i)_{i=1}^k$ of $B$.
\end{definition}

The following theorem collects some of the fine properties of $GBD(\Omega)$ functions, proved in  \cite[Proposition 6.1 and Theorems 6.2, 8.1, 9.1]{DalMasoJems}.
\begin{theorem}\label{fine propr of GBD}Let $u\in GBD(\Omega)$. Then the  following properties hold:

{\rm (a)} existence of the approximate symmetric  gradient: there exists $\mathcal{E}u\in L^1(\Omega;\Rdsym)$ such that for $\Lb^d$-a.e.\ $x\in\Omega$ we have
\begin{equation*}
    \underset{y\to x}{{\rm aplim}}\frac{(u(y)-u(x)-\mathcal{E}u(x)(y-x))\cdot(y-x)}{|x-y|^2}=0;
\end{equation*}
moreover, for every $\xi\in\Rd\setminus \{0\}$ and for $\hd$-a.e.\ $y\in\Pi^\xi$ we have 
\begin{equation*}
    \E u(y+t\xi)\xi\cdot\xi=\nabla u^{\xi}_y(t)
\end{equation*}
for $\Lb^1$-a.e.\ $t\in \Omega^\xi_y$, where $\nabla u^\xi_y$ denotes the density of the absolutely continuous part of $Du^\xi_y$ with respect to $\Lb^1$;

{\rm (b)} $J_u$ and its slices: the jump set $J_u$ is $(\hd,d-1)$-rectifiable; 
 for every $\xi\in\Rd\setminus \{0\}$ and for $\mathcal{H}^{d-1}$-a.e.\ $y\in\Pi^\xi$ we have 
\begin{gather}\label{eq:jump set to sliced jump set}
J_{u^\xi_y}=\{x\in J_u:[u](x)\cdot\xi\neq 0\}^\xi_y\subset (J_u)^\xi_y,\\ \label{eq:jump to sliced jump}
   (u^{\xi}_y)^{\pm}(t)=u^{\pm}(y+t\xi)\cdot\xi\quad \text{for every }t\in (J_u)^{\xi}_y, \\
   \label{Jxihatu}
   J^{|\xi|}_{u^\xi_y}\subset (J^1_u)^\xi_y,
\end{gather}
    where the normals at $J_u$ and $J_{u^\xi_y}$ are oriented in such a way that $\nu_u\cdot\xi\geq 0$ and $\nu_{u^{\xi}_y}=1$;
   moreover, setting 
    \begin{equation}\label{theta u}
        \Theta_u:=\Big\{x\in\Omega:\limsup_{\rho\to 0}\frac{\lambda_u(B_\rho(x))}{\rho^{d-1}}>0\Big\},
    \end{equation}
    we have that $\Theta_u$ is $(\hd,d-1)$-rectifiable, $J_u\subset \Theta_u$, and that $\hd(\Theta_u\setminus J_u)=0$.
\end{theorem}

It is easy to see that for a function $u\in GBD(\Omega)\setminus BD(\Omega)$ its jump $[u]$ may be not  integrable on $J_u$ with respect to $\hd$ (see, for instance, \cite[Example 12.3]{DalMasoJems}). More precisely, one can  show (see, for instance,  \cite[Proposition 3.2]{AmbrCosciaDalM} or \cite[Remark 2.17]{TassoContinuity}) that, if $u\in GBD(\Omega)\cap L^1(\Omega;\Rd)$ and  $[u]\in L^1(J_u,\mathcal{H}^{d-1})$, then $u$ belongs to $BD(\Omega)$.  
Nonetheless, we now show that $[u]$ is integrable on $J_u\setminus J^1_u$ with respect to $\hd$. This is the content of the following proposition. 

\begin{proposition}\label{thm:L1smalljumps}
Let $u\in GBD(\Omega)$. Then 
\begin{equation}\label{boundedness jumps}
    \int_{J_u\setminus J^1_u}|([u]\odot\nu_u)\xi\cdot\xi|\,{\rm d}\hd<+\infty
\end{equation}
for every $\xi\in\Rd$. This is equivalent to
$[u]\in L^1(J_u\setminus J^1_u,\mathcal{H}^{d-1})$.
\end{proposition}
\begin{proof} By homogeneity it is enough to show \eqref{boundedness jumps} for a fixed $\xi\in\Sd$. By definition of $\odot$ we have $([u]\odot\nu_u)\xi\cdot\xi=([u]\cdot \xi)(\nu\cdot\xi)$. Therefore, by the Area Formula \cite[12.4]{Simon1984}, we can write 
\begin{equation}\label{inequality new 2.11}
\begin{gathered}
    \int_{J_u\setminus J^1_u}|([u]\odot\nu_u)\xi\cdot\xi|\,{\rm d}\hd=\int_{J_u\setminus J^1_u}|([u]\cdot \xi)(\nu_u\cdot\xi)|\,{\rm d}\hd\\=\int_{\Pi^\xi}\Big(\int_{(J_u\setminus J^1_u)^\xi_y}|[u](y+t\xi)\cdot \xi|\,{\rm d}\hzero(t)\Big){\rm d}\hd(y)\\\leq \int_{\Pi^\xi}\Big(\int_{J_{u^\xi_y}\setminus J^1_{u^\xi_y}}|[u^\xi_y](t)|\,{\rm d}\hzero(t)\Big){\rm d}\hd(y),
\end{gathered}
\end{equation}
where in the inequality we have used \eqref{eq:jump set to sliced jump set}-\eqref{Jxihatu}. By (b) of Definition \ref{def:GBD}, for $\hd$-a.e.\ $y\in\Pi^\xi$  we have $u^\xi_y\in BV_{\rm loc}(\Omega^\xi_y)$, so that 
\begin{gather*}
   \int_{\Pi^\xi} \Big(\int_{J_{u^\xi_y}\setminus J^1_{u^\xi_y}}|[u^\xi_y](t)|\,{\rm d}\hzero(t)\Big){\rm d}\hd(y)=\int_{\Pi^\xi}|Du^\xi_y|(J_{u^\xi_y}\setminus J^1_{u^\xi_y}){\rm d}\hd(y)\\\leq \int_{\Pi^\xi} |Du^\xi_y|(\Omega^\xi_y\setminus J^1_{u^\xi_y}){\rm d}\hd(y)\leq \lambda (\Omega)<+\infty,
\end{gather*}
which, together with \eqref{inequality new 2.11}, concludes the proof of \eqref{boundedness jumps}.

From the polarisation identity  it follows that for every $\xi,\eta\in\Rd\setminus \{0\}$  one has
    \begin{gather*}
        |([u]\odot \nu_u)\xi\cdot\eta|\leq \frac{1}{4}\big( |([u]\odot \nu_u)(\xi+\eta)\cdot(\xi+\eta)|
       +|([u]\odot \nu_u)(\xi-\eta)\cdot(\xi-\eta)| \big)
    \end{gather*}
    on $J_u$.
    Hence, by \eqref{boundedness jumps} we deduce that $[u]\odot\nu_u\in L^1(J_u\setminus J^1_u, \mathcal{H}^{d-1})$
    Since for every $a,b\in\Rd$ we have  $|a||b|\leq \sqrt{2} |a\odot b|$,
    the proof is concluded.
\end{proof}

The space $GBD(\Omega)$ behaves nicely with respect to restriction to affine subspaces of the domain $\Omega$. This fact is made rigorous by the following result.
\begin{theorem}\label{theorem restriction} Assume $d\geq 2$. Let $u\in GBD(\Omega)$, let $V$ be a vector subspace of $\Rd$ of dimension $k$, with $1\leq k\leq d-1$, let $V^\perp$ be its orthogonal subspace, and let  $\pi_V\colon \Rd\to V $ be the orthogonal projection onto $V$. For every $z\in V^\perp$ and  $E\subset \Omega$ we set $E^V_z:=\{x\in V:z+x\in E\}=V\cap (E-z)$  and consider the function $u^V_z\colon\Omega^V_z\to V$  defined by $u^V_z(x):=\pi_V(u(z+x))$. Then the following properties hold: 
\begin{enumerate}
    \item[(a)] for $\mathcal{H}^{d-k}$-a.e.\ $z\in V^\perp$ we have $u^V_z\in GBD(\Omega^V_z)$;
    \item [(b)]for $\mathcal{H}^{d-k}$-a.e.\ $z\in V^\perp$ we have $J_{u^V_z}\subset (J_u)^V_z\cup N_z$ for a Borel set $N_z\subset V$ with $\mathcal{H}^{k-1}(N_z)=0$. 
\end{enumerate}
\end{theorem}
To prove property (b)  we need the following result concerning the relation  between the jump points of a function $u\in GBD(\Omega)$ and the jump points of its restriction to a hyperplane that does not intersect  the set $S_u$ of approximate discontinuity points.
\begin{proposition}\label{prop: jumps restricted}
   Assume $d\geq 2$. Let $u\in GBD(\Omega)$, $x_0\in\Omega$, and $\xi\in\Sd$. Assume that
    \begin{equation}\label{equation approximate continuity}
        \mathcal{H}^{d-1}(S_u\cap (x_0+\Pi^\xi))=0.
    \end{equation}
    Let $v\colon (\Omega-x_0)\cap \Pi^\xi \to \Pi^\xi$ be the function defined by $v(y):=\pi^\xi(\widetilde{u}(x_0+y))$ for $\mathcal{H}^{d-1}$-a.e.\ $y\in (\Omega-x_0)\cap\Pi^\xi$. Assume that there exist a direction $\nu\in\Sd\cap \Pi^\xi$ and two vectors $b^\pm\in \Pi^\xi$, with $b^+\neq b^-$, such that for every $\e>0$ we have
    \begin{equation}\label{def jump plane}
        \limsup_{\rho\to 0^+}\frac{\hd(\{y\in B_\rho(0)\cap\Pi^\xi:\pm y\cdot\nu>0,\,|v(y)-b^{\pm}|>\e\})}{\rho^{d-1}}=0.
    \end{equation}
 Then $x_0\in\Theta_u$.
\end{proposition}
\begin{proof}
   See \cite[Theorem 7.1]{DalMasoJems}.
\end{proof}
\begin{proof}[Proof of Theorem \ref{theorem restriction}]
    The proof of (a) can be found in \cite[Theorem 4.19]{DalMasoJems}. 
    
   We divide the proof of (b) into two steps.
   
   \noindent {\it Step 1.}  
 Assume that $k=d-1$ and let  $\xi\in\Sd$ be such that $V=\Pi^\xi$. We claim that for $\Lb^1$-a.e.\ $s\in\R$ there exists a Borel set $N_{s}\subset V$, with $\mathcal{H}^{d-2}(N_{s})=0$, such that 
 \begin{equation}\notag
     J_{u^V_{s\xi}}\subset (J_u)^V_{s\xi}+N_{s}.
 \end{equation}
 To prove this property, we observe that by the Fubini Theorem  the equality $\Lb^d(S_u)=0$ implies that   
    for $\Lb^1$-a.e.\ $s\in\R$   and for every $y_0\in \Pi^\xi=V$ condition \eqref{equation approximate continuity} holds with $x_0=s\xi+y_0$   , while the equality $\widetilde{u}=u$ $\Ld$-a.e.\ in $\Omega$ implies that for $\Lb^1$-a.e.\ $s\in\R$ we have \begin{equation}\label{new prop restriction}
   \widetilde{u}(s\xi+y)=u(s\xi+y) \quad \text{for $\hd$-a.e.\ $y\in \Pi^\xi$}.\end{equation}  Let us fix $s\in\R$ with these properties. Given  $y_0\in J_{u^V_{s\xi}}$,  we consider the function $v(y):=\pi^\xi(\widetilde{u}(s\xi+y_0+y))$ for $y\in\Pi^\xi$ and observe that by \eqref{new prop restriction} we have
   \begin{equation}\notag
       v(y)=u^V_{s\xi}(y_0+y) \quad \text{ for $\hd$-a.e.\ $y\in\Pi^\xi$}.
   \end{equation}
    Since $y_0\in J_{u^V_{s\xi}}$, $v$ satisfies \eqref{def jump plane}.  We can then apply Proposition \ref{prop: jumps restricted} to obtain that $s\xi+y_0\in \Theta_u$, which gives $y_0\in (\Theta_u)^V_{s\xi}$. Setting $N_{s}:=(\Theta_u\setminus J_u)^V_{s\xi}$, we have $y_0\in (J_u)^V_{s\xi}\cup N_s$, hence    $J_{u^V_{s\xi}}\subset (J_u)^V_{s\xi}\cup N_{s}$.  Since by Theorem \ref{fine propr of GBD}  we have $\Hd{\Theta_u\setminus J_u}=0$, it follows from \cite[Theorem 2.10.25]{Federer}  that $\mathcal{H}^{d-2}(N_{s})=0$ for $\Lb^1$-a.e.\ $s\in\R$. This concludes the proof of the claim. Of course, if $d=2$ this Step also completes the proof of (b).

\medskip

\noindent {\it Step 2.} Assume $d>2$. We now prove (b) by induction on the codimension of $V$.  Step 1 gives (b) when  the dimension of  $V$ is $d-1$. Given $1\leq k\leq d-2$, we  assume now that (b) holds for every subspace of dimension $k+1$ and we want to prove that it holds in  dimension $k$. Let us fix a subspace $V$ of  dimension $k$ and $\xi\in V^\perp\cap \Sd$. We consider the vector space $\widetilde{V}$ generated by $V$  and $\xi$. By (a) for $\mathcal{H}^{d-k-1}$-a.e.\ $z\in \widetilde{V}^\perp$ we have that  $u^{\widetilde{V}}_z\in GBD(\Omega^{\widetilde{V}}_z)$. 
    Using the inductive hypothesis, we deduce that for $\mathcal{H}^{d-k-1}$-a.e.\ $z\in \widetilde{V}^\perp$ there exists a Borel set $\widetilde{N}_z\subset \widetilde{V}$, with $\mathcal{H}^{k}(\widetilde{N}_{z})$=0, such that 
    \begin{equation}\label{jumps contained proof}
    J_{u^{\widetilde{V}}_{z}}\subset (J_u)^{\widetilde{V}}_z\cup \widetilde{N}_z.
    \end{equation}
    Let us fix $z\in \widetilde{V}^\perp$ satisfying both these properties. We observe that  for every  $s\in\R$ and $E\subset \Omega$ we have  \begin{equation}\label{equal sets new}
        E^V_{z+s\xi}=(E^{\widetilde{V}}_{z})^V_{s\xi}\,\,\,\text{ and }\,\,\, u^{V}_{z+s\xi}=(u^{\widetilde{V}}_z)_{s\xi}^V \,\,\text{ on $\in \Omega^V_{z+s\xi}$.}
    \end{equation}   
    
    Applying Step 1 with $\Rd$ replaced by $\widetilde{V}$ and with $u$ replaced by $u^{\widetilde{V}}_z$, we have that for $\Lb^1$-a.e.\ $s\in \R$ there exists a Borel set $N_{s}\subset V$, with $\mathcal{H}^{k-1}(N_s)=0$, such that
    \begin{equation}\notag  J_{(u_{z}^{\widetilde{V}})^V_{s\xi}}\subset (J_{u^{\widetilde{V}}_{z}})^V_{s\xi}\cup N_{s}.
    \end{equation} By 
    \eqref{jumps contained proof} this implies that
    \begin{equation}\notag J_{(u_{z}^{\widetilde{V}})^V_{s\xi}}\subset((J_u)^{\widetilde{V}}_z)^V_{s\xi}\cup (\widetilde{N}_{z})^V_{s\xi}\cup N_{s}.
    \end{equation}
    By \eqref{equal sets new} this gives 

    \begin{equation}\label{jump contained proof 5} J_{u^V_{z+s\xi}}\subset (J_u )^V_{z+s\xi}\cup (\widetilde{N}_{z})^V_{s\xi}\cup N_{s}.
    \end{equation}
    By \cite[Theorem 2.10.25]{Federer} we deduce that for $\Lb^1$-a.e.\ $s\in\R$ we have $\mathcal{H}^{k-1}((\widetilde{N}_{z})^V_{s\xi})=0$. Setting $N_{z+s\xi}:=N_{s}\cup (\widetilde{N}_z)^V_{s\xi}$, it follows that   $\mathcal{H}^{k-1}(N_{z+s\xi})=0$ for $\Lb^1$-a.e.\ $s\in\R$. Finally, from \eqref{jump contained proof 5} we deduce that 
    \begin{equation}\label{final projection}
        J_{u_{z+s\xi}^V}\subset (J_u)^V_{z+s\xi}\cup N_{z+s\xi}. 
    \end{equation}
    Since $V^\perp$ is the space generated by $\widetilde{V}^\perp$ and $\xi$ and  \eqref{final projection} holds for $\Lb^1$-a.e.\ $s\in\R$ and for $\mathcal{H}^{d-k-1}$-a.e.\ $z\in \widetilde{V}^\perp$, property (b) for $V$ follows from the Fubini Theorem. This concludes the proof of the inductive step and hence of the theorem.
\end{proof}

\section{Approximations by Riemann sums}\label{section Riemann}
In the proof of the main result of this paper, we will approximate various integrals with well-chosen Riemann sums, using a suitable version of a result that goes back to Hahn (see \cite{Hahn}). Similar results are proved also in \cite[Lemma A:1]{Santambrogio}. For technical reasons we use a construction described in \cite[Page 63]{Doob} and further developed in \cite[Lemma 4.12]{DalFraToa}. Since this result is crucial for our arguments, we give here the precise statement and a detailed proof. 

       \begin{lemma}\label{lemma:Riemann sums 1}
   Let $I=[a,b]\subset\R$  be a bounded closed interval. For every $z\in I$, $k\in\N$, and $i\in\Z$ let $t^k_i:=z+i/k$.   Let $(X,\|\cdot\|)$ be a Banach space and let $f\colon \R\to X$ be a Bochner integrable function such that $f=0$ on $\R\setminus I$. Then there exist an infinite subset $K\subset \N$ and an $\Lb^1$-negligible set $N\subset I$  such that  
    \begin{gather}
         \label{eq: prima lemma discre}\underset{k\in K}{\lim_{k\to +\infty}}\sum_{i\in \Z}\int_{t_{i}^{k}}^{t_{i+1}^{k}}\|f(t^{k}_i)-f(t)\|\, {\rm d}t=0,
    \end{gather}
    for every $z\in I\setminus N$. Moreover, given $h\in\N$ and setting $\mathcal{I}^k_h:=\{i\in\Z: [t^k_{i-h}, t^k_{i+h}]\subset I\}$ and
    $\mathcal{F}^k_h:=\{i\in\Z:t^k_i\in I\}\setminus  \mathcal{I}^{k}_h$, for every $\e>0$ there exists a Borel set  $I_\e\subset I$, with $\Lb^1(I\setminus I_\e)\leq \e$, such that 
    \begin{gather} \label{first uniform}
      \underset{k\in K}{\lim_{k\to +\infty}} \sum_{i\in \Z}\int_{t_{i}^{k}}^{t_{i+1}^{k}}\|f(t^{k}_i)-f(t)\|\, {\rm d}t=0 \quad \text{ uniformly  } \text{ for } z\in I_\e,\\  \label{second uniform}
       \underset{k\in K}{\lim_{k\to +\infty}} \frac{1}{{k}}\sum_{i\in \mathcal{I}^k_h}f(t_i^{k})=\int_{I}f(t)\, {\rm d}t \quad \text{ uniformly  } \text{ for } z\in I_\e,\\\label{new negligible}
       \underset{k\in K}{\lim_{k\to +\infty}} \frac{1}{{k}}\sum_{i\in\mathcal{F}^k_h}\!\|f(t_i^{k})\|\!=0\quad \text{ uniformly  } \text{ for } z\in I_\e.
    \end{gather}
    \end{lemma}
     \begin{proof}
         We follow closely the proof of \cite[Lemma 4.12]{DalFraToa}.  Given $k\in\N$, we consider the set 
        \begin{equation}\label{lemma banach1}
            \mathcal{J}^k:=\{i\in\Z:k(a-b)-1\leq i\leq k(b-a)\}.
         \end{equation} 
        Note that for every $z\in I$ we have 
        \begin{equation}\nonumber \label{banach 2}
        [t_i^k,t_{i+1}^k]\cap I\neq\emptyset \quad\Longrightarrow \quad i\in \mathcal{J}^k.
        \end{equation}
Integrating with respect to $z$ the sum on the left-hand side of \eqref{eq: prima lemma discre} and using Fubini's theorem we obtain
        \begin{align} 
           \notag \int_I\Big(\sum_{i\in \Z}\int_{t^{k}_{i}}^{t^k_{i+1}}\|f(t^{k}_{i})-&f(t)\|\, {\rm d} t\Big){\rm d}z = \int_I\Big(\sum_{i\in \mathcal J^k}\int_{t^{k}_{i}}^{t^k_{i+1}}\|f(t^{k}_{i})-f(t)\|\, {\rm d} t\Big){\rm d}z
            \\\notag
            &= \int_I\Big(\sum_{i\in\mathcal J^k}\int_{0}^{\frac{1}{k}}\|f(z+\tfrac{i}{k})-f(z+\tfrac{i}{k}+s)\|\, {\rm d}s\Big){\rm d} z
            \\\notag
            &\leq \sum_{i\in \mathcal{J}^{k}}\int_{0}^{\frac{1}{k}}\Big(\int_{-\infty}^{+\infty}\|f(z+\tfrac{i}{k})-f(z+\tfrac{i}{k}+s)\|\, {\rm d}z\Big){\rm d}s\\\label{eq:fubini discretisation lemma}
            & =\sum_{i\in \mathcal{J}^{k}}\int_{0}^{\frac{1}{k}}\Big(\int_{-\infty}^{+\infty}\|f(z)-f(z+s)\|\,{\rm d}z\Big){\rm d}s.
        \end{align}
        By the $L^1$-continuity of translations  for every $\e>0$ there exists $\delta>0$ such that
        \begin{align*}&\displaystyle 
            \int_{-\infty}^{+\infty}\|f(z)-f(z+s)\|\, {\rm d}z<\e
        \end{align*}
        whenever $0\leq s<\delta $. Hence,
        \begin{equation*}
            \int_{0}^{\frac{1}{k}}\Big(\int_{-\infty}^{+\infty}\|f(z)-f(z+s)\|\,{\rm d}z\Big){\rm d}s< \frac{\e}{k}
        \end{equation*}
        for $k>1/\delta$.
     Observing that the number of elements of $\mathcal{J}^k$ satisfies $\hzero(\mathcal{J}^{k})\leq 2k(b-a)+2$, from the previous inequality and \eqref{eq:fubini discretisation lemma} we deduce that 
        \begin{equation}\nonumber 
\label{finale riemann 1}\lim_{k\to+\infty}\int_I\Big(\sum_{i\in \Z}\int_{t^k_{i}}^{t^k_{i+1}}\|f(t^{k}_i)-f(t)\|\, {\rm d}t\Big){\rm d}z=0.
        \end{equation}
        Hence, there exists an infinite set $K\subset \N$  and an $\Lb^1$-negligible set $N\subset I$  such that  \eqref{eq: prima lemma discre} holds. 

     To prove the second part of the statement, we first observe that 
        by Egorov's theorem there exists a Borel set $I_\e\subset I$, with $\Lb^1(I\setminus I_\e)\leq \e$, such that \eqref{first uniform} holds.
        
        To prove  \eqref{second uniform} and \eqref{new negligible}, we first observe that, since the number of elements of
        $$
       \widehat{\mathcal F}^k_h:=\{i\in\Z:[t^k_i,t^k_{i+1}]\cap I\neq\emptyset\}\setminus  \mathcal{I}^{k}_h
        $$
        is less than $2h+2$, by the absolute continuity of the integral 
        we have
        \begin{equation}\label{15-04-a}
        \lim_{k\to+\infty}\sum_{i\in \widehat{\mathcal{F}}^{k}_h}\int_{t^k_{i}}^{t^k_{i+1}}\|f(t)\|\, {\rm d}t=0  \quad \text{ uniformly  } \text{ for } z\in I_\e.
\end{equation}
Since
\begin{equation}\notag
    \frac{1}{k}\sum_{i\in\mathcal{F}^k_h}\!\|f(t_i^{k})\|\!\leq  \sum_{i\in\Z}\int_{t_i^{k}}^{t_{i+1}^{k}}\!\|f(t_i^{k})-f(t)\|\,{\rm d}t + \sum_{i\in \mathcal{F}^{k}_h}\int_{t^k_{i}}^{t^k_{i+1}}\|f(t)\|\, {\rm d}t,
\end{equation}
\eqref{new negligible} follows from \eqref{first uniform}, \eqref{15-04-a}, and from the inclusion $\mathcal{F}^{k}_h\subset \widehat{\mathcal{F}}^{k}_h$. Similarly, the inequality
$$
\Big\| \frac{1}{{k}}\sum_{i\in \mathcal{I}^k_h}f(t_i^{k})-\int_{I}f(t)\, {\rm d}t\Big\|\le
         \sum_{i\in \mathcal{I}^k_h}\int_{t_{i}^{k}}^{t_{i+1}^{k}}\|f(t^{k}_i)-f(t)\|\, {\rm d}t 
         + 
          \sum_{i\in \widehat{\mathcal{F}}^{k}_h}\int_{t^k_{i}}^{t^k_{i+1}}\|f(t)\|\, {\rm d}t,
$$
together with
         \eqref{first uniform} and \eqref{15-04-a}, gives  \eqref{second uniform}.
       \end{proof}
       
       In the next technical result we will consider the Riemann sums associated with functions converging to 0 in $L^1$ and depending on an additional parameter. This result will be crucial in Section \ref{section conclusion}.
      
  \begin{lemma}\label{lemma: Riemann 2}
            Let $I:=[a,b]$ and $J:=[c,d]$ be  bounded closed intervals. For every $z_1\in I$, $k\in\N$, and $i\in\Z$ let $t^k_i:=z_1+\tfrac{i}{k}$.  Let $(X,\|\cdot\|)$ be a Banach space and  let $f_k\colon\R\times \R\to X$ be a sequence of Bochner integrable functions, with $f_k=0$ on $(\R\times \R)\setminus(I\times J) $, such that 
        \begin{equation}\label{pointwise convergence}
                \lim_{k\to+\infty} f_{k}(t,z_2)=0\quad \text{ for $\Lb^1$-a.e.\ $t\in I$ and $\Lb^1$ -a.e.\ $z_2\in J$}.
            \end{equation}
              Assume also that there exists an integrable function $g\colon\R\to[0,+\infty)$ such that 
              \begin{equation}\label{uniform bound lemma}
        \|f_k(t,z_2)\|\leq g(t) \quad \text{for $\Lb^1$-a.e.\ $t\in I$,   $\Lb^1\text{-a.e.}$ $z_2\in J$, and  \text{every } $k\in\N$}.
            \end{equation}
            Then there exists an $\Lb^2$-negligible set $N\subset I\times J$ and a infinite set $K\subset \N$ such that for every $(z_1,z_2)\in (I\times J)\setminus N$ we have
            \begin{equation}\label{clail lemma somme variabili genrico}
 \underset{k\in K}{\lim_{k\to +\infty}}\frac{1}{k}\sum_{i\in\Z}\|f_k(t^k_i,z_2)\|=0.
            \end{equation}
            Moreover, for every $\e>0$ there exists a Borel set $A_\e\subset I\times J$, with $\Lb^2((I\times J)\setminus A_\e)\leq \e$, such that 
                   \begin{equation}\label{Egorovv variabile}
            \underset{k\in K}{\lim_{k\to +\infty}} \frac{1}{k}\sum_{i\in\Z}\|f_k(t^k_i,z_2)\|=0\text{ uniformly} \text{ for } (z_1,z_2)\in A_\e.
        \end{equation}
        \end{lemma}
        \begin{proof}
        For $k\in\N$ let $\mathcal{J}^k$ be given by \eqref{lemma banach1}. Integrating the sum on the left-hand side of \eqref{clail lemma somme variabili genrico}  with respect to $z=(z_1,z_2)\in I\times J$  
 \begin{align}
            \displaystyle\notag &\int_{J}\Big(\int_I\Big(\frac{1}{k}\sum_{i\in \Z}\|f_k(t^k_i,z_2)\|\,\Big){\rm d}z_1\Big){\rm dz}_2=\int_{J}\Big(\int_{I}\Big(\frac{1}{k}\sum_{i\in \Z}\|f_k(z_1+\tfrac{i}{k},z_2)\|\,\Big){\rm d}z_1\Big){\rm d}z_2
            \\
            &\displaystyle\hspace{-0.5 cm}\leq \frac{1}{k}\sum_{i\in {\mathcal{J}^k}}\int_\R\Big(\int_{\R}\|f_k(z_1+\tfrac{i}{k},z_2)\|\, {\rm d}z_1\Big){\rm d}z_2\label{eq:fubini discretisation lemma 2}
           =\frac{\hzero(\mathcal{J}^k)}{k}\int_{\R}\Big(\int_\R\|f_k(z_1,z_2)\|\,{\rm d}z_1\Big){\rm d}z_2,
        \end{align}
        where in the inequality we have used the hypothesis on the supports of $f_k$.
       By \eqref{pointwise convergence}  for $\Lb^2$-a.e.\ every $z \in I\times J$ the sequence $\|f_k(z_1,z_2)\|$ converges to $0$ as $k\to+\infty$. Thanks to \eqref{uniform bound lemma},  by the Dominated Convergence Theorem we have $
\lim_{k}\|f_k\|_{L^1(\R^2;X)}=0.$
       This fact, together with \eqref{eq:fubini discretisation lemma 2} and the boundedness of  $\hzero({\mathcal{J}^k})/k$, implies  that 
\begin{equation}\notag\label{salvatrice}
     \lim_{k\to+\infty}\int_{I\times J}\Big(\frac{1}{k}\sum_{i\in \Z}\|f_k(t^k_i,z_2)\|\,\Big){\rm d}z=0,
       \end{equation}
which  gives \eqref{clail lemma somme variabili genrico}. By Egorov's Theorem we obtain also \eqref{Egorovv variabile}, concluding the proof.
        \end{proof}

   \begin{remark}\label{remark sequences}
 Suppose that  for every $n\in\N$ we have a function $(f^n)$ that satisfies the hypotheses of Lemma \ref{lemma:Riemann sums 1}. Arguing as in \cite[Remark 4.13]{DalFraToa}, one can find sets $K$ and $I_\e$ as in the statement of Lemma \ref{lemma: Riemann 2} for which \eqref{first uniform}-\eqref{new negligible} and \eqref{Egorovv variabile} hold with $f=f^n$ for every $n\in\N$.
 
  Similarly, if  for every $n\in\N$ we have a sequence  of functions $(f^n_k)_k$ that satisfies the hypotheses of Lemma \ref{lemma: Riemann 2}, the same argument shows that one can find sets $K$ and $U_\e$ as in the statement of Lemma \ref{lemma: Riemann 2} for which \eqref{clail lemma somme variabili genrico} and \eqref{Egorovv variabile} hold with $f_k=f^n_k$ for every $n\in\N$.
\end{remark}

\section{An auxiliary family of measures}\label{section auxiliary}
Given $u\in GBD(\Omega)$, we associate to it a family of measures, closely related to the measures $\lambda^\xi_u$ introduced in \eqref{eq:defmuxi}. For every $\xi\in\Rd\setminus \{0\}$ and every Borel set  $B\subset \Omega$ we set
\begin{equation}\label{def sigma}
\sigma_u^\xi(B):=|\xi|\int_{\Pi^\xi}Du^\xi_y((B\setminus J^1_{u})^\xi_y)\,{\rm d}\Hd{y}.
    \end{equation}
 For $\xi=0$, we set $\sigma^\xi_u(B)=0$ for every  Borel set $B\subset \Omega$ .

\begin{remark}\label{remark BD}
   Suppose that $u\in BD(\Omega)$, the space of {\textit{functions of bounded deformation}}, and let $Eu$ the distributional symmetric gradient of $u$, which by definition belongs to the space $\mathcal{M}_b(\Omega;\Rdsym)$. From the Structure Theorem for $BD(\Omega)$ functions \cite[Theorem 4.5]{AmbrCosciaDalM} it follows that 
   \begin{equation*}
       \sigma^{\xi}_u(B)=Eu(B\setminus J^1_u)\xi\cdot\xi
   \end{equation*}
   for every Borel set $B\subset \Omega$ and $\xi\in\Rd$. Since $\sigma_u^\xi$ is a measure, this equality can be extended to all functions $u\in BD_{\rm loc}(\Omega)$ such that $Eu\in\mathcal{M}_b(\Omega;\Rdsym)$.  This implies that in this case the function $\xi\mapsto\sigma^\xi_u(B)$ is quadratic in the sense of Definition \ref{def quadratic}. In particular, if $d=1$  this happens for every $u\in GBD(\Omega)$ thanks  to  Remark  \ref{1d gbd}.
\end{remark}
\begin{remark}\label{coincidenza sigma lambda}
    Let $\lambda^\xi_u$ be the measure introduced in \eqref{eq:defmuxi}. One can see that for every $\xi\in\Sd$, we have the equality $|\sigma^\xi_u|=\lambda^\xi_u\mres(\Omega\setminus J^1_u)$ as Borel measures on $\Omega$. In light of \cite[Theorem 3.103]{AmbrosioFuscoPallara}, this is an easy consequence of \eqref{def sigma} and of the fact that, by Theorem \ref{fine propr of GBD}, for $\hd$-a.e.\ $y\in \Pi^\xi$ we have the inclusion $J^1_{u^\xi_y}\subset (J^1_u)^\xi_y$. This fact, together with \eqref{def lambda u} and  the 2-homogeneity of the function $\xi\mapsto\sigma_u^\xi(\Omega)$ proved in Proposition \ref{prop: 2 homogeneity} below, implies that 
    \begin{equation} \label{lower bound}
        |\sigma^{\xi}_u|(\Omega)\leq |\xi|^2\lambda_u(\Omega\setminus J^1_u).
    \end{equation}
\end{remark}

 For $R>0$ let $\tau_R\colon \R\to\R$ be the $1$-Lipschitz functions defined by \begin{equation}\label{def tauR}
    \tau_R(s):=\begin{cases}-\frac{R}{2}&\text{ if }s\leq -\frac{R}{2},\\
        s& \text{if }\frac{-R}{2}\leq s\leq \frac{R}{2},\\
        \frac{R}{2}&\text{ if }s\geq \frac{R}{2}.
        \end{cases}
\end{equation} Thanks to Remark \ref{Lipschitz}, we have that $D_\xi(\tau_R(u\cdot\xi))\in \mathcal{M}_b(\Omega;\R)$ for every $u\in GBD(\Omega)$, $\xi\in\Sd$, and $R>0$. The following result shows that for Borel sets that do not intersect $J^1_u$ we can obtain the value of $\sigma^\xi_u(B)$ by considering the limit of   $D_\xi(\tau_R(u\cdot\xi))(B)$ as $R\to+\infty$.
\begin{proposition}\label{truncation}
    Let $u\in GBD(\Omega)$. Then for every $\xi\in\Sd$ we have
    \begin{gather}\label{teorema truncation 1}
\sigma_u^\xi(B)=\lim_{R\to+\infty}D_\xi(\tau_R(u\cdot\xi))(B),\\\label{teorema truncation 2}
\lambda^\xi_u(B)=\lim_{R\to+\infty}|D_\xi(\tau_R(u\cdot\xi))|(B),
    \end{gather}
    for every Borel set $B\subset \Omega$ with $\hd(B\cap J^1_u)=0$.
\end{proposition}
\begin{proof}

Let us fix $\xi\in\Sd$ and a Borel set $B\subset \Omega$ as in the statement. Since   $\hd(B\cap J^1_u)=0$, it follows that $\hd(\pi^\xi(B\cap J^1_u))=0$ and hence $B^\xi_y\cap (J^1_u)^\xi_y=\emptyset$ for $\hd$-a.e.\ $y\in \Pi^\xi$.  Recalling that by Theorem \ref{fine propr of GBD}  for $\hd$-a.e.\  $y\in \Pi^\xi$ the inclusion in \eqref{Jxihatu} holds, we also have that
\begin{equation}\label{4.4 bis}
B^\xi_y\cap J^{1}_{u^\xi_y}=\emptyset \text{ for $\hd$-a.e.\ $y\in \Pi^\xi$}.
\end{equation}Thus, by \eqref{eq:defmuxi} and \eqref{def sigma} we have 
\begin{gather}\label{come si sigma}
\sigma^\xi_u(B)= \int_{\Pi^\xi}Du^\xi_y((B\setminus J^1_u)^\xi_y)\, {\rm d}\hd(y)=\int_{\Pi^\xi}Du^\xi_y(B^\xi_y)\, {\rm d}\hd(y),\\ \label{come si lambda}
\lambda^\xi_u(B)= \int_{\Pi^\xi}|Du^\xi_y|(B^\xi_y\setminus J^{1}_{u^\xi_y})\, {\rm d}\hd(y)=\int_{\Pi^\xi}|Du^\xi_y|(B^\xi_y)\, {\rm d}\hd(y).
\end{gather}
In particular, since $\lambda^\xi_u\in \mathcal{M}^+_b(\Omega)$, we have \begin{equation}\label{integrable}\int_{\Pi^\xi}|Du^\xi_y|(B^\xi_y)\, {\rm d}\hd(y)<+\infty.\end{equation} 

We then remark that from (b) of  Definition \ref{def:GBD} and the chain rule for BV-functions (see \cite[Theorem 3.99]{AmbrosioFuscoPallara}) it follows that for $\mathcal{H}^{d-1}$-a.e.\ $y\in\Pi^\xi$ we have  $\tau_R(u^\xi_y)\in BV_{\rm loc}(\Omega^\xi_y)$ 
 and 
 \begin{equation}\label{eq: 1d chain rule}
D(\tau_R(u^{\xi}_y))=\tau_R'(v^\xi_y)\nabla u^{\xi}_y\Lb^1+\tau_R'(v^\xi_y)D^cu^\xi_y+[\tau_R(u^{\xi}_y)]\hzero\mres J_{\tau_R(u^\xi_y)},
 \end{equation}
 where for every $t\in\Omega^\xi_y\setminus J_{u^\xi_y}$ the function $v^\xi_y$ is defined by \begin{equation*}v^\xi_y(t):=\lim_{\e\to 0^+}\frac{1}{2\e}\int_{-\e}^{\e}u^\xi_y(t+s)\,{\rm d}s,\end{equation*}
 and we set $\tau^{\prime}_R(\pm \frac{R}{2})=0$.
 Since the measures on the  right-hand side of \eqref{eq: 1d chain rule} are mutually singular, we have also
\begin{equation}\label{eq: 1d chain rule totalvar}
|D(\tau_R(u^{\xi}_y))|=|\tau_R'(v^\xi_y)\nabla u^{\xi}_y|\Lb^1+|\tau_R'(v^\xi_y)||D^cu^\xi_y|+|[\tau_R(u^{\xi}_y)]|\hzero\mres J_{\tau_R(u^\xi_y)}.
 \end{equation}
 Moreover, we observe that for every $y\in \Pi^\xi$ such that $u^\xi_y\in BV_{\rm loc}(\Omega^\xi_y)$ we have  
\begin{gather}\nonumber 
    |D^cu^\xi_y|(\Omega^\xi_y\setminus J_{u^\xi_y})=|D^cu^\xi_y|(\Omega^\xi_y),\\
    \label{cose 1}
     \lim_{R\to +\infty}\tau_R'(v^\xi_y)(t)=1\quad \text{for } \Lb^1\text{-a.e.}\,\,t\in \Omega^\xi_y,\\ \lim_{R\to +\infty}\tau_R'(v^\xi_y)(t)=1\quad \text{for } |D^cu^\xi_y|\text{-a.e.}\,\,t\in \Omega^\xi_y,\\\label{cose 3}
     J_{u^\xi_y}= \bigcup_{R>0}J_{\tau_R(u^\xi_y)} \,\,\,\text{ and }\lim_{R\to +\infty} [\tau_R(u^{\xi}_y)](t)= [u^{\xi}_y](t)\quad \text{ for every } t\in J_{u^{\xi}_y}.
 \end{gather}
Additionally, from  \eqref{def tauR} and \eqref{eq: 1d chain rule totalvar}  it follows that 
 \begin{equation}\label{eq inequality from above 1d}
     |D(\tau_R(u^{\xi}_y))|(B^\xi_y)\leq \int_{B^\xi_y}|\nabla u^{\xi}_y|\,{\rm dt}+|D^cu^\xi_y|(B^\xi_y)+\int_{ B^\xi_y\cap J_{u^\xi_y}}|[u^{\xi}_y]|\,{\rm d}\hzero=|Du^\xi_y|(B^\xi_y).
 \end{equation}

Recalling that by Remark \ref{remar: finiteness} for $\hd$-a.e.\ $y\in\Pi^\xi$ we have  $|Du^\xi_y|(B^\xi_y)<+\infty$, from \eqref{eq: 1d chain rule}-\eqref{cose 3}
 and the Dominated Convergence Theorem we deduce for $\hd$-a.e $y\in\Pi^\xi$ that
     \begin{gather}\label{conclusion 1d}
\lim_{R\to+\infty}D(\tau_R(u^\xi_y))(B^\xi_y)=Du^\xi_y(B^\xi_y),\\\label{conclusione 1d 1}
 \lim_{R\to+\infty}|D(\tau_R(u^\xi_y))|(B^\xi_y)=|Du^\xi_y|(B^\xi_y).
\end{gather}
 Thanks to the general theory of slicing  (see \cite[Theorem 3.107]{AmbrosioFuscoPallara}) and \cite[Proposition 3.1]{DalMasoJems}) for every $R>0$ we have that
\begin{gather}\label{directional slice 4.3}
    D_\xi(\tau_R(u\cdot\xi))(B)=\int_{\Pi^\xi}D(\tau_R (u^\xi_y))(B^{\xi}_y)\,{\rm d}\hd(y),\\\label{directional slice 4.3 variation}
     |D_\xi(\tau_R(u\cdot\xi))|(B)=\int_{\Pi^\xi}|D(\tau_R (u^\xi_y))|(B^{\xi}_y)\,{\rm d}\hd(y),
\end{gather}
so that by \eqref{eq: 1d chain rule} and \eqref{eq: 1d chain rule totalvar} we have
\begin{align*}
        D_\xi(\tau_R(u\cdot\xi))(B)&=\int_{\Pi^\xi}\Big(\int_{B^\xi_y}\tau_R'(v^\xi_y)\nabla u^{\xi}_y\,{\rm d}t\\  &\,\,\,+\int_{B^\xi_y}\tau_R'(v^\xi_y)\frac{{\rm d}D^cu^\xi_y}{{\rm d}|D^cu^\xi_y|}\,{\rm d}|D^cu^\xi_y|+\int_{ B^\xi_y\cap J_{\tau_R(u^\xi_y)}}\!\!\!\![\tau_R(u^{\xi}_y)]\,{\rm d}\hzero\Big){\rm d}\hd(y),\\
         |D_\xi(\tau_R(u\cdot\xi))|(B)&=\int_{\Pi^\xi}\Big(\int_{B^\xi_y}|\tau_R'(v^\xi_y)\nabla u^{\xi}_y|\,{\rm d}t\\  &\,\,\,+\int_{B^\xi_y}\tau_R'(v^\xi_y)\,{\rm d}|D^cu^\xi_y|+\int_{ B^\xi_y\cap J_{\tau_R(u^\xi_y)}}\!\!\!\!|[\tau_R(u^{\xi}_y)]|\,{\rm d}\hzero\Big){\rm d}\hd(y),
\end{align*}
Finally, recalling \eqref{eq:def GBD}, \eqref{4.4 bis},  and \eqref{eq inequality from above 1d},  we can apply  the Dominated Convergence Theorem  and from  \eqref{come si sigma}, \eqref{come si lambda}, and  \eqref{conclusion 1d}-\eqref{directional slice 4.3 variation} we obtain \eqref{teorema truncation 1} and  \eqref{teorema truncation 2}. 
\end{proof}

 The main goal of Sections \ref{section main}-\ref{section d larger} will be  proving that for every $u\in GBD(\Omega)$ and every Borel set $B\subset \Omega$  the function $\xi\mapsto \sigma^\xi_u(B)$ is quadratic. Toward this end, in the rest of this section we investigate some of the properties of the measure defined by \eqref{def sigma}.
We first show that the function $\xi\mapsto\sigma^{\xi}_u(B)$ is $2$-homogeneous. 

\begin{proposition}\label{prop: 2 homogeneity}
    Let $u\in GBD(\Omega)$ and let $B\subset \Omega$ be a Borel set. Then the function $\xi\mapsto \sigma^\xi_u(B)$ is 2-homogeneous.
\end{proposition}
\begin{proof}
  Let us fix $t\neq0$ and $\xi \in \Rd\setminus\{0\}$. By \eqref{2.13 bis}, with $B$ replaced by $B\setminus J^1_u$, we have 
     \begin{equation}
\nonumber Du^{t\xi}_y((B\setminus J^1_u)^{t\xi}_y)=|t|Du^{\xi}_y((B\setminus J^1_u)^{\xi}_y).
    \end{equation}
    Hence, 
$$
       \sigma^{t\xi}_u(B)= |t\xi|\int_{\Pi^\xi}Du^{t\xi}_y((B\setminus J^1_u)^{t\xi}_y)\,{\rm d}\mathcal{H}^1{(y)}
        =t^2|\xi|\int_{\Pi^\xi}Du^{\xi}_y\big((B\setminus J^1_u)^{\xi}_y\big)\,{\rm d}\mathcal{H}^1(y)
        =t^2     \sigma^\xi_u(B).
$$
 This shows that $\xi\mapsto\sigma_u^\xi(B)$ is $2$-homogeneous, concluding the proof.
\end{proof}

In the next proposition we give an explicit formula for $\sigma^\xi_u(B)$ when the Borel set $B$ is contained in $J_u$. This shows in particular that in this case the function $\xi\mapsto\sigma^\xi_u(B)$ is quadratic in the sense of Definition \ref{def quadratic}.
\begin{proposition}
    \label{lemma:jumps smaller than 1}Let $u\in GBD(\Omega)$ and let $B\subset J_u$ be a Borel set. Then for every $\xi\in\Rd\setminus\{0\}$ we have
\begin{align}\label{claim:jump without slice}
&\displaystyle \sigma^\xi_u(B)=\int_{B\setminus J^1_u}(([u]\odot \nu_u)\xi\cdot\xi)\, {\rm d} \Hd{y}.
\end{align}

\end{proposition}
\begin{proof}
 Let us fix $\xi\in\Rd\setminus\{0\}$. Since a  change of sign of $\nu_u$ implies a change of sign of $[u]$, we may assume without loss of generality that $\nu_u\cdot\xi\geq 0$. 
Thanks to Proposition \ref{thm:L1smalljumps} the integral in the right-hand side of \eqref{claim:jump without slice} is well-defined. Hence, by \eqref{eq:jump to sliced jump} and the Area Formula \cite[12.4]{Simon1984} we infer
\begin{gather*}
       \sigma^\xi_u(B)=|\xi|\int_{\Pi^\xi}Du^\xi_y((B\setminus J^1_u)^\xi_y)\, {\rm d}\Hd{y}
        =|\xi|\int_{\Pi^\xi}\Big(\int_{(B\setminus J^1_u)^\xi_y}[u^\xi_y](t)\,{\rm d}\mathcal{H}^0(t)\Big){\rm d}\Hd{y}\\=\int_{B\setminus J^1_u}([u]\cdot\xi)(\nu_u\cdot\xi)\, {\rm d}\Hd{y}=\int_{B\setminus J^1_u}(([u]\odot \nu_u)\xi\cdot\xi)\, {\rm d}\Hd{y},
    \end{gather*}
    which concludes the proof of \eqref{claim:jump without slice}.
\end{proof}

Given two $\hd$-measurable sets $A,B\subset \Rd$, we write $A\simeq B$ when $\hd(A\triangle B)=0$, where $\triangle$ denotes the symmetric difference. We now present a decomposition result for $GBD(\Omega)$ functions, which states that any function $u\in GBD(\Omega)$ can be written as the sum of two functions $v$ and $w$, with $v\in SBV(\Omega;\Rd)$, $w\in GBD(\Omega)$, $J_v\simeq J_u\setminus J^1_u$, and $J_w\simeq J^1_u$. We refer to \cite{AmbrosioFuscoPallara} for the definition and properties of the space $SBV$ of {\it special functions of bounded variation}.
  \begin{proposition}\label{Gagliardo}Let $u\in GBD(\Omega)$. Then there exists $v\in SBV(\Omega;\Rd)$ such that
\begin{gather}\label{simeq}
    J_v\simeq J_u\setminus J^1_u,\\ \label{jump v zero} [v]=[u]\,\,\,\, \text{ and }\,\,\,\, \nu_v=\nu_u \,\,\, \text{ $\hd$-a.e.\ on } J_u\setminus J^1_u.
\end{gather}
In particular, setting $w:=u-v$, we have $w\in GBD(\Omega)$,  $u=w+v$, \begin{gather} \label{jump set w}J_w\simeq J^1_w\simeq J^1_u,\\ [\label{jump w}w]=[u] \,\,\,\,\text{  $\hd$-a.e.\ $ J^1_u$}.
\end{gather}
\end{proposition}
\begin{proof}
    Since $J_u\setminus J^1_u$ is $(\hd,d-1)$-rectifiable and  $[u]$ is integrable on $J_u\setminus J^1_u$  by Proposition \ref{thm:L1smalljumps}, the proof of the statements concerning $v$ can be obtained arguing  as in \cite[Theorems 3.1 and 4.1]{DalToaGagliardo}. 

    The inclusion $w\in GBD(\Omega)$ is due to  the vector space properties of $GBD(\Omega)$. Equalities \eqref{jump set w} and  \eqref{jump w} follow from \eqref{simeq} and \eqref{jump v zero}.
 \end{proof}

 From this result, we derive the following useful consequence. 
 \begin{lemma}\label{lemma gagliardo}
     Let $u\in GBD(\Omega)$ and  $v\in SBV(\Omega;\Rd)$ be as in Proposition \ref{Gagliardo} and let $w:=u-v$. Assume that for a Borel set $B\subset\Omega$ the function $\xi\mapsto\sigma^\xi_w(B)$ is quadratic. Then the function $\xi\mapsto\sigma^\xi_u(B)$ is quadratic as well. 
 \end{lemma}
  \begin{proof}
      As $v\in SBV(\Omega;\Rd)$, from Remark \ref{remark BD} it follows immediately  that $\xi\mapsto\sigma_v^\xi(B)$ is quadratic, so that  the function $\xi\mapsto \sigma^\xi_w(B)+\sigma^\xi_v(B)$ is also quadratic. We claim that for every $\xi\in\Rd\setminus \{0\}$ we have
      \begin{equation}\label{claim lemma galgiardo}
        \sigma^\xi_u(B)=\sigma^\xi_w(B)+\sigma^\xi_v(B).
      \end{equation}
      To prove this, let us fix $\xi\in\Rd\setminus\{0\}$. By definition we have $u=w+v$. Thus, for $\hd$-a.e.\ $y\in\Pi^\xi$  it holds 
      \begin{equation}\label{decompo gagliardo}
          Du^\xi_y=Dw^\xi_y+Dv^\xi_y
      \end{equation}
      as Borel measures on $\Omega^\xi_y$. Moreover, by construction we have $J^1_v\simeq\emptyset$ and  $J_w\simeq J^1_w\simeq J^1_u$. This implies that \begin{equation}
           \notag Dw^\xi_y((B\setminus J^1_w)^\xi_y))= Dw^\xi_y((B\setminus J^1_u)^\xi_y)\quad \text{ for $\hd$-a.e.\ $y\in\Pi^\xi$}.
      \end{equation} Since $v\in SBV(\Omega;\R^d)$, from \eqref{simeq} we deduce that $|Dv|(J^1_u)=0$. By slicing we obtain that
      \begin{equation}\nonumber 
          |Dv^\xi_y|((J^1_u)^\xi_y)=0\quad \text{ for $\hd$-a.e.\ $y\in\Pi^\xi$}.
          \end{equation}for $\hd$-a.e.\ $y\in\Pi^\xi$. Recalling that $J^1_v\simeq \emptyset$ by \eqref{simeq} and \eqref{jump v zero}, this implies that  \begin{equation}  \nonumber Dv^\xi_y((B\setminus J^1_v)^\xi_y)=Dv^\xi_y(B^\xi_y) =Dv^\xi_y((B\setminus J^1_u)^\xi_y).\end{equation} These  remarks, together with \eqref{decompo gagliardo}, imply that for $\hd$-a.e.\ $y\in \Pi^\xi$ it holds
      \begin{equation*}
          Dw^\xi_y((B\setminus J^1_w)^\xi_y)+ Dv^\xi_y((B\setminus J^1_v)^\xi_y)= D(w^\xi_y+v^\xi_y)((B\setminus J^1_u)^\xi_y)=Du^\xi_y((B\setminus J^1_u)^\xi_y).
      \end{equation*}
      Integrating this equality, by \eqref{def sigma} we obtain \eqref{claim lemma galgiardo}, concluding the proof.
  \end{proof}

\section{The case of dimension \texorpdfstring{$d=2$}{}}\label{section main}

In this and the next section we assume that $\Omega\subset \R^2$. Our aim is to prove the following result. 
\begin{theorem}\label{prop:mainingredient}
   Let $u\in GBD(\Omega)$  and let $B\subset \Omega$ be a Borel set. Then the function $\xi\mapsto \sigma^\xi_u(B)$
    is quadratic.
\end{theorem}
\begin{proof}Thanks to Proposition \ref{Gagliardo} and Lemma \ref{lemma gagliardo}, it is not restrictive to assume that $J_u\simeq J^1_u$. Moreover, it is enough to prove the result when $B$ is an open set, which will be denoted by $U$.

 To prove that $\xi\mapsto\sigma^\xi_u(U)$ is quadratic  we use Proposition \ref{character quadratic}. 
Since by Proposition \ref{prop: 2 homogeneity} the function $\xi\mapsto\sigma^\xi_u(U)$ is 2-homogeneous and by Remark \ref{coincidenza sigma lambda}   it satisfies the lower bound (c) of Proposition \ref{prop: 2 homogeneity},
 we are left with proving the {\it parallelogram identity} \begin{equation}
    \label{eq:parallelogram identity}\sigma^{\xi+\eta}_u(U)+\sigma_u^{\xi-\eta}(U)=2\sigma^\xi_u(U)+2\sigma^\eta_u(U),
\end{equation}
for every $\xi,\eta\in\R^2\setminus\{0\}$.  

To this aim, we fix $\xi,\eta\in\R^2\setminus\{0\}$. Note that, if $\xi$ and $\eta$ are not linearly independent, then the parallelogram identity follows from $2$-homogeneity, so  we may assume $\xi$ and $\eta$ to be linearly independent. We also note that it is not restrictive  to assume that  $U$ is a parallelogram of the form 
\begin{equation}\label{parallelogramU} 
U=\{s\xi+t\eta:s\in(0,\alpha) \text{ and }t\in(0,\beta)\},
\end{equation}
for suitable constants $\alpha,\beta>0$ and with $U\subset \subset \Omega$. Indeed, every open set $U$ contained in $\Omega$ can be approximated by a sequence $(U_k)_k$ of disjoint unions of such parallelograms for which $\sigma^\zeta_u(U)=\lim_k\sigma^\zeta_u(U_k)$ for every $\zeta\in\{\xi,\eta,\xi+\eta,\xi-\eta\}$.
 
The vectors $\xi$ and $\eta$, as well as the parallelogram $U$, are kept fixed throughout the rest of the proof. 

To prove the parallelogram identity \eqref{eq:parallelogram identity}, we will use Lemma \ref{lemma:Riemann sums 1} to approximate, by means of Riemann sums, each integral appearing in the definition of the terms occurring in \eqref{eq:parallelogram identity}. This will allow us to prove that the obtained approximations satisfy, up to an arbitrarily small error, the parallelogram identity.

In order to construct these approximations, we need to introduce some notation first. Given a point $\omega\in \R^2$, for every $k\in\N$ and for every $i,j\in\Z$ we set (see Figure \ref{parallelogram})
\begin{equation}
\label{def xkij}
x^{k}_{i,j}:=\omega+\tfrac{i}{k}\xi+\tfrac{j}{k}\eta.
\end{equation}
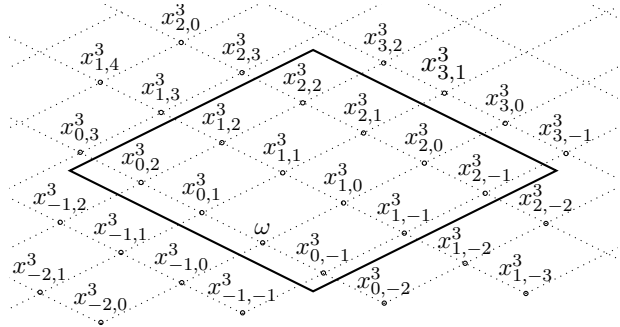
\begin{figure}[h]
    \centering
    \begin{tikzpicture}[scale=0.8]
    \clip (-4.5,-3) rectangle (5.5,4.5);
    \draw[thick]{(0.5,-0.25) -- (4.5,1.75) -- (0.5,3.75) -- (-3.5,1.75) --cycle};
 \def\parallelogram{(0,0) -- (2,1) -- (0.5,1.75) -- (-2,1) -- cycle} ;
  \foreach \x in {-14,-13,-12,-11,-10,-9,-8,-7,-6,-5,-4,-3,-2,-1,0,1,2,3,4,5,6,7,8,9,10,11,12,13,14} {
  \draw[dotted]{(0+\x*2.33+4,\x*0.166-4*0.07142857142)--(12+\x*2.33+4,6+\x*.166-0.07142857142*4)};
  }
  \foreach \y in {-14,-13,-12,-11,-10,-9,-8,-7,-6,-5,-4,-3,-2,-1,0,1,2,3,4,5,6,7,8,9,10,11,12,13,14} {
  \draw[dotted]{(0+\y*2.33+4,\y*0.166-4*0.07142857142)--(-12+\y*2.33+4,6+\y*0.166-0.07142857142*4)}
  ;
  }
  \draw(0+1+1.33,0*0.166-4*0.07142857142+0.66+1.5)  circle (1.0 pt) node [above,scale=0.85] {$x^3_{2,0}$};
\draw(0+1+2*1.33,0*0.166-4*0.07142857142+0.66+1.5+0.66)  circle (1.0 pt) node [above,scale=0.85] {$x^3_{3,0}$};
\draw(0+2*1.33,0*0.166-4*0.07142857142+0.66+2+0.66)  circle (1.0 pt) node [above] {$x^3_{3,1}$};
\draw(0-1+2*1.33,0*0.166-4*0.07142857142+0.66+2.5+0.66)  circle (1.0 pt) node [above,scale=0.85] {$x^3_{3,2}$};
\draw(0+2+2*1.33,0*0.166-4*0.07142857142+0.66+1+0.66)  circle (1.0 pt) node [above,scale=0.85] {$x^3_{3,-1}$};
  
  \draw(0+2+1.33,0*0.166-4*0.07142857142+0.66+1)  circle (1.0 pt) node [above,scale=0.85] {$x^3_{2,-1}$};
  \draw(0+3+1.33,0*0.166-4*0.07142857142+0.66+0.5)  circle (1.0 pt) node [above,scale=0.85] {$x^3_{2,-2}$};
\draw(0+1.33,0*0.166-4*0.07142857142+0.66+2)  circle (1.0 pt) node [above,scale=0.85] {$x^3_{2,1}$};
\draw(-1+1.33,0*0.166-4*0.07142857142+0.66+2.5)  circle (1.0 pt) node [above,scale=0.85] {$x^3_{2,2}$};
\draw(-2+1.33,0*0.166-4*0.07142857142+0.66+3)  circle (1.0 pt) node [above,scale=0.85] {$x^3_{2,3}$};
\draw(-2,0*0.166-4*0.07142857142+3)  circle (1.0 pt) node [above,scale=0.85] {$x^3_{1,3}$};\draw(-3,0*0.166-4*0.07142857142+3.5)  circle (1.0 pt) node [above,scale=0.85] {$x^3_{1,4}$};
\draw(-1,0*0.166-4*0.07142857142+2.5)  circle (1.0 pt) node [above,scale=0.85] {$x^3_{1,2}$};
\draw(0,0*0.166-4*0.07142857142+2)  circle (1.0 pt) node [above,scale=0.85] {$x^3_{1,1}$};
\draw(1,0*0.166-4*0.07142857142+1.5)  circle (1.0 pt) node [above,scale=0.85] {$x^3_{1,0}$};
\draw(2,0*0.166-4*0.07142857142+1)  circle (1.0 pt) node [above,scale=0.85] {$x^3_{1,-1}$};
\draw(3,0*0.166-4*0.07142857142+0.5)  circle (1.0 pt) node [above,scale=0.85,yshift=-0.05 cm] {$x^3_{1,-2}$};
\draw(4,0*0.166-4*0.07142857142)  circle (1.0 pt) node [above,scale=0.85] {$x^3_{1,-3}$};
\draw(1-1.33,0*0.166-4*0.07142857142+1.5-0.66)  circle (1.0 pt) node [above,scale=0.85] {$\omega$};
\draw(2-1.33,0*0.166-4*0.07142857142+1-0.66)  circle (1.0 pt) node [above,scale=0.85] {$x^3_{0,-1}$};
\draw(3-1.33,0*0.166-4*0.07142857142+0.5-0.66)  circle (1.0 pt) node [above,scale=0.85,yshift=-0.05 cm] {$x^3_{0,-2}$};
\draw(-1.33,0*0.166-4*0.07142857142+2-0.66)  circle (1.0 pt) node [above,scale=0.85] {$x^3_{0,1}$};
\draw(-2.33,0*0.166-4*0.07142857142+2.5-0.66)  circle (1.0 pt) node [above,scale=0.85] {$x^3_{0,2}$};
\draw(-3.33,0*0.166-4*0.07142857142+3-0.66)  circle (1.0 pt) node [above,scale=0.85] {$x^3_{0,3}$};
\draw(-3+1.33,0*0.166-4*0.07142857142+0.66+3.5)  circle (1.0 pt) node [above,scale=0.85] {$x^3_{2,0}$};
\draw(1-2*1.33,0*0.166-4*0.07142857142+1.5-2*0.66)  circle (1.0 pt) node [above,scale=0.85] {$x^{3}_{-1,0}$};
\draw(1-3*1.33,0*0.166-4*0.07142857142+1.5-3*0.66)  circle (1.0 pt) node [above,scale=0.85] {$x^{3}_{-2,0}$};
\draw(-3*1.33,0*0.166-4*0.07142857142+2-3*0.66)  circle (1.0 pt) node [above,scale=0.85] {$x^{3}_{-2,1}$};
\draw(-2*1.33,0*0.166-4*0.07142857142+2-2*0.66)  circle (1.0 pt) node [above,scale=0.85] {$x^{3}_{-1,1}$};
\draw(-1-2*1.33,0*0.166-4*0.07142857142+2.5-2*0.66)  circle (1.0 pt) node [above,scale=0.85] {$x^{3}_{-1,2}$};
\draw(2-2*1.33,0*0.166-4*0.07142857142+1-2*0.66)  circle (1.0 pt) node [above,scale=0.85] {$x^{3}_{-1,-1}$};
    \end{tikzpicture}
    \vspace{-1.5 cm}
    \caption{\label{parallelogram} The parallelogram $U$ and the grid of points $x^{k}_{i,j}$ associated to $\omega\in U$ and $k=3$}
\end{figure}

\noindent Since the points $x^k_{i,j}$ will be instrumental to the discretisation of  the summands in \eqref{eq:parallelogram identity}, which are integrals over the straight lines $\Pi^\zeta$ for $\zeta\in\{\xi,\eta,\xi+\eta,\xi-\eta\}$, we consider also the projections  of the points $x^{k}_{i,j}$ onto these lines. For every $k\in\N$, $i,j\in\Z$, and  $\zeta\in\{\xi,\eta,\xi+\eta,\xi-\eta\}$ we set
\begin{equation}\label{def projections}
y_{i,j}^{k,\zeta}:=\pi^{\zeta}(x_{i,j}^{k})=\pi^{\zeta}(\omega)+\tfrac{i}{k}\pi^{\zeta}(\xi)+\tfrac{j}{k}\pi^{\zeta}(\eta)\in \Pi^{\zeta}. 
\end{equation}
We observe that   $y^{k,\xi}_{i,j}$ \text{ depends only on }$j$ \text{ and }  $y^{k,\eta}_{i,j}$ \text{ depends only on }$i$, while $y^{k,\xi+\eta}_{i,j}$  depends only on $i-j $ and  $y^{k,\xi-\eta}_{i,j}$ \text{ depends only on } $i+j$.
When we want to underline the dependence of these families on a single index, we set 
\begin{equation}\label{1 parameter family}
\begin{gathered}
    y^{k,\xi}_j=y^{k,\xi}_{0,j}, \quad \,\,\, y^{k,\eta}_i=y^{k,\eta}_{i,0},\\
     y^{k,\xi+\eta}_j=y^{k,\xi+\eta}_{j,0}=y^{k,\xi+\eta}_{0,-j}, \quad  y^{k,\xi-\eta}_j=y^{k,\xi-\eta}_{j,0}=y^{k,\xi-\eta}_{0,j},
\end{gathered}
\end{equation}
see Figure \ref{figure projections}.
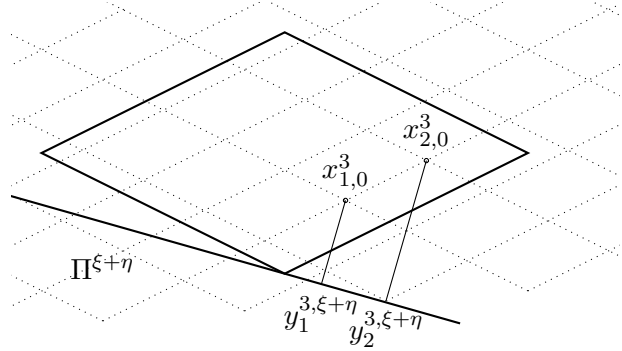
\begin{figure}[h]\vspace{-0.3 cm}
\begin{tikzpicture}[scale=0.8]
     \clip (-4.5,-3) rectangle (5.5,4.5);
    \draw[thick]{(0,0) -- (4,2) -- (0,4) -- (-4,2) --cycle};
  \foreach \x in {-14,-13,-12,-11,-10,-9,-8,-7,-6,-5,-4,-3,-2,-1,0,1,2,3,4,5,6,7,8,9,10,11,12,13,14} {
  \draw[dotted]{(0+\x*2.33+4,\x*0.166-4*0.07142857142)--(12+\x*2.33+4,6+\x*.166-0.07142857142*4)};
  }
  \foreach \y in {-14,-13,-12,-11,-10,-9,-8,-7,-6,-5,-4,-3,-2,-1,0,1,2,3,4,5,6,7,8,9,10,11,12,13,14} {
  \draw[dotted]{(0+\y*2.33+4,\y*0.166-4*0.07142857142)--(-12+\y*2.33+4,6+\y*0.166-0.07142857142*4)}
  ;
  }
  \draw[thick]{(3*3.5*0.27472112789,-3*0.27472112789)--(-10*3.5*0.27472112789,10*0.27472112789)};
  \draw (-3,0.5) node [below] {$\Pi^{\xi+\eta}$};
\draw(1+1.33,0*0.166-4*0.07142857142+0.66+1.5)  circle (1.0 pt) node [above] {$x^3_{2,0}$};
\draw[thin]{(1+1.33,0*0.166-4*0.07142857142+0.66+1.5)--(-0.66+1+1.33,0*0.166-4*0.07142857142+0.66+1.5-0.66*3.5)};\draw[thin]{(-1.33+1+1.33,0*0.166-4*0.07142857142+0.66+1.5-0.66)--(-0.33-1.33+1+1.33,0*0.166-4*0.07142857142+0.66+1.5-0.66-0.33*3.5)};
\draw[thin]{(-0.33-1.33+1+1.33,0*0.166-4*0.07142857142+0.66+1.5-0.66-0.33*3.5)--(-0.40-1.33+1+1.33,0*0.166-4*0.07142857142+0.66+1.5-0.66-0.40*3.5)};
\draw(-1.33+1+1.33,0*0.166-4*0.07142857142+0.66+1.5-0.66)  circle (1.0 pt) node [above] {$x^3_{1,0}$};
\draw(-0.40-1.33+1+1.33,0*0.166-4*0.07142857142+0.66+1.5-0.66-0.40*3.5) node [below] {$y^{3,\xi+\eta}_{1}$};
\draw(-0.66+1+1.33,0*0.166-4*0.07142857142+0.66+1.5-0.66*3.5) node [below] {$y^{3,\xi+\eta}_{2}$};
\end{tikzpicture}\vspace{-1.5 cm}
    \caption{Projections of the points $x^{k}_{i,j}$ onto  the straight line $\Pi^{\xi+\eta}$}
    \label{figure projections}
\end{figure}

It is clear from these definitions that for every $\zeta\in\{\xi,\eta,\xi+\eta,\xi-\eta\}$ and $i,j\in\Z$ there exists a unique real number $t^{k,\zeta}_{i,j}$ 
such that 
\begin{equation}\label{def:coordinates}
x_{i,j}^k=y_{i,j}^{k,\zeta}+t^{k,\zeta}_{i,j}\zeta.
\end{equation}

Let $C_{\xi,\eta}:=\big(|\xi|^2|\eta|^2-(\xi\cdot\eta)^2\big)^{1/2}>0$. We observe that for $i,j\in\Z$ we have 
\begin{equation}\label{def ellexi}
\begin{gathered}
    \displaystyle  k|y^{k,\xi}_{i,j+1}-y_{i,j}^{k,\xi}|=|\pi^{\xi}(\eta)|=\frac{1}{|\xi|}C_{\xi,\eta},
    \\
    \displaystyle   k|y_{i+1,j}^{k,\eta}-y_{i,j}^{k,\eta}|=|\pi^{\eta}(\xi)|=\frac{1}{|\eta|}C_{\xi,\eta}
 \\
    \displaystyle   k|y_{i,j}^{k,\xi+\eta}-y_{i,j+1}^{k,\xi+\eta}|=k|y_{i,j}^{k,\xi+\eta}-y_{i+1,j}^{k,\xi+\eta}|=|\pi^{\xi+\eta}(\xi)|=|\pi^{\xi+\eta}(\eta)|=\frac{1}{|\xi+\eta|}C_{\xi,\eta},
    \\
    \displaystyle k |y_{i,j}^{k,\xi-\eta}-y_{i,j+1}^{k,\xi-\eta}|= k |y_{i,j}^{k,\xi-\eta}-y_{i-1,j}^{k,\xi-\eta}|=|\pi^{\xi-\eta}(\xi)|=|\pi^{\xi-\eta}(\eta)|=\frac{1}{|\xi-\eta|}C_{\xi,\eta}.
\end{gathered}
\end{equation} 

  For technical reasons, which will appear in Lemmas \ref{lemma measurable R2} and \ref{measure app 1}, it is convenient to replace the set $J_u=J^1_u$ by a set $J\subset U$ that can be written as countable union of compact sets.  Since $\mathcal{H}^{1}(J^1_u)<+\infty$ by Remark \ref{re:RemarkAlmiTasso}, there exist a countable family of compact sets $K_n\subset J^1_u\cap U$ and two Borel sets $N_1\subset N\subset U$ with $\huno(N_1)=\huno(N)=0$ such that 
  \begin{equation}\label{J1 rectifiable}
     J^1_u\cap U=\big(\bigcup_{n\in\N}K_n\big)\cup N_1\quad \text{and}\quad J_u\cap U=\big(\bigcup_{n\in\N}K_n\big)\cup N.
  \end{equation}
  We set \begin{equation}\label{def Jhat}
      J:=\bigcup_{n\in\N}K_n
  \end{equation}
  and observe that for every $\zeta\in\{\xi,\eta,\xi+\eta,\xi-\eta\}$ and for $\huno$-a.e.\ $y\in\Pi^\zeta$ we have the equality $(J^1_u\cap U)^\zeta_y=(J_u\cap U)^\zeta_y=J^\zeta_y$. In particular, by \eqref{def sigma} we have that
  \begin{equation} \label{sigma con hat}
\sigma_u^\zeta(B):=|\zeta|\int_{\Pi^\zeta}Du^\zeta_y((B\setminus  J)^\zeta_y)\,{\rm d}\huno(y)
  \end{equation}
  for every $\zeta\in\{\xi,\eta,\xi+\eta,\xi-\eta\}$ and every Borel set $B\subset U$.
  
   The following two lemmas will be used in the choice of $\omega$ to obtain an approximation of $\sigma^\zeta_u(U)$ by means of suitable Riemann sums.
\begin{lemma}\label{fundamental lemma 1} 
Let  $\zeta\in\{\xi,\eta,\xi+\eta,\xi-\eta\}$. Then  for $\Lb^2$-a.e.\ $\omega\in U$ the following conditions are simultaneously satisfied:
\begin{enumerate}
     \item [(a)]{\it properties of the slices}:  for every $k\in \N$ and $i,j\in \Z$ we have \begin{gather}\nonumber 
    u^{\zeta}_{y^{k,\zeta}_{i,j}}\in BV(U^{\zeta}_{\!y^{k,\zeta}_{i,j}}) \quad \text{ and } \quad 
     (J^1_u\cap U)^{\zeta}_{y^{k,\zeta}_{i,j}}= J^{\zeta}_{y^{k,\zeta}_{i,j}};
     \end{gather}
        \item[(b)]{\it the points $x^{k}_{i,j}$ are directional Lebesgue points}: for every $k\in \N$ and  $i,j\in \Z$, with $x^k_{i,j}\in U$, we have 
    \begin{equation}\label{claim lebesgue}
        \lim_{\e\to 0^+}\frac{1}{2\e}\int_{-\e}^\e|u(x^{k}_{i,j}+s\zeta)\cdot\zeta-u(x^{k}_{i,j})\cdot \zeta|\,{\rm d}s=0.
    \end{equation}
\end{enumerate}  
\end{lemma}
\begin{proof}
    Taking into account that $U\subset \subset \Omega$ and recalling the definition of $GBD(\Omega)$ and Remark \ref{remar: finiteness}, from  \eqref{J1 rectifiable} and \eqref{def Jhat} it follows that there exists a Borel set $N_\zeta\subset \Pi^\zeta$, with $\mathcal{H}^1(N_\zeta)=0$, such that for every $y\in \Pi^\zeta\setminus N_\zeta$ we have $u^\zeta_y\in BV(U^\zeta_y)$ and $(J^1_u\cap U)^\zeta_y= J^\zeta_y$. Let $N^\infty_\zeta:=\bigcup_{(i,j)\in\Z^2}\big(N_\zeta-\tfrac{i}{k}\pi^\zeta(\xi)-\tfrac{j}{k}\pi^\zeta(\eta)\big)$.   It is immediate to check that $\huno(N^\infty_\zeta)=0$. By  \eqref{def projections} we have
\begin{equation*}
y^{k,\zeta}_{i,j}=\pi^\zeta(\omega)+\tfrac{i}{k}\pi^\zeta(\xi)+\tfrac{j}{k}\pi^\zeta(\eta),
\end{equation*}
so that if $\pi^\zeta(\omega)\notin N^\infty_\zeta$,  we have $u^{\zeta}_{y^{k,\zeta}_{i,j}}\in BV(U^{\zeta}_{y^{k,\zeta}_{i,j}})$ and that $(J^1_u\cap U)^\zeta_{y^{k,\zeta}_{i,j}}=J^\zeta_{y^{k,\zeta}_{i,j}}$. This proves that for $\Lb^2$-a.e $\omega\in U$ condition (a).

Let us prove (b).  Let $B$ be the $\Lb^2$-measurable set defined by 
\begin{equation}\nonumber 
    B:=\Big\{x\in U:\limsup_{\e\to 0^+}\frac{1}{2\e}\int_{-\e}^\e|u(x+s\zeta)\cdot\zeta-u(x)\cdot\zeta|\,{\rm d}s>0\Big\}.
\end{equation}
For every $y\in \Pi^\zeta\setminus N_\zeta$ we have $u^{\zeta}_{y}\in BV(U^\zeta_y)$ and the slices $B^\zeta_y$ satisfy
\begin{equation}\nonumber 
    B^\zeta_y=\Big\{t\in U^\zeta_y:\limsup_{\e\to 0^+}\frac{1}{2\e}\int_{-\e}^\e|u_y^\zeta(s+t)-u_y^\zeta(t)|\,{\rm d}s>0\Big\}.
\end{equation}
 Therefore, by the Lebesgue Differentiation Theorem $\Lb^1(B^\zeta_y)=0$ for every  $y\in \Pi^\zeta\setminus N_\zeta$ and  by the Fubini Theorem this implies that $\Lb^2(B)=0$. We observe that
 \begin{equation}\label{lebesgue 1d}
    \lim_{\e\to 0^+}\frac{1}{2\e}\int_{-\e}^\e|u^\zeta_y(t+s)-u^{\zeta}_y(t)|\,{\rm d}s=0
\end{equation}
for every $t\in \Pi^\zeta_y\setminus B^\zeta_y$.

Recalling that by \eqref{def:coordinates} we have that
\begin{equation}\label{recall def punto}
x^{k}_{i,j}=y^{k,\zeta}_{i,j}+t^{k,\zeta}_{i,j}\zeta,
\end{equation}
and that $y^{k,\zeta}_{i,j}\notin N_\zeta$ by the first step,
from \eqref{lebesgue 1d} we deduce that \eqref{claim lebesgue} holds whenever 
\begin{equation}\label{qualcosa}
 t^{k,\zeta}_{i,j}\notin N^\zeta_{y^{k,\zeta}_{i,j}}\,.
 \end{equation}

Thus, to prove (b) it is enough to show that, for given $i,j$, and $\zeta$, condition \eqref{qualcosa} holds for $\Lb^2$-a.e.\ $\omega\in U$. Observing that 
$y^{k,\zeta}_{i,j}\cdot\zeta=0$ and recalling \eqref{def xkij}, if we multiply \eqref{recall def punto} by $\zeta/|\zeta|^2$  we obtain that
\begin{equation*}
t^{k,\zeta}_{i,j}=\tfrac{x^{k}_{i,j}\cdot\zeta}{|\zeta|^2}=\tfrac{\omega\cdot \zeta}{|\zeta|^2}+\tfrac{i}{k}\tfrac{\xi\cdot\zeta}{|\zeta|^2}+\tfrac{j}{k}\tfrac{\eta\cdot\zeta}{|\zeta|^2}.
\end{equation*}
Hence, \eqref{qualcosa} holds whenever 
\begin{equation}\label{fine lebesgue}
\tfrac{\omega\cdot\zeta}{|\zeta|^2}\notin N^{\zeta}_{y^{k,\zeta}_{i,j}}-\tfrac{i}{k}\tfrac{\xi\cdot\zeta}{|\zeta|^2}-\tfrac{j}{k}\tfrac{\eta\cdot\zeta}{|\zeta|^2}.
\end{equation}
Recalling that by \eqref{def projections} $y^{k,\zeta}_{i,j}$ has the form $\pi^{\zeta}(\omega)+z$ for some $z\in \Pi^\zeta$, depending on $k$, $\zeta$, $i$, and $j$, we deduce that \eqref{fine lebesgue} holds for $\Lb^2$-a.e.\ $\omega\in U$. This proves that \eqref{qualcosa} holds for $\Lb^2$-a.e.\ $\omega\in U$, concluding the proof.\end{proof}
  
\noindent {\it Proof of Theorem \ref{prop:mainingredient} (continuation).} Given $i,j\in \Z$ and $\zeta\in\{\xi,\eta,\xi+\eta,\xi-\eta\}$ we set 
\begin{gather}\label{def intervals zeta}
   I_{i,j}^{k,\zeta}:=[t_{i,j}^{k,\zeta},t_{i,j}^{k,\zeta}+\tfrac{1}{k}),
\end{gather}
where $t_{i,j}^{k,\zeta}$ are defined in \eqref{def:coordinates}.
We note that 
\begin{equation}\label{5.17bis}
   [x_{i,j}^k,x_{i,j}^k+\tfrac{1}{k}\zeta)=\{y^{k,\zeta}_{i,j}+t\zeta:t\in I^{k,\zeta}_{i,j}\}.
\end{equation}

For $h\in\N$ we set 
\begin{equation}\label{5.18}
\mathcal{J}^k_h:=\{(i,j)\in\Z^2:  x_{i,j}^k\pm\tfrac{h}{k}\zeta \in U \text{ for every }\zeta\in\{\xi,\eta,\xi+\eta,\xi-\eta\}\}.
\end{equation}      
Since every $\omega\in\R^2$ can be written in a unique way as $\omega=s\xi+t\eta$ with $s, t\in\R$, 
by \eqref{parallelogramU} and \eqref{def xkij}
we have
\begin{equation}\label{Jk new}
\mathcal{J}^k_h:=\{(i,j)\in\Z^2: 0<s+\tfrac{i\pm h}{k}<\alpha \text{ and }0<t+\tfrac{j\pm h}{k}<\beta\}.
\end{equation} 

In the following lemma, given  a sequence $(\omega_k)_k$ of elements of $U$, we consider the points $x_{i,j}^k$ and $y^{k,\zeta}_{i,j}$  defined by \eqref{def xkij} and \eqref{def projections} with $\omega=\omega_k$.
We recall that  $C_{\xi,\eta}>0$ is the constant which appears in  \eqref{def ellexi}.

\begin{lemma}\label{fundamental lemma 2} 
There exists an infinite set $K\subset \N$  and, for every $\e>0$, a  Borel set $U_\e\subset U$, with $\Lb^2(U\setminus U_\e)\leq \e$, such that for every sequence $(\omega_k)_{k\in \N}$ in  $U_\e$  and for every  $\zeta\in\{\xi,\eta,\xi+\eta,\xi-\eta\}$  conditions  (a) and (b) of Lemma \ref{fundamental lemma 1} are satisfied and  
\begin{equation} 
\label{eq:riemannXi1}\lim_{\substack{k\to+\infty\\ k\in K}}\frac{C_{\xi,\eta}}{k}\sum_{(i,j)\in\mathcal{J}^{k}}Du^{\zeta}_{\!y_{i,j}^{k,\zeta}}(I^{k,\zeta}_{i,j}\setminus  J^{\zeta}_{y_{i,j}^{k,\zeta}})=\sigma^\zeta_u(U)
\end{equation}  
for every sequence $(\mathcal{J}^k)$ in $\Z^2$ for which there exists $h\in\N$ such that $\mathcal{J}^k_h\subset \mathcal{J}^k\subset \mathcal{J}^k_1$ for every $k\in\N$.
\end{lemma}
\begin{proof} Thanks to Lemma \ref{fundamental lemma 1}, there exists a Borel set $U_0\subset U$, with $\Lb^2(U_0)=\Lb^2(U)$, such that (a) and (b) hold for every $\omega\in U_0$, so that  we only need to show that there exist an infinite set $K\subset \N$ and for every $\e>0$, a Borel set $U_\e\subset U_0$, with  $\Lb^2(U\setminus U_\e)\leq \e$, such that \eqref{eq:riemannXi1} holds.
The proof for every $\zeta\in\{\xi,\eta,\xi+\eta,\xi-\eta\}$ will be carried out in four steps, first for $\zeta=\xi$, then for $\zeta=\eta$, next for $\zeta=\xi+\eta$, and finally for $\zeta=\xi-\eta$. Starting from the second step, $\N$ is replaced by the set $K$ of the previous step and we may assume that
 $U_\e$ is contained in the corresponding set of the previous step, so that the sets $K$ and $U_\e$ obtained at the end 
satisfy \eqref{eq:riemannXi1} for every $\zeta\in\{\xi,\eta,\xi+\eta,\xi-\eta\}$.

We begin by proving the result for $\zeta=\xi$.
We observe that every $\omega\in\R^2$ can be written in a unique way as 
\begin{equation*}
    \omega=z_1\eta+z_2\xi,
\end{equation*}
with $z_1,z_2\in\R$.
 For every $k\in\N$ and $i,j\in\Z$ by \eqref{def projections} we have 
 \begin{equation}\label{alternative projections}
y^{k,\xi}_{i,j}=\big(z_1+\tfrac{j}{k}\big)\pi^\xi(\eta).
\end{equation}
    We set $I:=(0,\beta)=\{t
    \in\R:t
    \pi^{\xi}(\eta)\in \pi^\xi(U)\}$. For every $h,k\in\N$ and $z_1\in I$ we define
    \begin{gather*}\label{Ikz1}
    \mathcal{I}^{k}_{h}(z_1):=\big\{j\in\Z: 
   z_1+\tfrac{j\pm h}k\in I\big\} =  
    \big\{j\in\Z:     y^{k,\xi}_{0,j}\pm \tfrac{h}{k}\pi^{\xi}(\eta)\in \pi^\xi(U) \big\},
    \\
        \mathcal{F}^{k}_{h}(z_1):=\big\{j\in\Z: z_1+\tfrac{j}k\in I\big\}\setminus      \mathcal{I}^{k}_{h}(z_1).
        \label{Fkz1}
    \end{gather*} 
    Let $N_\xi\subset \Pi^\xi$ be the $\mathcal{H}^1$-negligible Borel set introduced at the beginning of the proof of Lemma \ref{fundamental lemma 1} for $\zeta=\xi$  and   consider the Borel set $M_\xi:=\{t\in\R: t\pi^\xi(\eta)\in N_\xi\}$.  Applying  Lemma \ref{lemma:Riemann sums 1} to the function defined for $t\in\R$ by
    \begin{equation*}\label{function f} 
    \begin{aligned}
       f(t):=\begin{cases} Du^\xi_{t
        \pi^{\xi}(\eta)}((U\setminus  J)^\xi_{t \pi^{\xi}(\eta)})&\text{ if $t\in\R\setminus M_\xi$},\\
        0 &\text{ if $t\in M_\xi$},
        \end{cases}
    \end{aligned}
    \end{equation*}
    which vanishes out of $I$, and recalling  \eqref{def ellexi}, \eqref{sigma con hat}, and \eqref{alternative projections}, we obtain an infinite set $H\subset \N $ and, for every $\e>0$, 
    a Borel set $I_\e\subset I$, with   $\Lb^1(I\setminus I_\e)\leq \e$ and $I_\e\cap M_\xi=\emptyset$, such that for every $h\in\N$ we have
    \begin{gather}\notag
    \lim_{\substack{k\to+\infty\\k\in H}}
    \frac{ C_{\xi,\eta}}{k} \sum_{j\in\mathcal{I}^{k}_{h}{(z_1)}}
   Du^{\xi}_{y^{k,\xi}_{0,j}}((U\setminus  J)^{\xi}_{y^{k,\xi}_{0,j}})
    =|\xi||\pi^\xi(\eta)|\int_{I}Du^{\xi}_{s\pi^\xi(\eta)}((U\setminus  J)^{\xi}_{s\pi^\xi(\eta)})\, {\rm d}s\\\label{funzione 1}
     =|\xi|\int_{\Pi^\xi}Du^{\xi}_{y}((U\setminus  J)^{\xi}_{y})\, {\rm d}\huno(y)=\sigma^\xi_u(U) \qquad\hbox{uniformly for }z_1\in I_\e,
      \\
           \lim_{\substack{k\to+\infty\\k\in H}}
       \frac{1}{k} \sum_{j\in\mathcal{F}^{k}_{h}{(z_1)}}
Du^{\xi}_{y^{k,\xi}_{0,j}}((U\setminus  J)^{\xi}_{y^{k,\xi}_{0,j}})=0 \qquad\hbox{uniformly for }z_1\in I_\e.
\label{5.27}
    \end{gather}
     
   We set   \begin{gather}\label{def Vxi}
   V_\e:=\{\omega\in U:\omega=z_1\eta+z_2\xi,\  z_1\in {I_\e,}\ z_2\in \R\}\end{gather} and  observe that $\Lb^2(U\setminus V_\e)\leq c\e$ 
   for a constant $c>0$ depending only on $\xi$, $\eta$, $\alpha$, and $\beta$. 
  For every $h,k\in\N$ we set  
  \begin{gather}
    \label{def Wk} W^k_h:=\{x\in U:\text{ either }x+\tfrac{h}k\xi \text{ \,\,or\,\, } x-\tfrac{h}k\xi \text{ \,\,\, does not belong to }U \},\\\label{fkxilemma}
f^k_h(t,z_2):=\begin{cases}|Du^\xi_{t\pi^\xi(\eta)}|((W^k_h\setminus  J)^\xi_{t\pi^\xi(\eta)})\quad &\text{ if }t \in I\setminus M_\xi\hbox{ and }z_2\in (0,\alpha+\beta),\\
      0\quad &\text{ otherwise},
      \end{cases}\\  \label{gxilemma}g(t):=\begin{cases}|Du^\xi_{t\pi^\xi(\eta)}|((U\setminus  J)^\xi_{t\pi^\xi(\eta)})\qquad &\text{ for }  t\in I\setminus M_\xi,\\
      0& \text{ otherwise.}\end{cases}
  \end{gather}  
 We observe that $0\le f^k_h(t,z_2)\leq g(t)$ for $\Lb^1$-a.e.\ $t\in I$ and  $\Lb^1$-a.e.\ $z_2\in \R$. 
Let $\lambda$ be a measure as in Definition~\ref{def:GBD}.
Since ${\lambda}(W^k_h)$ converges to $0$ as $k\to+\infty$, from \eqref{def:GBD xiquadro} and \eqref{Jxihatu} we deduce that the sequence $(f^k_h)$ converges  to $0$ in $L^1(I\times \R)$ as $k\to+\infty$. 
   Thanks to Lemma \ref{lemma: Riemann 2} and Remark \ref{remark sequences}, applied with $\N$ replaced by $H$, we can find an infinite set $K\subset H \subset \N$ and a Borel set $U_\e\subset V_\e\subset U$, with $\Lb^2(U\setminus U_\e)\leq c \e$   for a constant $c>0$ depending only on $\xi,\eta$, $\alpha$, and $\beta$, such that for every $h\in\N$
   \begin{align}\label{new applicaiton lemma 2}
 \lim_{\substack{k\to+\infty\\k\in K}}\frac{1}{k}\sum_{j\in\mathcal{I}^{k}_{h}(z_1)}|Du^{\xi}_{\!y_{0,j}^{k,\xi}}|((W^k_h\setminus  J)^{\xi}_{y_{0,j}^{k,\xi}})
= \lim_{\substack{k\to+\infty\\k\in{ K}}}\frac{ 1}{k}\sum_{j\in\mathcal{I}^{k}_{h}{(z_1)}}f_h^k(z_1+\tfrac{j}{k},z_2)=0
   \end{align}
uniformly for $\omega=z_1\eta+z_2\xi\in U_\e$, where the first equality follows from \eqref{alternative projections}.

For every $j\in \Z$ let $ \mathcal{J}^k(j):=\{i\in \Z:(i,j)\in  \mathcal{J}^k\}$.
Since $\mathcal{J}^k_h\subset \mathcal{J}^k\subset \mathcal{J}^k_1$, by \eqref{def intervals zeta} and \eqref{def Wk} for every $j\in \mathcal{I}^{k}_{h}(z_1)$ it holds 
\begin{equation*}
    (U\setminus  J)^{\xi}_{y^{k,\xi}_{0,j}}=(W^k_h\setminus  J)^{\xi}_{y_{0,j}^{k,\xi}}\cup \bigcup_{i\in \mathcal{J}^{k}(j)}\big(I^{k,\xi}_{i,j}\setminus  J^{\xi}_{y^{k,\xi}_{0,j}}\big).
\end{equation*}
Hence,
\begin{equation*}
\Big|D u^\xi_{y^{k,\xi}_{0,j}}\big( (U\setminus  J)^{\xi}_{y^{k,\xi}_{0,j}}\big)
- \sum_{i\in \mathcal{J}^{k}(j)}D u^\xi_{y^{k,\xi}_{0,j}}\big(I^{k,\xi}_{i,j}\setminus  J^{\xi}_{y^{k,\xi}_{0,j}}\big)\Big|
\le \big|D u^\xi_{y^{k,\xi}_{0,j}}\big|   \big((W^k_h\setminus  J)^{\xi}_{y_{0,j}^{k,\xi}}\big).
\end{equation*}
Recalling that  $y_{i,j}^{k,\xi}=y_{0,j}^{k,\xi}$,  the previous inequality
gives
\begin{gather*}
\Big|\sum_{j\in \mathcal{I}^{k}_{h}(z_1)}D u^\xi_{y^{k,\xi}_{0,j}}\big( (U\setminus  J)^{\xi}_{y^{k,\xi}_{0,j}}\big)
- \sum_{(i,j)\in \mathcal{J}^{k}}D u^\xi_{y^{k,\xi}_{0,j}}\big(I^{k,\xi}_{i,j}\setminus  J^{\xi}_{y^{k,\xi}_{0,j}}\big)\Big|
\\
\le \big|D u^\xi_{y^{k,\xi}_{0,j}}\big|   \big((W^k_h\setminus  J)^{\xi}_{y_{0,j}^{k,\xi}}\big) + \sum_{j\in \mathcal{F}^{k}_{h}(z_1)} \big|
D u^\xi_{y^{k,\xi}_{0,j}}\big( (U\setminus  J)^{\xi}_{y^{k,\xi}_{0,j}}\big)\big|.
\end{gather*}
Combining \eqref{funzione 1}, \eqref{5.27}, and  \eqref{new applicaiton lemma 2}, we 
obtain  \eqref{eq:riemannXi1} for $\zeta=\xi$. The proof for the case $\zeta=\eta$ can be obtained by arguing as above, exchanging the roles of $\xi$ and $\eta.$

   In the case $\zeta=\xi+\eta$ we argue as follows. First, we write  every $\omega\in\R^2$ as  
    \begin{equation*}
        \omega=z_1\eta+z_2(\xi+\eta),
    \end{equation*}
    with $z_1,z_2\in\R$, so that by \eqref{def projections} 
    \begin{equation*}
        y^{k,\xi+\eta}_{i,j}=(z_1+\tfrac{j}{k})\pi^{\xi+\eta}(\eta)+\tfrac{i}{k}\pi^{\xi+\eta}(\xi)=(z_1+\tfrac{j-i}{k})\pi^{\xi+\eta}(\eta).
    \end{equation*}
    Setting $m:=j-i$,   we have 
    \begin{equation}\label{proiezione con m}
         y^{k,\xi+\eta}_{i,j}=(z_1+\tfrac{m}{k})\pi^{\xi+\eta}(\eta).
    \end{equation}
We now set $I:=\{t\in\R:t\pi^{\xi+\eta}(\eta)\in\pi^{\xi+\eta}(U)\}$ and for every $k\in\N$ and $z_1\in I$ we define
\begin{equation}\label{def indici in proof}
    \mathcal{I}^{k}(z_1):=\{m\in\Z:z_1+\tfrac{m}{k}\in I\}.
\end{equation}

Let $N_{\xi+\eta}$ be the $\huno$-negligible Borel introduced at the beginning of the proof of Lemma \ref{fundamental lemma 1} for $\zeta=\xi+\eta$ and consider the Borel set $M_{\xi+\eta}:=\{t\in\R: t\pi^{\xi+\eta}(\eta)\in N_{\xi+\eta}\}$. We can apply Lemma \ref{lemma:Riemann sums 1} to the function defined for $t\in\R$ by
\begin{equation*}
    h(t):= \begin{cases} Du^{\xi+\eta}_{t
        \pi^{\xi+\eta}(\eta)}((U\setminus  J)^{\xi+\eta}_{t \pi^{\xi+\eta}(\eta)})&\text{ if $t\in \R\setminus  M_{\xi+\eta}$},\\
        0&\text{ if $t\in M_{\xi+\eta}$}.
        \end{cases}
\end{equation*}
and arguing as in the previous part of the proof  we obtain an infinite set $H $ contained in the set $K$ obtained in the previous steps and, for every $\e>0$, a Borel set $I_\e\subset I$, with   $\Lb^1(I\setminus I_\e)\leq \e$ and $I_\e\cap M_{\xi+\eta}=\emptyset$,  such that 
\begin{gather}\label{uniform in proof}
    \lim_{\substack{k\to+\infty\\k\in H}}
  \frac{ C_{\xi,\eta}}{k}  \sum_{m\in\mathcal{I}^{k}{(z_1)}}
    Du^{\xi+\eta}_{y^{k,\xi+\eta}_{0,m}}((U\setminus  J)^{\xi}_{y^{k,\xi}_{0,m}})=\sigma^{\xi+\eta}_u(U) \qquad\hbox{uniformly for }z_1\in I_\e.
    \end{gather}
    For every $h,k\in\N$ now define 
\begin{equation*}
    W^k_h:=\{x\in U:\text{there exists } \zeta\in\{\pm \xi,\pm\eta,\pm(\xi+\eta),\pm(\xi-\eta)\} \text{ such that }x+\tfrac{h}{k}\zeta\notin U\}
\end{equation*}
and we observe that \eqref{Jk new} and the inclusions $ \mathcal{J}^{k}_h\subset \mathcal{J}^{k}\subset  \mathcal{J}^{k}_1$ imply that
\begin{gather}\notag
  \{(i,j)\in\Z^2:x^k_{i,j}\in U\setminus W^k_h\}\subset   \mathcal{J}^{k},\\
  \notag  \mathcal{J}^{k} \subset \{(i,j)\in\Z^2:[x^k_{i,j},x^k_{i,j}\pm \tfrac1k\zeta]\subset U\hbox{ for every }\zeta\in\{\xi,\eta,\xi+\eta,\xi-\eta\}\}.
\end{gather}
It follows from \eqref{def intervals zeta} and \eqref{5.17bis} that
\begin{equation} \label{Jk aux}
    (U\setminus  J)^{\xi+\eta}_{y^{k,\xi+\eta}_{0,m}}=(W^k_h\setminus  J)^{\xi+\eta}_{y^{k,\xi+\eta}_{0,m}}\cup\bigcup_{\substack{(i,j)\in\mathcal{J}^k\\j-i=m}}(I^{k,\xi+\eta}_{i,j}\setminus  J^{\xi+\eta}_{y^{k,\xi+\eta}_{0,m}}).
\end{equation}

For every $k\in\N$ we now define $V_\e$ , $W^k_h$,  $f^k_h$, and $g$  as in \eqref{def Vxi}-\eqref{gxilemma}, with $\xi$ replaced by $\xi+\eta$.
Arguing as in the first part of the proof, we obtain that $(f^k_h)$ converges to $0$ in $L^1(I\times \R)$ as $k\to+\infty.$ Hence, recalling \eqref{Jk new}, we may apply Lemma \ref{lemma: Riemann 2}, with $\N$ replaced by $H$, to obtain  an infinite set $K\subset H \subset \N$ and a Borel set $U_\e\subset V_\e\subset U$, with $\Lb^2(U\setminus U_\e)\leq c \e$ 
    for a constant $c>0$ depending only on $\xi,\eta$, $\alpha$, and $\beta$, such that 
\begin{equation}\label{2404}\lim_{\substack{k\to+\infty\\k\in K}}\frac{1}{k}\sum_{m\in\mathcal{I}^{k}(z_1)}\!\!\!\!\!|Du^{\xi+\eta}_{\!y_{0,m}^{k,\xi+\eta}}|((W^k_h\setminus  J)^{\xi+\eta}_{y_{0,m}^{k,\xi+\eta}})
= \lim_{\substack{k\to+\infty\\k\in{ K}}}\frac{ 1}{k}\sum_{m\in\mathcal{I}^{k}{(z_1)}}f^{k}_h(z_1+\tfrac{m}{k},z_2)=0
\end{equation}
uniformly for $\omega=z_1\eta+z_2(\xi+\eta)\in U_\e$. Recalling the equality $y^{k,\xi+\eta}_{i,j}=y^{k,\xi+\eta}_{0,m}$ for $j-i=m$, from \eqref{Jk aux}  it follows that 
\begin{equation*}
|Du^{\xi+\eta}_{y^{k,\xi+\eta}_{0,m}}((U\setminus  J)^{\xi+\eta}_{y^{k,\xi+\eta}_{0,m}})-\sum_{\substack{(i,j)\in\mathcal{J}^k\\j-i=m}} Du^{\xi+\eta}_{y^{k,\xi+\eta}_{0,m}}(I^{k,\xi+\eta}_{i,j}\setminus  J^{\xi+\eta}_{y^{k,\xi+\eta}_{0,m}})|\leq |Du^{\xi+\eta}_{y^{k,\xi+\eta}_{0,m}}|((W^k_h\setminus  J)^{\xi+\eta}_{y_{0,m}^{k,\xi+\eta}}).
\end{equation*}
Since by \eqref{proiezione con m} and \eqref{def indici in proof} we have
\begin{equation*}
\sum_{m\in\mathcal{I}^{k}(z_1)}\sum_{\substack{(i,j)\in\mathcal{J}^k\\j-i=m}}=\sum_{(i,j)\in\mathcal{J}^k},
\end{equation*}
combining \eqref{uniform in proof} and \eqref{2404} with the previous inequality, we obtain \eqref{eq:riemannXi1} for $\zeta=\xi+\eta$. 
The proof for $\zeta=\xi-\eta$ is similar.
\end{proof}
\noindent {\it Proof of Theorem \ref{prop:mainingredient} (continuation).} By Lemma \ref{fundamental lemma 2}  we may choose a sequence $(\omega_k)_k\subset U$ and an infinite set $K\subset\N$ such that  for every $\zeta\in\{\xi,\eta,\xi+\eta,\xi-\eta\}$ conditions (a) and (b) of Lemma \ref{fundamental lemma 1} hold and 
\begin{equation}\notag 
\label{riemann finto}\lim_{\substack{k\to+\infty \\ k\in K}}\frac{C_{\xi,\eta}}{k}\sum_{(i,j)\in\mathcal{J}^{k}}Du^{\zeta}_{\!y_{i,j}^{k,\zeta}}(I^{k,\zeta}_{i,j}\setminus  J^{\zeta}_{y_{i,j}^{k,\zeta}})=\sigma^\zeta_u(U),
\end{equation}
for every $\mathcal{J}^k$ such that $\mathcal{J}^k_3\subset \mathcal{J}^k\subset \mathcal{J}^k_1$ (see \eqref{5.18}), where the projections $y^{k,\zeta}_{i,j}$ are defined taking $\omega=\omega_k$ in \eqref{def projections}. 

To present the technique we will employ in the sequel, let us assume for a moment that for every $(i,j)\in\mathcal{J}^{k}_1$ the segments $[x_{i,j}^k,x_{i+1,j}^k]$, $[x_{i,j}^k,x_{i,j+1}^k]$, $[x_{i,j}^k,x_{i+1,j+1}^k]$,  $[x^k_{i+1,j},x^{k}_{i+1,j+1}]$,  $[x^k_{i,j+1},x^{k}_{i+1,j+1}]$,  and   $[x^k_{i+1,j},x_{i,j+1}^k]$ do not intersect the set $ J$, which implies $I^{k,\zeta}_{i,j}\setminus  J^{\zeta}_{y_{i,j}^{k,\zeta}}=I^{k,\zeta}_{i,j}$, hence
\begin{gather} 
\label{riemann finto 2}\lim_{\substack{k\to+\infty \\ k\in K}}\frac{C_{\xi,\eta}}{k}\sum_{(i,j){\in\mathcal{J}^{k}}}Du^{\zeta}_{\!y_{i,j}^{k,\zeta}}(I^{k,\zeta}_{i,j})=\sigma^\zeta_u(U),
\end{gather}
whenever $\mathcal{J}^k_3\subset \mathcal{J}^k\subset \mathcal{J}^k_1$.
 By (a) and (b) of Lemma \ref{fundamental lemma 1} and \eqref{5.17bis}, for every $k\in\N$ and $(i,j)\in\mathcal{J}^k_1$ we have that $u^\zeta_{\!y^{k,\zeta}_{i,j}}\in BV(U^\zeta_{\!y_{i,j}^{k,\zeta}})$ and  
\begin{gather}\label{FTC}
Du^\zeta_{y^{k,\zeta}_{i,j}}(I^{k,\zeta}_{i,j})=(u(x^{k,\zeta}_{i,j}+\tfrac{1}{k}\zeta)-u(x^{k,\zeta}_{i,j}))\cdot\zeta\quad \text{ for every $\zeta\in \{\xi,\eta,\xi+\eta,\xi-\eta\}$.}
\end{gather}
Let $e_1:=(1,0)$ and $e_2:=(0,1)$ and set  \begin{equation}\notag\label{hat check jk}\mathcal{J}^k_2+e_1:=\{(i+1,j):(i,j)\in \mathcal{J}^k_2\}\quad \text{ and }\quad  \mathcal{J}_2^k+e_2:=\{(i,j+1):(i,j)\in \mathcal{J}^k_2\},
 \end{equation} and observe that  by \eqref{Jk new} we have 
$\mathcal{J}^k_3\subset \mathcal{J}^k_2+e_1\subset \mathcal{J}^k_1$ and $\mathcal{J}^k_3\subset \mathcal{J}^k_2+e_2\subset \mathcal{J}^k_1$.
From \eqref{def xkij} and  \eqref{FTC} it follows that 
\begin{equation*}
\begin{gathered}
\sum_{(i,j){\in\mathcal{J}^{k}_2}}Du^\xi_{y^{k,\xi}_{i,j}}(I^{k,\xi}_{i,j})=\sum_{(i,j)\in \mathcal{J}^{k}_2}\big(u(x_{i+1,j}^{k})-u(x_{i,j}^{k})\big)\cdot\xi,\\
\sum_{(i,j){\in{\mathcal{J}}^{k}_2+e_2}}Du^\xi_{y^{k,\xi}_{i,j}}(I^{k,\xi}_{i,j})=\sum_{(i,j)\in {\mathcal{J}}^{k}_2+e_2}\big(u(x_{i+1,j}^{k})-u(x_{i,j}^{k})\big)\cdot\xi,\\
\sum_{(i,j)\in \mathcal{J}^{k}_2}Du^\eta_{y^{k,\eta}_{i,j}}(I^{k,\eta}_{i,j})=\sum_{(i,j)\in \mathcal{J}^{k}_2}\big(u(x_{i,j+1}^{k})-u(x_{i,j}^{k})\big)\cdot\eta,\\
\sum_{(i,j)\in {\mathcal{J}}^{k}_2+e_1}Du^\eta_{y^{k,\eta}_{i,j}}(I^{k,\eta}_{i,j})=\sum_{(i,j)\in {\mathcal{J}}^{k}_2+e_1}\big(u(x_{i,j+1}^{k})-u(x_{i,j}^{k})\big)\cdot\eta,\\
\sum_{(i,j)\in \mathcal{J}^{k}_2}Du^{\xi+\eta}_{y^{k,\xi+\eta}_{i,j}}(I^{k,\xi+\eta}_{i,j})=\sum_{(i,j)\in \mathcal{J}^{k}_2}\big(u(x_{i+1,j+1}^{k})-u(x_{i,j}^{k})\big)\cdot(\xi+\eta),\\
\sum_{(i,j)\in \mathcal{J}^{k}_2}Du^{\xi-\eta}_{y^{k,\xi-\eta}_{i,j}}(I^{k,\xi-\eta}_{i,j})=\sum_{(i,j)\in \mathcal{J}^{k}_2}\big(u(x_{i+1,j}^{k})-u(x_{i,j+1}^{k})\big)\cdot(\xi-\eta).
\end{gathered}
\end{equation*}
Thus,
\begin{equation*}
\begin{gathered}
\sum_{(i,j){\in\mathcal{J}^{k}_2}}Du^{\xi+\eta}_{y^{k,\xi+\eta}_{i,j}}(I^{k,\xi+\eta}_{i,j})+ \sum_{(i,j)\in \mathcal{J}^{k}_2}Du^{\xi-\eta}_{y^{k,\xi-\eta}_{i,j}}(I^{k,\xi-\eta}_{i,j})
  \\=\sum_{(i,j)\in \mathcal{J}^{k}_2}\!\!\big(u(x^k_{i+1,j+1})-u(x^k_{i,j})\big)\cdot(\xi+\eta)+\sum_{(i,j)\in \mathcal{J}^{k}_2}\big(u(x^{k}_{i+1,j})-u(x^k_{i,j+1})\big)\cdot(\xi-\eta)\end{gathered}
  \end{equation*}
  \begin{gather}\label{uno dei conti finali}=
\sum_{(i,j)\in \mathcal{J}^{k}_2}(u(x^k_{i+1,j+1})-u(x^k_{i,j+1}))\cdot \xi+\sum_{(i,j)\in \mathcal{J}^{k}_2} (u(x^k_{i+1,j})-u(x^k_{i,j}))\cdot \xi
  \\\nonumber 
  \,\,\,\,+\sum_{(i,j)\in \mathcal{J}^{k}_2}(u(x^k_{i+1,j+1})-u(x^{k}_{i+1,j}))\cdot\eta +\sum_{(i,j)\in \mathcal{J}^{k}_2}(u(x^k_{i,j+1})-u(x^{k}_{i,j}))\cdot\eta
  \\\nonumber =\sum_{(i,j)\in \mathcal{J}^{k}_2} \!\!\!\!\!\ Du^{\xi}_{y^{k,\xi}_{i,j}}(I^{k,\xi}_{i,j})+\!\!\!\!\!\!\sum_{(i,j)\in {\mathcal{J}}^{k}_2+e_2} \!\!\!\!\!\!\!Du^{\xi}_{y^{k,\xi}_{i,j}}(I^{k,\xi}_{i,j})+\!\!\!\!\!\!\!\!\sum_{(i,j)\in {\mathcal{J}}^{k}_2+e_1}\!\!\!\!\!\!\!\! Du^{\eta}_{y^{k,\eta}_{i,j}}(I^{k,\eta}_{i,j})+\!\!\!\sum_{(i,j)\in \mathcal{J}^{k}_2}\!\!\!\! Du^{\eta}_{y^{k,\eta}_{i,j}}(I^{k,\eta}_{i,j}).
\end{gather}
Thanks to \eqref{riemann finto 2} we obtain
    \begin{gather*}
    \sigma_{u}^{\xi+\eta}(U)+\sigma_u^{\xi-\eta}(U)=2\sigma^{\xi}_u(U)+2\sigma^{\eta}_u(U),
\end{gather*}
which implies that $\xi\mapsto \sigma^{\xi}_u(U)$ is quadratic.

Unfortunately, the hypothesis that for every $k\in\N$ and $(i,j)\in\mathcal{J}^k_1$   the every one of the six  segments  $[x_{i,j}^k,x_{i+1,j}^k]$, $[x_{i,j}^k,x_{i,j+1}^k]$, $[x_{i,j}^k,x_{i+1,j+1}^k]$,  $[x^k_{i+1,j},x^{k}_{i+1,j+1}]$,  $[x^k_{i,j+1},x^{k}_{i+1,j+1}]$,  and   $[x_{i+1,j}^k,x_{i,j+1}^k]$ do not intersect the set $ J$ is almost never satisfied. Therefore, for every $k\in\N$ we introduce the set $\mathcal{G}^k\subset \Z^2$ of good indices, defined as
\begin{equation}\label{def good}
\begin{gathered}\hspace{-0.2 cm}\mathcal{G}^k:=\big\{(i,j){\in\mathcal{J}^k_2}: \text{ none of the segments } [x_{i,j}^k,x_{i+1,j}^k],\, [x_{i,j}^k,x_{i,j+1}^k]\,,\,[x_{i,j}^k,x_{i+1,j+1}^k],\\ [x^k_{i+1,j},x^{k}_{i+1,j+1}],\, [x^k_{i,j+1},x^{k}_{i+1,j+1}],\,  [x_{i+1,j}^k,x_{i,j+1}^k] \text{ intersects }   J\big\},
\end{gathered}
\end{equation}
Note that by \eqref{5.17bis} we have\begin{equation}
\begin{gathered}\label{star 1}I^{k,\zeta}_{i,j}\cap  J^{\zeta}_{y^{k,\zeta}_{i,j}}=\emptyset\quad \text{ for every }(i,j)\in\mathcal{G}^k,\\I^{k,\xi}_{i,j}\cap  J^{\xi}_{y^{k,\xi}_{i,j}}=\emptyset\quad \text{ for every }(i,j)\in{\mathcal{G}}^k+e_2,\\
I^{k,\eta}_{i,j}\cap  J^{\eta}_{y^{k,\eta}_{i,j}}=\emptyset\quad \text{ for every }(i,j)\in{\mathcal{G}}^k+e_1,
\end{gathered}
\end{equation}
where 
\begin{equation*}
     {\mathcal{G}}^k+e_1:=\{(i+1,j)\in\Z^2:(i,j)\in \mathcal{G}^k\}\quad \text{ and }\quad {\mathcal{G}}^k+e_2:=\{(i,j+1):(i,j)\in \mathcal{G}^k\}.
\end{equation*}
To prove the result in the general case, in the next section (see Theorem \ref{vanishing indices}) we shall show that the sequence $(\omega_k)_k\subset U$ and the infinite set $K\subset \N$ can be chosen in such a way that  conditions (a) and (b) of Lemma \ref{fundamental lemma 1} hold and, in addition, for every $\zeta\in\{\xi,\eta,\xi+\eta,\xi-\eta\}$,
\begin{equation}
\begin{gathered} \label{g angelici}
\lim_{\substack{k\to+\infty \\ k\in K}}\frac{C_{\xi,\eta}}{k}\sum_{(i,j)\in\mathcal{G}^k}Du^{\zeta}_{y_{i,j}^{k,\zeta}}(I^{k,\zeta}_{i,j})=\sigma^{\zeta}_u(U),\\
\lim_{\substack{k\to+\infty \\ k\in K}}\frac{C_{\xi,\eta}}{k}\sum_{(i,j)\in{\mathcal{G}}^k+e_2}Du^{\xi}_{y_{i,j}^{k,\xi}}(I^{k,\xi}_{i,j})=\sigma^{\xi}_u(U),\\
\lim_{\substack{k\to+\infty \\ k\in K}}\frac{C_{\xi,\eta}}{k}\sum_{(i,j)\in{\mathcal{G}}^k+e_1}Du^{\eta}_{y_{i,j}^{k,\eta}}(I^{k,\eta}_{i,j})=\sigma^{\eta}_u(U).
\end{gathered}
\end{equation}

Assuming that these equalities hold, we now conclude the proof in the general case.  Observing that \eqref{FTC} still holds for $(i,j)\in\mathcal{G}^k$, and also for $(i,j)\in \mathcal{G}^k+e_1$ when $\zeta=\eta$ and  for $(i,j)\in \mathcal{G}^k+e_2$ when $\zeta=\xi$, repeating the arguments that led to \eqref{uno dei conti finali} we obtain 
\begin{gather*}
\sum_{(i,j){\in\mathcal{G}^k}}Du^{\xi+\eta}_{y^{k,\xi+\eta}_{i,j}}(I^{k,\xi+\eta}_{i,j})+\sum_{(i,j)\in\mathcal{G}^k}Du^{\xi-\eta}_{y^{k,\xi-\eta}_{i,j}}(I^{k,\xi-\eta}_{i,j})\\=
  \sum_{(i,j)\in\mathcal{G}^k} Du^{\xi}_{y^{k,\xi}_{i,j}}(I^{k,\xi}_{i,j})+\sum_{(i,j)\in{\mathcal{G}}^k+e_2} Du^{\xi}_{y^{k,\xi}_{i,j}}(I^{k,\xi}_{i,j})\\
+\sum_{(i,j)\in{{\mathcal{G}}^k}+e_1} Du^{\eta}_{y^{k,\eta}_{i,j}}(I^{k,\eta}_{i,j})+\sum_{(i,j)\in{\mathcal{G}^k}} Du^{\eta}_{y^{k,\eta}_{i,j}}(I^{k,\eta}_{i,j}).
\end{gather*}
Multiplying the previous equality by $C_{\xi,\eta}/k$ and using \eqref{g angelici}  we obtain \eqref{eq:parallelogram identity}. This concludes the proof.
\end{proof}

\section{Conclusion of the proof in dimension \texorpdfstring{$d=2$}{}}\label{section conclusion}

In this section we prove a technical result, which concludes the proof of Theorem \ref{prop:mainingredient}. Throughout this section $u$, $\xi,\,\eta$, and $U$ are as in Section \ref{section main} and we use the notation introduced in the proof of Theorem \ref{prop:mainingredient}. In particular, we recall that $\mathcal{G}^k$ is defined by \eqref{def good}. Before stating the main result of this section
we introduce the set of  bad indices $\mathcal{B}^k\subset \Z^2$, defined as 
\begin{equation}
    \begin{gathered}\label{def bad}
        \mathcal{B}^k:=\big\{(i,j)\in\mathcal{J}^k_2: \text{ one of the segments } [x_{i,j}^k,x_{i+1,j}^k],\, [x_{i,j}^k,x_{i,j+1}^k]\,,[x_{i,j}^k,x_{i+1,j+1}^k],\\  [x^k_{i+1,j},x^{k}_{i+1,j+1}],\, [x^k_{i,j+1},x^{k}_{i+1,j+1}],\,  [x_{i+1,j}^k,x_{i,j+1}^k] \text{ intersects }   J\big\},
    \end{gathered}
\end{equation}

\begin{theorem}\label{vanishing indices} There exist an infinite set $K\subset \N$ and a sequence $(\omega_k)_{k\in\N}\subset U$ such that  for  every $\zeta\in\{\xi,\eta,\xi+\eta,\xi-\eta\}$ conditions $($a$)$, $($b$)$ of Lemma \ref{fundamental lemma 1} and \eqref{eq:riemannXi1} of Lemma \ref{fundamental lemma 2} hold and the following equalities are satisfied: \begin{gather}\label{da scartare nuova}
\lim_{\substack{k\to+\infty \\ k\in K}}\frac{C_{\xi,\eta}}{k}\sum_{(i,j)\in\mathcal{G}^k}Du^{\zeta}_{y_{i,j}^{k,\zeta}}(I^{k,\zeta}_{i,j})=\sigma^{\zeta}_u(U),\\
\label{da scartare nuova 1}
\lim_{\substack{k\to+\infty \\ k\in K}}\frac{C_{\xi,\eta}}{k}\sum_{(i,j)\in{\mathcal{G}}_2^k+e_2}Du^{\xi}_{y_{i,j}^{k,\xi}}(I^{k,\xi}_{i,j})=\sigma^{\xi}_u(U),\\\label{da scartare nuova 2}
\lim_{\substack{k\to+\infty \\ k\in K}}\frac{C_{\xi,\eta}}{k}\sum_{(i,j)\in{\mathcal{G}}_2^k+e_1}Du^{\eta}_{y_{i,j}^{k,\eta}}(I^{k,\eta}_{i,j})=\sigma^{\eta}_u(U),
\\
  \label{da scartare 1}\lim_{\substack{k\to+\infty \\ k\in K}}\,\,\,\frac{1}{k}\sum_{(i,j)\in \mathcal{B}^{k}}
|Du^{\zeta}_{y^{k,\zeta}_{i,j}}|(I^{k,\zeta}_{i,j}\setminus  J^{\zeta}_{y^{k,\zeta}_{i,j}})=0,\\
\label{da scartare 3}
\lim_{\substack{k\to+\infty \\ k\in K}}\,\,\,\frac{1}{k}\sum_{(i,j)\in {\mathcal{B}}^{k}+e_2}
|Du^{\xi}_{y^{k,\xi}_{i,j}}|( I^{k,\xi}_{i,j}\setminus  J^{\xi}_{y^{k,\xi}_{i,j}})=0,\\\label{da scartare 4}
\lim_{\substack{k\to+\infty \\ k\in K}}\,\,\,\frac{1}{k}\sum_{(i,j)\in {\mathcal{B}}^{k}+e_1}
|Du^{\eta}_{y^{k,\eta}_{i,j}}|(I^{k,\eta}_{i,j}\setminus  J^{\eta}_{y^{k,\eta}_{i,j}})=0,
     \end{gather} 
     where the points $y^{k,\zeta}_{i,j}$ introduced in \eqref{1 parameter family} are defined by taking $\omega=\omega_k$.   
 \end{theorem}

The crucial part in the proof of this result is proving \eqref{da scartare 1}-\eqref{da scartare 4}, as \eqref{da scartare nuova}-\eqref{da scartare nuova 2} can then be obtained from \eqref{eq:riemannXi1} by difference, using \eqref{star 1}. We only prove \eqref{da scartare 1}, as the proof of \eqref{da scartare 3} and \eqref{da scartare 4} are similar. This  proof is extremely technical. The arguments we are going to use require  some additional notation.

Given $\zeta\in\{\xi,\eta,\xi-\eta,\xi+\eta\}$, we set  
\begin{equation}\label{eq: def zeta check}
\text{ $\bar{\zeta}:=\xi$ \,\,if \,\,$\zeta\in\{\eta,\xi+\eta\}$ }\quad  \text{ and \quad   $ \bar{\zeta}:=\eta$ \,\,if \,\,$\zeta\in\{\xi,\xi-\eta\}$}.
    \end{equation} 
We observe that $\zeta$ and $\bar{\zeta}$ are linearly independent.
For $y\in\R^2$, $k\in\N$, $j\in\Z$, and $\zeta\in\{\xi,\eta,\xi+\eta,\xi-\eta\}$ let $t^{k,\zeta}_{j}(y)$  be the real number characterised by
\begin{equation}\label{def intersections coord}
    y+t^{k,\zeta}_{j}(y)\zeta \in \{\omega+\tfrac{j}{k}\zeta +s \bar{\zeta}:s\in\R\}.
\end{equation}
 We also set 
\begin{equation}\label{def intersections}
\begin{gathered}
    \displaystyle 
    x^{k,\zeta}_{j}(y):= y+t^{k,\zeta}_{j}(y)\zeta.
\end{gathered}
\end{equation}
In other words, $x^{k,\zeta}_j(y)$ is the intersection of the straight lines $\{y+t\zeta:t\in\R\}$ and $\{\omega+\tfrac{j}{k}\zeta +s \bar{\zeta}:s\in\R\}$.
Note that the family of straight lines  $(\{\omega+\tfrac{j}{k}\zeta +s \bar{\zeta}:s\in\R\})_{j\in\Z}$ coincides with the family of  the straight lines parallel to $ \bar{\zeta}$ passing through one of the points 
$x^k_{i,j}$ for $i,j\in\Z$.

Note that  for every $j\in\Z$ and $t\in\R$ we have 
\begin{equation}\notag\label{tkzetaj+1}
t^{k,\zeta}_{j+1}(y)=t^{k,\zeta}_{j}(y)+\tfrac{1}{k}\quad\hbox{and}\quad t^{k,\zeta}_{j}(y+t\zeta)=t^{k,\zeta}_{j}(y)-t,
\end{equation}
which give
\begin{equation}\label{xkzeta y+tzeta}
x^{k,\zeta}_{j+1}(y)=x^{k,\zeta}_{j}(y)+\tfrac{1}{k}\zeta
\quad\hbox{and}\quad
x^{k,\zeta}_j(y+t\zeta)=x^{k,\zeta}_j(y).
\end{equation}
Moreover,
\begin{equation}\label{xkzeta y and tkzeta}
[x^{k,\zeta}_j(y),x^{k,\zeta}_{j+1}(y))=\{y+t\zeta:t\in [t^{k,\zeta}_j(y),t^{k,\zeta}_{j+1}(y))\}.
\end{equation}
Therefore for every $y\in\R^2$ 
each straight line $\{y+t\zeta:t\in\R\}$ can be written as disjoint union of segments in the following way
\begin{equation}\label{decomposition of a line}
   \{y+t\zeta:t\in\R\}= \bigcup_{j\in\Z}[x^{k,\zeta}_j(y),x^{k,\zeta}_{j+1}(y)).
\end{equation}

We need to introduce some sets  which are useful to establish \eqref{da scartare 1} and whose definition requires some additional notation. 
Let us fix $\zeta\in\{\xi,\eta,\xi+\eta,\xi-\eta\}$ and $k\in\N$. In view of \eqref{decomposition of a line}, for every $x\in\R^2$ there exists a unique $j\in\Z$ such that 
\begin{equation}\label{def zeta}
    x\in [x^{k,\zeta}_{j}(x),x^{k,\zeta}_{j+1}(x)).
\end{equation}
We define the map $z^{k,\zeta}\colon \R^2\to \R^2$ (see Figure \ref{figure parallelogram}) as 
\begin{equation}\label{true def z}
    z^{k,\zeta}(x)=x_j^{k,\zeta}(x),
\end{equation}
where $j\in\Z$ is the unique index such that \eqref{def zeta} holds.
\begin{figure}[h]\hspace*{-0.5cm}
\begin{tikzpicture}[scale=0.7]
\def\parallelogram{ (0,0)--(4,2) -- (1,3.5) -- (-3,1.5)-- cycle}   
\draw[thick] \parallelogram;

\draw[thick, dotted](-3,1.5)--(-5,2.5);
\draw[thick,dotted] (-2,1)--(1,2.5);
\draw[fill=black](1,2.5)  circle (2.0 pt) node [above] {$x$};
\draw[thick, dotted](-0,0)--(2,-1);
\draw[fill=black](-2,1)  circle (2.0 pt) node [below, yshift=-0.1 cm] {$z^{k,\xi}(x)$};
\draw  (-5,2.5) node [below, yshift=0 cm,yshift=-0.3 cm, xshift=-0.7 cm]{$\{x^k_{i,j}+s\eta:s\in\R\}$};
\draw[thick, dotted] (4,2)--(6,1);
\draw  (5,1.5) node [above,xshift=0.4 cm, yshift=0.1 cm,xshift=1.0 cm]{$\{x^k_{i+1,j}+s\eta:s\in\R\}$};
\draw[thick,dotted]{(1,3.5)--(0,4)};
\end{tikzpicture}
\caption{ \label{figure parallelogram} The map $x\mapsto z^{k,\xi}(x)$.}
\end{figure}
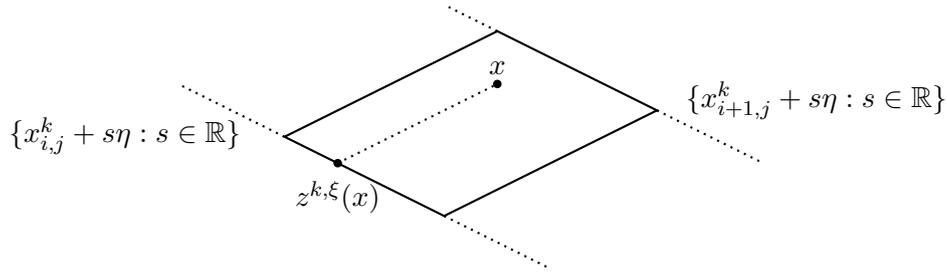

By \eqref{xkzeta y+tzeta} and \eqref{xkzeta y and tkzeta}  for every $y\in\R^2$
\begin{equation}\label{star}
z^{k,\zeta}(y+t\zeta)=x^{k,\zeta}_j(y)\quad\hbox{for every }t\in  [t^{k,\zeta}_j(y),t^{k,\zeta}_{j+1}(y)).
\end{equation}

Geometrically (see Figure \ref{figure parallelogram}), $z^{k,\zeta}(x)$ is given by $x-t\zeta$ where $t\geq 0$ is the smallest number such that  $x-t\zeta$ belongs to 
one  of the straight lines parallel to $\bar\zeta$ passing through one of the points $x^k_{i,j}$.
By this geometric characterisation, $z^{k,\zeta}$ is a Borel function.

We consider the union $S$ of the sides and the diagonals of the parallelogram of vertices at $0$, $\xi,\xi+\eta,\eta$, i.e.,   \begin{equation}\label{union of segments}
    S:=[0,\xi]\cup[0,\eta]\cup[0,\xi+\eta]\cup[\eta,\xi]\cup [\eta,\xi+\eta]\cup[\xi,\xi+\eta].
    \end{equation}
 For $k\in\N$  and $\zeta\in \{\xi,\eta,\xi+\eta,\xi-\eta\}$ we introduce the set
\begin{gather}\label{def Akzeta}
    E^{k,\zeta}:=\{x\in \R^2:(z^{k,\zeta}(x)+\tfrac{1}{k}S)\cap  J\neq\emptyset\}.
\end{gather} 
 For $y\in\R^2$ we define  
 \begin{equation}
 \begin{gathered}\label{Ikzy Nkzy}
\mathcal{I}^{k,\zeta}(y):=\{i\in\Z:(x^{k,\zeta}_i(y)+\tfrac1kS)\cap  J\neq \emptyset\},
\quad
\mathcal{N}^{k,\zeta}(y):=\hzero(\mathcal{I}^{k,\zeta}(y)),
 \\
 E^{k,\zeta}(y):=\bigcup_{i\in \mathcal{I}^{k,\zeta}(y)}[t^{k,\zeta}_i(y),t^{k,\zeta}_{i+1}(y)).
 \end{gathered}
 \end{equation}
 By \eqref{star} we have the equality \begin{equation}\label{eq:definition notation slicing}E^{k,\zeta}(y)=(E^{k,\zeta})^\zeta_y \quad \text{for every $y\in \R^2$}
 \end{equation}  and by \eqref{xkzeta y+tzeta} we have
\begin{equation}\label{Nky+tzeta}
 \mathcal{I}^{k,\zeta}(y+t\zeta)=\mathcal{I}^{k,\zeta}(y)\quad\hbox{and}\quad \mathcal{N}^{k,\zeta}(y+t\zeta)=\mathcal{N}^{k,\zeta}(y)
 \quad\hbox{for every }t\in\R.
\end{equation}
 
  Let $x\in\R^2$ and let $j\in\Z$ be the unique index such that \eqref{def zeta} holds.
  By \eqref{xkzeta y+tzeta} for every $i\in\Z$ we have
  \begin{equation*}
  x^{k,\zeta}_i(x) = x^{k,\zeta}_i(x+\tfrac{i-j}k\zeta) = z^{k,\zeta}(x+\tfrac{i-j}k\zeta),
  \end{equation*}
  where the last equality follows from \eqref{true def z}, since $x+\tfrac{i-j}k\zeta\in [x^{k,\zeta}_i(x+\tfrac{i-j}k\zeta),x^{k,\zeta}_{i+1}(x+\tfrac{i-j}k\zeta))$ by \eqref{xkzeta y+tzeta}.
  Recalling the definition of $E^{k,\zeta}$ in \eqref{def Akzeta}, the equalities above imply that
\begin{equation}\notag\label{new Ikzeta}
  \mathcal{I}^{k,\zeta}(x)=\{i\in\Z:(z^{k,\zeta}(x+\tfrac{i-j}k\zeta)+\tfrac1kS)\cap J\neq \emptyset\}
  =\{i\in\Z:x+\tfrac{i-j}k\zeta\in  E^{k,\zeta}\},
 \end{equation}
 which gives
\begin{equation}\label{new Nkzeta}
\mathcal{N}^{k,\zeta}(x)=\hzero(
 \{i\in\Z:x+\tfrac{i}k\zeta\in E^{k,\zeta}\}).
 \end{equation}
 
 For $\zeta\in \{\xi,\eta,\xi+\eta,\xi-\eta\}$, $k,m\in\N$, and $y\in\R^2$ we set 
\begin{gather}\label{def hatE checkE}
 \hat E^{k,\zeta}_{m}:=\{x\in  E^{k,\zeta}: \mathcal{N}^{k,\zeta}(x)\le m\},\quad  \Check E^{k,\zeta}_{m}:=\{x\in  E^{k,\zeta}: \mathcal{N}^{k,\zeta}(x)> m\},
\\\label{6.23 bis}
 \hat E^{k,\zeta}_{m}(y):= \begin{cases}
        E^{k,\zeta}(y)&\text{if $\mathcal{N}^{k,\zeta}(y)\leq m$},\\
        \emptyset &\text{otherwise},
    \end{cases}\quad
 \Check E^{k,\zeta}_{m}(y):= \begin{cases}
        E^{k,\zeta}(y)&\text{if $\mathcal{N}^{k,\zeta}(y)>m$},\\
        \emptyset &\text{otherwise}.
    \end{cases}
\end{gather}

By \eqref{eq:definition notation slicing} and \eqref{Nky+tzeta}  we have 
\begin{equation}\label{slice of hats}\hat E^{k,\zeta}_{m}(y)=(\hat E^{k,\zeta}_{m})^\zeta_y\,\,\, \text{ and } \,\,\,\Check E^{k,\zeta}_{m}(y)=(\Check E^{k,\zeta}_{m})^\zeta_y.\end{equation}
 
All these sets are  Borel  measurable as the following lemma shows.
\begin{lemma}\label{lemma measurable R2}
    The sets $E^{k,\zeta}$, $\hat E^{k,\zeta}_m$, and $\check E^{k,\zeta}_m$ are Borel measurable. Moreover, the function $\mathcal{N}^{k,\zeta}$ is Borel measurable on $\R^2$.
\end{lemma}
\begin{proof}
    For every set $B\subset \R^2$ we define
    \begin{gather*}
        E^{k,\zeta}_{B}:=\{x\in \R^2:( z^{k,\zeta}(x)+\tfrac{1}{k}S)\cap B\neq \emptyset\},\\
         F_B:=\{z\in \R^2: (z+\tfrac{1}{k}S)\cap B\neq \emptyset\}.
    \end{gather*}
    We begin by proving that for a compact set $K\subset \R^2$ the set $E^{k,\zeta}_K$ 
    is Borel measurable. To this aim we note that  the set $F_K$ is closed and that $ E_{K}^{k,\zeta}=\{x\in \R^2:z^{k,\zeta}(x)\in F_K\}$. Recalling that $z^{k,\zeta}$ is Borel measurable, we conclude that  $E^{k,\zeta}_K$ is Borel measurable.
    By \eqref{def Jhat} we have $ J=\bigcup_{n\in\N}K_n$, where $K_n$ are  compact sets. This gives that $E^{k,\zeta}=\bigcup_{n\in \N}E^{k,\zeta}_{K_n}$. Since the sets $E^{k,\zeta}_{K_n}$ are Borel measurable, so is $E^{k,\zeta}$. 

     To prove that $\hat E^{k,\zeta}_m$ Borel measurable, we observe that by \eqref{new Nkzeta}
      and \eqref{def hatE checkE} a point $x$ belongs to $\hat E^{k,\zeta}_m$ if and only if  the number of indices $i\in\Z$ such that $x\in E^{k,\zeta}-\tfrac{i}{k}\zeta$ is less than or equal to $m$. This implies that, setting $E_i^{k,\zeta}:=E^{k,\zeta}-\tfrac{i}{k}\zeta$,  we have
     \begin{equation*}
     \textstyle \hat E^{k,\zeta}_m=\big\{x\in \R^2: \sum_{i}\chi_{E^{k,\zeta}_i}(x)\leq m\big\},
     \end{equation*}
     where $\chi_{E^{k,\zeta}_i}$ is the characteristic function of $E^{k,\zeta}_i$.
     Since the sets $E^{k,\zeta}_i$ are Borel measurable, we deduce that $\hat E^{k,\zeta}_m$ is Borel measurable. The Borel measurability of $\check E^{k,\zeta}_m$ follows from the equality $\check E^{k,\zeta}_m=E^{k,\zeta}\setminus \hat E^{k,\zeta}_m$.

     To prove that the function $\mathcal{N}^{k,\zeta}$ is Borel measurable, it is enough to observe that by \eqref{new Nkzeta}
we have $\mathcal{N}^{k,\zeta}=\sum_{i\in\Z}\chi_{E^{k,\zeta}_i}$.
\end{proof}
 
\begin{remark}\label{re:dependence on omega}
 All the sets defined in \eqref{def Akzeta} and \eqref{def hatE checkE} depend non-trivially on $\omega\in \R^2$, since by \eqref{def xkij} every choice of $\omega$ determines different points $x^k_{i,j}$ and thus, by \eqref{def intersections}, different sets in \eqref{def Akzeta} and \eqref{def hatE checkE} as well. Not to overburden the notation, we do not indicate the dependence of such objects on $\omega$. The measurability issues with respect to $\omega$ will be dealt with in the Appendix.
\end{remark}

To use the sets $\hat E^{k,\zeta}_m$ and $\check E^{k,\zeta}_m$ in our estimates we need the properties proved in the following two lemmas, whose proofs are postponed.
We observe that, since $\xi$ and $\eta$ are linearly independent, given $\zeta\in\{\xi,\eta,\xi+\eta,\xi-\eta\}$, every point $\omega\in\R^2$ can be written in a unique way as \begin{equation}\label{coordinates omega}\omega=z_1 \bar{\zeta}+z_2\zeta, \end{equation}for suitable $z_1,z_2\in\R$. We set 
\begin{equation}\label{def: Izeta}
    I^\zeta:=\begin{cases}
        (0,\alpha)\quad &\text{ if }\zeta=\xi \text{ or }\zeta=\xi-\eta,\\
        (0,\beta)\quad &\text{ if }\zeta=\eta \text{ or }\zeta=\xi+\eta,\\
    \end{cases}
\end{equation}
and we observe that by \eqref{parallelogramU} 
\begin{equation}\label{proiezione di U}
    \pi^{ \bar{\zeta}}(U):=\{z_2\pi^{ \bar{\zeta}}(\zeta): z_2\in I^\zeta\}.
\end{equation}
\begin{lemma}\label{lemma smallness of A>m}
   There exist a constant $C>0$  such that for every $\zeta\in\{\xi,\eta,\xi+\eta,\xi-\eta\}$ and for every $\e>0$ there exist an infinite set $K^\zeta_\e\subset \N$ and  a Borel  set $I_\e^\zeta\subset I^\zeta$, with $\Lb^1(I^\zeta\setminus I^\zeta_\e)\leq \e$, such that
     \begin{equation} \label{claim lemma smallness}
\mathcal{H}^{1}\big(\pi^{\zeta}(\check E^{k,\zeta}_m)\big)\leq C/m
     \end{equation}
     for every $m\in\N$,  $\omega\in U^\zeta_\e:=U\cap \{z_1 \bar{\zeta}+z_2\zeta: z_1\in\R,\, z_2\in I^\zeta_\e\}$, and $k\in K^\zeta_\e$.
 \end{lemma}

 \begin{lemma}\label{lemma vanishing integral}
    Let $m\in\N$ and $\zeta\in\{\xi,\eta,\xi+\eta,\xi-\eta\}$. Then we have 
\begin{gather}\label{trascu 2}
\lim_{k\to+\infty}|Du^\zeta_y|(\hat E^{k,\zeta}_m(y)\cap U^{\zeta}_y\setminus  J^{\zeta}_y)=0\quad \text{ for }\huno \text{-a.e.\ }y\in \Pi^\zeta,
\\\label{trascurabile 1}
\lim_{k\to+\infty}\int_{\Pi^\zeta}|Du^\zeta_y|(\hat E^{k,\zeta}_m(y)\cap U^\zeta_y\setminus  J^{\zeta}_y)\, {\rm d}\mathcal{H}^{1}(y)=0.
\end{gather}
\end{lemma}

To prove Lemma \ref{lemma smallness of A>m} we need the following elementary result.
\begin{lemma}\label{lemma prob}
Let $F\subset\R $ be a finite set and let $a>0$ and $b<c$. Then 
\begin{equation*}
    \int_{\R}\mathcal{H}^0([at+b,at+c]\cap F)\,{\rm d}t= \frac{c-b}{a}\hzero(F).
\end{equation*}
\end{lemma}
\begin{proof}
    By the Fubini Theorem we have that 
    \begin{align*}
       \int_{\R}&\mathcal{H}^0([at+b,at+c]\cap F)\,{\rm d}t=  \int_{\R}\Big(\int_F\chi_{[at+b,at+c]}(s)\,{\rm d}\hzero(s)\Big){\rm d}t\\&=
       \int_{F}\Big(\int_\R\chi_{[\frac{s-c}{a},\frac{s-b}{a}]}(t)\,{\rm d}t\Big){\rm d}\hzero(s)= \frac{c-b}{a}\hzero(F).
    \end{align*}
    This concludes the proof.
\end{proof}

{\it Proof of Lemma \ref{lemma smallness of A>m}.} We observe that $\pi^{\zeta}(\check E^{k,\zeta}_m)=\{y\in \Pi^\zeta: \mathcal{N}^{k,\zeta}(y)>m\}$ by \eqref{Nky+tzeta} and \eqref{def hatE checkE}. Since $\mathcal{N}^{k,\zeta}$ is Borel measurable by Lemma \ref{lemma measurable R2}, $\pi^{\zeta}(\check E^{k,\zeta}_m)$  is a Borel set. 
By \v{C}eby\v{s}\"{e}v's inequality  we obtain 
\begin{equation} \label{claim a-m}
\mathcal{H}^{1}\big(\pi^{\zeta}(\check E^{k,\zeta}_m)\big)=\mathcal{H}^{1}\big(\{y\in \Pi^\zeta: \mathcal{N}^{k,\zeta}(y)>m\}\big)
  \leq  \frac{1}{m}\int_{\Pi^\zeta}\mathcal{N}^{k,\zeta}(y)\,{\rm d}\mathcal{H}^{1}(y)
\end{equation}
for every $\zeta\in\{\xi,\eta,\xi+\eta,\xi-\eta\}$, $k,m\in\N$, $\omega\in U$.
To conclude the proof it is enough to show that there exist a constant $C>0$ and, for every $\e>0$ and $\zeta\in\{\xi,\eta,\xi+\eta,\xi-\eta\}$, an infinite set $K^\zeta_\e\subset \N$ and a Borel set $I^\zeta_\e\subset I^\zeta$,  with $\Lb^1(I^\zeta\setminus I^\zeta_\e)\leq \e$, such that 
\begin{equation}\label{claim chebishev}
    \int_{\Pi^\zeta}\mathcal{N}^{k,\zeta}(y)\,{\rm d}\mathcal{H}^1(y)\leq C
\end{equation}
for every $\omega\in U^\zeta_\e:=U\cap \{z_1 \bar{\zeta}+z_2\zeta: z_1\in\R,\, z_2\in I^\zeta_\e\}$ and $k\in K^\zeta_\e$.

We observe that by \eqref{union of segments}  we can write $ S=S_1\cup\dots\cup S_6$,
where   
\begin{equation}\label{def: Szeta}S_1:=[0, \bar{\zeta}],
  \qquad S_2:=\begin{cases}S_1+\xi \quad \text{ if }& \zeta \in\{\xi,\xi-\eta\},\\
        S_1+\eta \quad \text{ if }& \zeta\in\{\eta,\xi+\eta\},
    \end{cases}
\end{equation}
while $S_3,$...$,$ $S_6$ are the other four segments of $S$, which are transversal to $ \bar{\zeta}$. By \eqref{Ikzy Nkzy} this implies that
\begin{equation}\notag
    \mathcal{I}^{k,\zeta}(y)=\bigcup_{h=1}^6\{j\in\Z:(x^{k,\zeta}_j(y)+S_h)\cap  J\neq \emptyset\}
\end{equation}
for every $y\in\Pi^\zeta$, hence
\begin{equation}\notag
    \mathcal{N}^{k,\zeta}(y)\leq \sum_{h=1}^6\hzero(\{j\in\Z:(x^{k,\zeta}_j(y)+S_h)\cap  J\neq \emptyset\}).
\end{equation}
Thus, 
\begin{equation}\label{integral brutto}
\int_{\Pi^\zeta} \mathcal{N}^{k,\zeta}(y)\,{\rm d}\huno(y)\leq\sum_{h=1}^6\int_{\Pi^\zeta}\mathcal{H}^0(\{j\in\Z:(x^{k,\zeta}_{j}(y)+\tfrac{1}{k}S_h)\cap  J\neq\emptyset\})\,{\rm d}\mathcal{H}^{1}(y).
\end{equation}

  Let us fix  $\zeta\in\{\xi,\eta,\xi+\eta,\xi-\eta\}$ and $\e>0$. We claim that there exist a constant $c_1>0$, independent of $\omega,$ $k$, and $\e$, an infinite set $K^\zeta_\e\subset \N$, independent of $\omega$, and  a Borel set $I^\zeta_\e\subset I^\zeta$,  with $\Lb^1(I^\zeta\setminus I^\zeta_\e)\leq \e$, such that 
\begin{equation}\label{claim con m}
\int_{\Pi^\zeta}\mathcal{H}^0(\{j\in\Z:(x^{k,\zeta}_{j}(y)+\tfrac{1}{k}S_h)\cap  J\neq\emptyset\})\,{\rm d}\mathcal{H}^{1}(y) \leq c_1,
 \end{equation}
 for every $\omega\in U^\zeta_\e$, $k\in K^\zeta_\e$, and $h\in\{1,...,6\}$. 
 To prove this claim, we first observe that for every $y\in \Pi^\zeta$  we have
 \begin{equation}\label{equazione h}
   \mathcal{H}^0(\{j\in\Z:(x^{k,\zeta}_{j}(y)+\tfrac{1}{k}S_h)\cap  J\neq\emptyset\})\leq\sum_{j\in\Z}\mathcal{H}^{0}((x^{k,\zeta}_j(y)+\tfrac{1}{k}S_h)\cap J).
 \end{equation}
 
We consider first the case $h=1$. We prove that there exists a constant $c>0$, independent of $\omega$, $k$, $\e$, and  $j$, such that 
\begin{equation}\label{new aux 1}
    \int_{\Pi^\zeta}\mathcal{H}^{0}((x^{k,\zeta}_j(y)+\tfrac{1}{k}S_1)\cap J)\, {\rm d}\huno(y)\leq \frac{c}{k}\mathcal{H}^0( J^{ \bar{\zeta}}_{y_j^{k, \bar{\zeta}}})
\end{equation}
for every $k\in\N$ and $j\in\Z$. If $\mathcal{H}^0( J^{ \bar{\zeta}}_{y_j^{k, \bar{\zeta}}})=+\infty$ there is nothing to prove. Let us fix $k$ and $j$ such that $\mathcal{H}^0( J^{ \bar{\zeta}}_{y_j^{k, \bar{\zeta}}})<+\infty.$

We parametrise $\Pi^\zeta$ by $y=s\pi^{\zeta}( \bar{\zeta})$ with $s\in\R$. 
and observe that \eqref{new aux 1} is equivalent to
\begin{equation}\label{new aux 2}
    \int_{\R}\mathcal{H}^{0}((x^{k,\zeta}_j(s\pi^\zeta( \bar{\zeta}))+\tfrac{1}{k}S_1)\cap J)\, {\rm d}s\leq \frac{c}{k}\mathcal{H}^0( J^{ \bar{\zeta}}_{y_j^{k, \bar{\zeta}}})
\end{equation} 
for a possibly different constant $c$, independent of $\omega,\, k\,$, $\e$, and $j$.
By \eqref{def projections}, \eqref{1 parameter family}, and \eqref{eq: def zeta check}  we have that 
$\pi^{ \bar{\zeta}}(\omega+\frac{j}{k}\zeta)=y^{k, \bar{\zeta}}_j$.
Therefore by the comments after \eqref{def intersections} 
\begin{equation}\label{mumble}\text{ $\{x^{k,\zeta}_{j}(y)\}=\{x^{k,\zeta}_j(s\pi^\zeta( \bar{\zeta}))\}$=$\{s\pi^\zeta( \bar{\zeta})+t\zeta:t\in\R\}$$\cap\{y^{k, \bar{\zeta}}_j+t \bar{\zeta}:t\in\R\}$}.
\end{equation} This implies that  for every $s\in\R$ there exists a unique $\tau^{k, \bar{\zeta}}_{j}(s)\in\R$ such that $y_{j}^{k, \bar{\zeta}}+\tau^{k, \bar{\zeta}}_{j}(s) \bar{\zeta}=x^{k,\zeta}_j(s\pi^\zeta( \bar{\zeta}))$, so that by \eqref{def: Szeta} we have
\begin{equation}\notag \label{new aux 3}
    \mathcal{H}^{0}((x^{k,\zeta}_j(s\pi^\zeta( \bar{\zeta}))+\tfrac{1}{k}S_1)\cap J)=\mathcal{H}^{0}([\tau^{k, \bar{\zeta}}_j(s),\tau^{k, \bar{\zeta}}_j(s)+\tfrac{1}{k}]\cap J^{ \bar{\zeta}}_{y^{k, \bar{\zeta}}_j}).
\end{equation}
 Elementary geometric arguments show  that  there exist a constant $c>0$, independent of $\omega$, $k$, $\e$, and $j$, and  a constant $d\in\R$, depending on $\omega$, $k$, and $j$, such that 
 \begin{equation}\notag
    \tau^{k, \bar{\zeta}}_j(s)=\frac{s}{c}+d \quad \text{ for every $s\in\R$}.
\end{equation}
Since by assumption $\mathcal{H}^0( J^{ \bar{\zeta}}_{y_j^{k, \bar{\zeta}}})<+\infty$, we may apply Lemma \ref{lemma prob} to obtain that 
 \begin{equation}\label{auxxe}
\int_{\R}\mathcal{H}^{0}([\tau^{k, \bar{\zeta}}_j(s),\tau^{k, \bar{\zeta}}_j(s)+\tfrac{1}{k}]\cap J^{ \bar{\zeta}}_{y^{k, \bar{\zeta}}_j})\,{\rm ds}\leq \frac{c}{k}\mathcal{H}^0( J^{ \bar{\zeta}}_{y_j^{k, \bar{\zeta}}}).
 \end{equation}
 As the constant $c$ depends only on $\xi$ and $\eta$, this proves \eqref{new aux 2}, which gives \eqref{new aux 1}. Arguing in a similar way, we prove that 
\begin{equation}\label{new aux}
    \int_{\Pi^\zeta}\mathcal{H}^{0}((x^{k,\zeta}_j(y)+\tfrac{1}{k}S_2)\cap J)\, {\rm d}\huno(y)\leq \frac{c}{k}\mathcal{H}^0( J^{ \bar{\zeta}}_{y_{j\pm 1}^{k, \bar{\zeta}}}),
\end{equation}
where the sign in $\pm 1$ depends on the specific value of $\zeta$, according  to \eqref{def: Szeta}.

By  \eqref{def projections}, \eqref{1 parameter family}, \eqref{eq: def zeta check},   and \eqref{coordinates omega}  we have the equality  $y^{k, \bar{\zeta}}_j=(z_2+\frac{j}{k})\pi^{ \bar{\zeta}}(\zeta)$ for every $j\in\Z$ and for a suitable $z_2\in\R$. Let $f\colon\R\to [0,+\infty]$ be defined by 
\begin{equation*}
    f(s):=\mathcal{H}^{0}( J^{\bar{\zeta}}_{s\pi^{ \bar{\zeta}}(\zeta)})\quad \text{for every $s\in\R$}.
\end{equation*}
Observe that, since $ J\subset U$,  by \eqref{proiezione di U} the function $f$ vanishes outside of $I^\zeta$. Since by \eqref{J1 rectifiable} and \eqref{def Jhat} we have  $\huno(J)=\huno(J^1_u\cap U)$  and $\huno(J^1_u\cap U)<+\infty$ by Remark \ref{re:RemarkAlmiTasso}, the function  $f$ is integrable by Lemma \ref{h0 slices}. Thus, we may apply Lemma \ref{lemma:Riemann sums 1} and we obtain an infinite set $K^\zeta_\e\subset \N$ and a Borel set $I^\zeta_\e\subset I^\zeta$, with $\Lb^1(I^\zeta\setminus {I}^\zeta_\e)\leq \e$, such that
\begin{equation*}
   \lim_{\substack{k\to+\infty \\ k\in K^\zeta_\e}}\frac{1}{k}\sum_{j\in\Z}\mathcal{H}^0( J^{ \bar{\zeta}}_{y_j^{k, \bar{\zeta}}})=\int_{{I}^\zeta}\mathcal{H}^0( J^{ \bar{\zeta}}_{s\pi^{ \bar{\zeta}}(\zeta)})\, {\rm d}s=\frac{1}{|\pi^{\bar\zeta}(\zeta)|}\int_{\Pi^{\bar{\zeta}}}\mathcal{H}^0( J^{ \bar{\zeta}}_{y})\, {\rm d}\huno(y)\leq \frac{1}{|\pi^{\bar\zeta}(\zeta)|}\mathcal{H}^1( J)
\end{equation*}
uniformly for $z_2\in {I}^\zeta_\e$. Hence, up to removing a finite number of elements from $K_\e^\zeta$, we may assume that
\begin{equation}\label{eq: qualcosa}
\frac{1}{k}\sum_{j\in\Z}\mathcal{H}^0( J^{ \bar{\zeta}}_{y_j^{k, \bar{\zeta}}})\leq \frac{1}{|\pi^{\bar\zeta}(\zeta)|}\huno( J)+1
\end{equation}
for every $k\in K^\zeta_\e$ and $z_2\in{I^\zeta_\e}$. Together with \eqref{equazione h}, \eqref{new aux 1}, and \eqref{new aux}, this implies \eqref{claim con m} for $h=1$ and $h=2$.

Let us now fix  $h\in\{3,\dots,6\}$. 
 For every $j\in\Z$, let $L_j$ be the strip defined by $L_j:=\{x\in \R^2:\pi^{ \bar{\zeta}}(x)\in [y^{k, \bar{\zeta}}_{j-1}, y^{k, \bar{\zeta}}_{j+1}]\}$, that is to say, the region of the plane between the straight lines $\{y^{k, \bar{\zeta}}_{j-1}+s \bar{\zeta}:s\in \R\}$ and $\{y^{k, \bar{\zeta}}_{j+1}+s \bar{\zeta}:s\in \R\}$. 
 Since $\pi^{\bar\zeta}(S_h)$ is equal to $[0,\pi^\xi(\eta)]$ or $[\pi^\xi(\eta),0]$ if $\bar\zeta=\xi$, while  $\pi^{\bar\zeta}(S_h)=[0,\pi^\eta({\xi})]$ if $\bar\zeta=\eta$, from \eqref{def projections} and  \eqref{1 parameter family} we see that in every case we have the inclusion \begin{equation*}
 \bigcup_{s\in\R}\{y^{k,\bar\zeta}_j+s\bar\zeta+\tfrac{1}{k}S_h\}\subset L_{j}\quad \text{ for every $j\in\Z$}.  
 \end{equation*} 

 Let $\tilde{\zeta}\in\R^2\setminus\{0\}$ a vector in  the direction of the segment $S_h$. Recalling that $x_{j}^{k,\zeta}(y)\in \{y_j^{k,\bar{\zeta}}+s{\bar\zeta}:s\in\R\}$ for every $y\in\Pi^\zeta$ by \eqref{mumble}, we deduce from the previous inclusion that  $(x^{k,\zeta}_j(y)+\tfrac{1}{k}S_h)\cap J\subset J\cap L_j$, hence
 \begin{equation}\label{stima fini}
  \mathcal{H}^{0}((x^{k,\zeta}_j(y)+\tfrac{1}{k}S_h)\cap J)\leq\hzero((J\cap L_{j})_{\pi^{\widetilde{\zeta}}(x_{j}^{k,{\zeta}}(y))}^{\widetilde{\zeta}})
 \end{equation}
 for every $j\in\Z$. 
 As $\zeta$ and $\bar\zeta$ are linearly independent and the same holds for $\tilde\zeta$ and $\bar\zeta$, the map  $y\mapsto \pi^{\tilde{\zeta}}(x_{j}^{k,{\zeta}}(y))$ from $\Pi^\zeta$ into $\Pi^{\tilde\zeta}$ is affine and invertible. Moreover, its linear part is independent of $\omega$, $k$, and $j$, since it depends only on  $\zeta$, $\bar{\zeta}$, and $\tilde{\zeta}$. By Lemma \ref{h0 slices} we then have
 \begin{equation*}
     \int_{\Pi^{\widetilde\zeta}}\hzero((J\cap L_{j})^{\widetilde\zeta}_{z})\,{\rm d}\huno(z)\leq \huno(J\cap L_j).
 \end{equation*}

 Using the map  $y\mapsto \pi^{\widetilde{\zeta}}(x_{j}^{k,{\zeta}}(y))$ as change of variables,  from the previous inequality and \eqref{stima fini} it follows that
 \begin{equation}\notag
     \int_{\Pi^\zeta}\mathcal{H}^{0}((x^{k,\zeta}_j(y)+\tfrac{1}{k}S_h)\cap J)\,{\rm d}\huno(y)\leq  \gamma\huno( J\cap L_j)
 \end{equation}
 for  every  $j\in\Z$, $k\in\N$, and $\omega\in U$,  with $\gamma>0$,  a constant depending only on $\zeta$, $\bar\zeta$, and $\tilde\zeta$.
 Observing that every point $x\in\R^2$ belongs at most to three strips of the form $L_j$, from the previous inequality it follows that   
 \begin{equation*}\label{somma fin}
      \sum_{j\in\Z}\int_{\Pi^\zeta}\mathcal{H}^{0}((x^{k,\zeta}_j(y)+\tfrac{1}{k}S_h)\cap J)\,{\rm d}\huno(y)\leq 3\gamma\huno( J).
 \end{equation*}
This inequality, together with \eqref{equazione h}, yields 
\begin{equation*}
    \int_{\Pi^\zeta}\mathcal{H}^0(\{j\in\Z:(x^{k,\zeta}_{j}(y)+\tfrac{1}{k}S_h)\cap  J\neq\emptyset\})\,{\rm d}\mathcal{H}^{1}(y)\leq 3\gamma\huno( J).
\end{equation*}
Since $\huno(J)<+\infty$, this  proves \eqref{claim con m} for $h\in\{3,\dots,6\}$.  Therefore, \eqref{claim con m} holds for every $h\in\{1,\dots, 6\}$. Thanks to  \eqref{integral brutto},  from \eqref{claim con m}  we obtain \eqref{claim chebishev}, which by \eqref{claim a-m} concludes the proof. 
\qed

\vspace{0.3 cm}

Before proving Lemma \ref{lemma vanishing integral}, we state a result about one-dimensional measures, which shows that non-atomic measures  satisfy a suitable uniform absolute continuity property.
\begin{lemma}\label{lemma:diffusenonatomistic}
    Let $I=[a,b]\subset \R$ be a bounded closed interval  and let $\mu\in\mathcal{M}^+_b(I)$ be a measure such that $\mu(\{t\})=0$ for every $t\in I$. Then for every $m\in\N$ and $\e>0$ there exists $\delta(\e,m)>0$ such that for every $\delta\in (0,\delta(\e,m))$ we have 
    \begin{equation*}
\mu\big(\bigcup_{\ell=1}^m(t_\ell-\delta,t_\ell+\delta)\cap I\big)\leq \e \quad \text{for every }(t_1,...,t_m)\in I^m.
    \end{equation*}
\end{lemma}
\begin{proof}
        We argue by contradiction. Suppose that there exist $m \in\N$ and $\e>0$ such that for a sequence $\delta_k>0$ converging to $0$ we have
        \begin{equation}\notag
            \mu\big(\bigcup_{\ell=1}^m(t^k_\ell-\delta_k,t^k_\ell+\delta_k)\cap I
            \big)> \e \quad \text{for some }(t^k_1,...,t^k_m)\in I^m.
        \end{equation}
         Given $\delta>0$, this implies that
         \begin{equation}\notag 
            \mu\big(\bigcup_{\ell=1}^m(t^k_\ell-\delta,t^k_\ell+\delta)\cap I
            \big)>\e \quad \text{ for all $k\in\N$ sufficiently large}.
         \end{equation}Since $I$ is compact, there exists a subsequence, not relabelled, and a point $(t_1,...,t_m)\in I^m$ such that $(t_1^k,...,t^k_m)$ converges to $(t_1,...,t_m)$ as $k\to +\infty$. Since $(t^k_\ell-\delta, t^k_\ell+\delta)\subset (t_\ell-2\delta,t_\ell+2\delta)$ for $k$ large
         we deduce that 
         \begin{equation*}
\mu\big(\bigcup_{\ell=1}^m(t_\ell-2\delta,t_\ell+2\delta)\cap I
            \big)>\e. 
         \end{equation*}
         Since $\delta>0$ is arbitrary, we obtain 
         \begin{equation*}
 \mu\big(\bigcup_{\ell=1}^m\{t_\ell\}\big)\geq \e,
         \end{equation*}
         in contradiction with our hypotheses. This concludes the proof.
    \end{proof}

\noindent {\it Proof of Lemma \ref{lemma vanishing integral}.}
 We begin by noting that for $\huno$-a.e.\ $y\in \Pi^\zeta$ we have $u^{\zeta}_y\in BV(U^\zeta_y)$ by Remark \ref{remar: finiteness}
 and,  recalling \eqref{J1 rectifiable} and  \eqref{def Jhat}, by Theorem \ref{fine propr of GBD} we have also $J_{u^\zeta_y}\cap U^\zeta_y\subset  J^\zeta_y=(J^1_u)^\zeta_y$. Let us fix $y\in \Pi^\zeta$ such that these two conditions hold. By standard properties of $BV$-functions in dimension one we have $Du^\xi_y(\{t\})=0$ for every $t\in (U\setminus J)^\zeta_y$. We recall that by definition $\hat{E}^{k,\zeta}_m(y)$ is the union of at most $m$ intervals of length $\tfrac{1}{k}$ (see \eqref{Ikzy Nkzy}). Hence, Lemma \ref{lemma:diffusenonatomistic},  applied to the measure $\mu:=|Du^\xi_y|\mres (U^\zeta_y\setminus J^\zeta_y)$ defined on the closure of the interval $U^\zeta_y$, implies that  \eqref{trascu 2} holds.

We observe that the integral in \eqref{trascurabile 1} is well-defined, since by Lemma \ref{lemma measurable R2} the set $\hat E^{k,\zeta}_m$ is Borel measurable and $\hat{E}^{k,\zeta}_m(y)$ is its slice by \eqref{slice of hats}. Since  $|Du^\zeta_y|(\hat E^{k,\zeta}_m(y)\cap U^\zeta_y\setminus  J^{\zeta}_y)\leq |Du^\zeta_y|(U^\zeta_y\setminus   J^{\zeta}_y)$ for $\huno$-a.e.\ $y\in \Pi^\zeta$ and the function $y\mapsto|Du^\zeta_y|(U^\zeta_y\setminus   J^{\zeta}_y)$  is integrable on $\Pi^\zeta$ with respect to $\huno$ by  \eqref{def:GBD con J1}, the Dominated Convergence Theorem implies \eqref{trascurabile 1}, concluding the proof.\qed
 
\vspace{0.3 cm}

For technical reasons, instead of a single $\omega\in U$ 
we have to consider a sequence $(\omega_k)_k\subset U$; accordingly, the points $y^{k,\zeta}_j$, introduced in \eqref{1 parameter family}, and the sets $\hat E^{k,\zeta}_m(y)$ and $\check E^{k,\zeta}_m(y)$ are defined by taking $\omega=\omega_k$. 
\begin{lemma}\label{fundamental lemma} 
For every $\e>0$ there exist $m_\e\in\N$ and an infinite set $K_\e\subset \N$ with the following property: for every $k\in K_\e$ there exists a Borel set $U^k_\e\subset U$, with $\Lb^2(U\setminus U^k_\e)\leq \e$, such that  for every  $\zeta\in\{\xi,\eta,\xi+\eta,\xi-\eta\}$ and for every sequence $(\omega_k)_{k\in  K_\e}$, with  $\omega_k\in U^k_\e$ for every $k\in K_\e$, the following conditions are simultaneously satisfied:

\begin{gather}
\label{condizione infernale}
\frac{1}{k}\sum_{j\in\Z}|Du^{\zeta}_{\!y^{k,\zeta}_{j}}|(\check{E}^{k,\zeta}_{m_\e}(y_j^{k,\zeta})\cap U^\zeta_{y^{k,\zeta}_j}\setminus  J^{\zeta}_{\!y^{k,\zeta}_{j}})\leq \e\quad \text{ for every $k\in K_\e$},\\
     \label{condizione che sembrava infernale}
  \underset{k\in K_\e}{\lim_{k\to+\infty}} \,
  \frac{1}{k}\sum_{j\in\Z}|Du^{\zeta}_{\!y^{k,\zeta}_{j}}|(\hat{E}^{k,\zeta}_{m_\e}(y_j^{k,\zeta})\cap U^\zeta_{y^{k,\zeta}_j}\setminus  J^{\zeta}_{\!y^{k,\zeta}_{j}})=0.
 \end{gather}

\end{lemma}
\begin{proof}
Let us fix $\zeta\in\{\xi,\eta,\xi+\eta,\xi-\eta\}$ and $\e>0$.  Arguing as in Lemma \ref{fundamental lemma 2} we see that it is enough to prove that there exist $m_\e\in\N$ and an infinite set $K_\e\subset \N$ with the following property: for every $k\in K_\e$ there exists a Borel set $U^k_\e\subset U$, with $\Lb^2(U\setminus U^k_\e)\leq \e$, such that  for every sequence $(\omega_k)_{k\in  K_\e}$, with  $\omega_k\in U^k_\e$ for every $k\in K_\e,$ conditions \eqref{condizione infernale} and \eqref{condizione che sembrava infernale} hold for this particular $\zeta$.

 We observe that every $\omega\in\R^2$ can be written in a unique way as $\omega=z_1 \bar{\zeta}+z_2\zeta$, where $z_1,z_2\in\R$ and $\bar\zeta$ is defined by \eqref{eq: def zeta check}. By \eqref{def intersections coord} the numbers $t^{k,\zeta}_i(y)$ depend on $\omega$ only through $z_2$, hence the same holds for $x^{k,\zeta}_j(y)$,  $\hat E^{k,\zeta}_m(y)$, and $\check E^{k,\zeta}_m(y)$ (see \eqref{def intersections} and  \eqref{6.23 bis}).  We also remark that there exists a  constant $c_1>0$ such that if $\omega\in U$ then $|z_1|\leq c_1$.
   Moreover, there exists a constant $c_2>0$ such that, if $B\subset \R^2$ is  a Borel set and $A=\{\omega\in U: \omega=z_1 \bar{\zeta}+z_2\zeta \text{ with }(z_1,z_2)\in B\}$, then 
  \begin{equation}\label{formula con czeta}
      \Lb^2(A)\leq c_2\Lb^2(B). 
  \end{equation}
 
 Let $C>0$ be the constant of Lemma \ref{lemma smallness of A>m}. Thanks to  Lemma \ref{lemma smallness of A>m}, we can find an infinite set $H_\e^\zeta\subset \N$ and a Borel set $I^\zeta_\e\subset I^\zeta$, with 
 \begin{equation}\label{L1 differenza}\Lb^1(I^\zeta\setminus I^\zeta_\e)\leq \e/{(4c_1c_2)},
 \end{equation} such that \eqref{claim lemma smallness} holds for every $m\in\N$, $k\in H_\e^\zeta$, and $\omega\in U_\e^k:=U\cap \{z_1\bar\zeta+z_2\zeta: z_1,\in \R,\, z_2\in I^\zeta_\e\}$. Moreover, let $N_\zeta\subset \Pi^\zeta$ be the $\huno$-negligible Borel set introduced at the beginning of the proof of Lemma \ref{fundamental lemma 1}.
 
 For every $m\in\N$, $k\in H_\e^\zeta$, $y\in \Pi^\zeta$, and $z_2\in\R$ we define 
    \begin{gather}\label{def hkm aux}
         h^{k,\zeta}_{m,\e}(y,z_2):=\begin{cases}
             |Du^\zeta_y|(\check E^{k,\zeta}_m(y)\cap U^\zeta_{y}\setminus  J^{\zeta}_y)&\text{ if $y\in \Pi^\zeta\setminus N_\zeta$ and  }z_2\in I^\zeta_\e,\\
             0 &\text{ if $y\in N_\zeta$ or $z_2\notin I^\zeta_\e$},
            \end{cases} 
     \end{gather}
     where $\check E^{k,\zeta}_m(y)$ is defined using $\omega$ of the form $z_1 \bar{\zeta}+z_2\zeta$ with an arbitrary $z_1\in\R$. We observe that  the function $h^{k,\zeta}_{m,\e}$ is Borel measurable on $\Pi^\zeta\times \R$ thanks to Lemma \ref{appendix principal}.  We also define   
     \begin{gather}\label{hzeta} g^\zeta(y):=\begin{cases}
            |Du^\zeta_y|((U\setminus  J)^\zeta_y)&\text{if $y\in\Pi^\zeta\setminus N_\zeta$},\\
             0 &\text{if  $y\in N_\zeta$}.
            \end{cases}
    \end{gather}
    Note that for every $y\in \Pi^\zeta$  we have $\check E^{k,\zeta}_m(y)\cap U^\zeta_y\setminus  J^\zeta_y\subset (U\setminus  J)^{\zeta}_y$,  hence $h^{k,\zeta}_{m,\e}(y,z_2)\leq g^\zeta(y)$ for every $y\in\Pi^\zeta$ and $z_2\in \R$. 
   
    We set  
    \begin{equation*}\label{chat} M^\zeta:=\huno(\pi^\zeta(U))+1.\end{equation*}
    By \eqref{def:GBD con J1} we have  $g^\zeta\in L^1(\Pi^\zeta,\huno)$, so that by the absolute continuity of the integral there exists $\delta_\e>0$ such that
    \begin{equation}\label{epsilondelta}
        \int_{B}g^\zeta(y)\,{\rm d}\huno(y)< \frac{\e^2|\pi^\zeta( \bar{\zeta})|^2}{4\max\{\alpha,\beta\}M^\zeta c_2}
    \end{equation}
    for every $\huno$-measurable set $B\subset \Pi^\zeta$  with $\huno(B)<\delta_\e$.

    Since by \eqref{slice of hats} and \eqref{def hkm aux}  the inequality $h^{k,\zeta}_{m,\e}(y,z_2)>0$ implies $y\in\pi^\zeta(\check E^{k,\zeta}_m)$ and $z_2\in I^\zeta_\e$,
    we have
    \begin{equation}\label{Fubinazzo}
    \begin{gathered}
    \int_\R\Big(\int_{\R} h^{k,\zeta}_{m,\e}(z_1\pi^\zeta( \bar{\zeta}) ,z_2)\,{\rm d}z_1\Big){\rm d}z_2=|\pi^\zeta( \bar{\zeta})|^{-1}\int_{{I}_\e^\zeta}\Big(\int_{\pi^\zeta({\check E^{k,\zeta}_m})}h^{k,\zeta}_{m,\e}(y,z_2)\,{\rm d}\huno(y)\Big)\,{\rm d}z_2\\\leq |\pi^\zeta( \bar{\zeta})|^{-1}\int_{{I}_\e^\zeta}\Big(\int_{\pi^\zeta(\check E^{k,\zeta}_m)}g^\zeta(y)\,{\rm d}\huno(y)\Big){\rm d}z_2.
    \end{gathered}
    \end{equation}
    
   Let us fix $m_\e\in\N$ such that $m_\e>C\delta^{-1}_\e$, where $\delta_\e$ is  a constant such that \eqref{epsilondelta} holds. By \eqref{claim lemma smallness} we have $\huno(\pi^\zeta(\Check{E}^{k,\zeta}_{m_\e}))< \delta_\e$,
    for every $k\in H_\e^\zeta$ and  $\omega=z_1 \bar{\zeta}+z_2\zeta$ with $z_2\in I^\zeta_\e$. Since $\Lb^1(I^\zeta_\e)\leq \Lb^1(I^\zeta)\leq \max\{\alpha,\beta\}$ by \eqref{def: Izeta}, from  \eqref{epsilondelta}  and  \eqref{Fubinazzo} we infer
that \begin{eqnarray}\label{quasifine lemma brutto}
 \int_\R\Big(\int_{\R} h^{k,\zeta}_{m_\e,\e}(z_1\pi^\zeta( \bar{\zeta}) ,z_2){\rm d}z_1\Big){\rm d}z_2<
\frac{\e^2|\pi^\zeta( \bar{\zeta})|}{4M^\zeta c_2}
  \end{eqnarray}
  for every $k\in H_\e^\zeta$.

Let  $\mathcal{P}^{k,\zeta}:=\{j\in\Z:y^{k,\zeta}_j \text{ belongs to }\pi^\zeta(U)\}$. Since $\omega=z_1\bar\zeta+z_2\zeta$,   by \eqref{def projections}, \eqref{1 parameter family}, and \eqref{eq: def zeta check}  we have 
\begin{equation} \label{JDocasion}
    y^{k,\zeta}_j=z_1\pi^\zeta( \bar{\zeta})+\tfrac{j}{k}\pi^\zeta({ \bar{\zeta}}) \quad \text{ for }j\in\Z.
\end{equation}
We set  $A^\zeta:=\{(z_1,z_2)\in\R^2: \omega=z_1 \bar{\zeta}+z_2\zeta\in U\}$. 
  Setting $z=(z_1,z_2)$, by the Fubini Theorem we obtain that  \begin{gather*}\label{smallness double integral}
\int_{A^\zeta}\Big(\frac{1}{k}\sum_{j\in \Z} h^{k,\zeta}_{m_\e,\e}(y^{k,\zeta}_{j},z_2)\Big){\rm d}z= \frac{1}{k}\sum_{j\in {\mathcal{P}^{k,\zeta}}} \int_{A^\zeta}h^{k,\zeta}_{m_\e,\e}((z_1+\tfrac{j}{k})\pi^\zeta( \bar{\zeta}),z_2)\,{\rm d}z
        \\\leq 
       \frac{1}{k} \!\!\!\sum_{j\in {\mathcal{P}^{k,\zeta}}} \int_\R\Big(\int_\R h^{k,\zeta}_{m_\e,\e}(z_1\pi^\zeta( \bar{\zeta}),z_2)\,{\rm d} z_1\Big){\rm d}z_2= 
        \frac{\hzero({\mathcal{P}^{k,\zeta}})}{k}\!\!\int_\R\Big(\int_{\R} h^{k,\zeta}_{m_\e,\e}(z_1\pi^\zeta( \bar{\zeta}) ,z_2)\,{\rm d}z_1\Big){\rm d}z_2.
        \end{gather*}
      Since $\hzero({\mathcal{P}^{k,\zeta}})\leq kM^\zeta |\pi^\zeta( \bar{\zeta})|^{-1}$,   this inequality, together with \eqref{quasifine lemma brutto}, implies that  
\begin{equation}\label{orange is the new black}
    \int_{A^\zeta}\Big( \frac{1}{k}\sum_{j\in \Z} h^{k,\zeta}_{m_\e,\e}(y^{k,\zeta}_{j},z_2)  \Big){\rm d}z\leq \frac{\e^2}{4c_2}
\end{equation}
for every $k\in\ H_\e^\zeta$.

    We now set 
    \begin{equation}\label{def Ukme}
    V^{k,\zeta}_{\e}:=\Big\{z\in A^\zeta:\frac{1}{k}\sum_{j\in \Z} h^{k,\zeta}_{m_\e,\e}(y^{k,\zeta}_{j},z_2)>\e\Big\}\cup F^{k,\zeta},
    \end{equation}
    where 
    \begin{equation*}
        F^{k,\zeta}:=\Big\{z\in A^\zeta:{y^{k,\zeta}_j}\in N_\zeta \text{ for some }j\in \mathcal{P}^{k,\zeta}\Big\}.
    \end{equation*}
   Since $\huno(N_\zeta)=0$ by  \eqref{JDocasion} we obtain that $\Lb^2(F^{k,\zeta})=0$. 
  In light of {C}eby\v{s}\"{e}v's inequality and \eqref{orange is the new black}, we infer that
    \begin{equation}\notag
    \Lb^2(V^{k,\zeta}_\e)\leq \frac{\e}{4c_2}.
    \end{equation}
We also introduce the set 
\begin{equation}\label{Vtilde}
   W^{k,\zeta}_\e:=V^{k,\zeta}_\e\cup\Big\{z\in A^\zeta: z_2 \in I^\zeta \setminus I^{\zeta}_\e\Big\}
\end{equation}
and observe that, since $\Lb^2(\big\{z\in A^\zeta: z_2 \in I^\zeta \setminus I^{\zeta}_\e\big\})\leq c_1\Lb^1(I^\zeta\setminus I^{\zeta}_\e)\leq \e/(4c_2)$ by \eqref{L1 differenza},
  we have 
   \begin{equation}\label{stima vzeta}
\Lb^2(W^{k,\zeta}_\e)\leq \frac{\e}{2c_2}.
   \end{equation}
   It follows immediately from \eqref{def hkm aux}, \eqref{def Ukme}, and  \eqref{Vtilde} that 
   \begin{equation}\label{equality on vtilde}
    \frac{1}{k} \sum_{j\in\Z}|Du^{\zeta}_{\!y^{k,\zeta}_{j}}|(\check{E}^{k,\zeta}_{m_\e}(y_j^{k,\zeta})\cap U^\zeta_{y^{k,\zeta}_{j}}\setminus  J^{\zeta}_{\!y^{k,\zeta}_{j}}) =\frac{1}{k} \sum_{j\in \Z} h^{k,\zeta}_{m_\e,\e}(y^{k,\zeta}_{j},z_2)\leq \e
   \end{equation}
  for every $k\in H_\e^\zeta$ and  for  every  $z\in A^\zeta\setminus W^{k,\zeta}_\e$.

     For every $k,m\in\N$, $y\in\Pi^\zeta$, and $z_2\in\R$ we set 
\begin{equation}\label{def funzioni ausiliare in proof lemma}
        \begin{aligned}
         g^{k,\zeta}_{m}(y,z_2):=
            \begin{cases}
                |Du^\zeta_y|(\hat E^{k,\zeta}_m(y)\cap U^\zeta_y\setminus  J^\zeta_y)\, &\text{ if $y\in \Pi^\zeta\setminus N_\zeta$} \text{ and } z_2\in I^\zeta,\\
             0 &\text{ if $N_\zeta$ or $z_2\notin I^\zeta$}.
            \end{cases}
        \end{aligned}
    \end{equation}
    By Lemma \ref{appendix principal} the function $g^{k,\zeta}_m$ is Borel measurable on $\Pi^\zeta\times\R$. We also observe that by Lemma \ref{lemma vanishing integral} for every $z_2\in\R$ and for $\huno$-a.e.\ $y\in \pi^\zeta(U)$  the sequence $g_{ m}^{k,\zeta}(y,z_2)$ converges to $0$ as $k\to+\infty$. Arguing as before, we obtain that $g^{k,\zeta}_{ m}(y,z_2)\leq g^{\zeta}(y)$ for every $y\in \Pi^\zeta$ and  $z_2\in \R$, where $g^\zeta$ is the function defined by \eqref{hzeta}.

    We set 
    \begin{equation}\label{def scalare}
       f^{k,\zeta}_\e(z_1,z_2):= g^{k,\zeta}_{m_\e}(z_1\pi^\zeta( \bar{\zeta}),z_2)
    \end{equation}
    for every $(z_1,z_2)\in\R^2$ and observe that $f_\e^{k,\zeta}=0$ out of a suitable bounded set. 
    Using Lemma \ref{lemma: Riemann 2} and Remark \ref{remark sequences} with $\N$ replaced by $H_\e^\zeta$, we obtain an infinite set $K^\zeta_\e\subset H_\e^\zeta$ and  a Borel set $B^\zeta_\e\subset A^\zeta$, with $\Lb^2(B^\zeta_\e)\leq \e/(2c_2)$, such that for every $m\in\N$ \begin{equation}\label{uniform convergence}
\underset{k\in K_\e}{{\lim_{k\to+\infty}}}\frac{1}{k}\sum_{j
\in\Z} f_\e^{k,\zeta}(z_1+\tfrac{j}{k},z_2)=0 \,\,\,\text{uniformly for  } (z_1,z_2)\in A^\zeta\setminus B^\zeta_\e.
\end{equation}
We set $C^\zeta_\e:=\{\omega\in U: \omega=z_1 \bar{\zeta}+z_2\zeta \text{ with }(z_1,z_2)\in B^\zeta_\e\}$ and observe that from \eqref{formula con czeta} it follows that \begin{equation}\label{stima Czeta}
\Lb^2(C^\zeta_\e)\leq \frac{\e}{2}.
\end{equation}
Moreover, from \eqref{def funzioni ausiliare in proof lemma}, \eqref{def scalare}, and \eqref{uniform convergence} we deduce that 
\begin{equation}\label{new uniform}
    \underset{k\in K_\e}{{\lim_{k\to+\infty}}}\frac{1}{k}\sum_{j
\in\Z}  |Du^{\zeta}_{\!y^{k,\zeta}_{j}}|(\hat{E}^{k,\zeta}_{m_\e}(y_j^{k,\zeta})\cap U^\zeta_{y^{k,\zeta}_{j}}\setminus  J^{\zeta}_{\!y^{k,\zeta}_{j}})=0 \,\,\,\text{uniformly for  }\omega\in U\setminus C^\zeta_\e.
\end{equation}
We set $D_\e^{k,\zeta}=\{\omega\in U: \omega=z_1 \bar{\zeta}+z_2\zeta \text{ with }(z_1,z_2)\in  W^{k,\zeta}_\e \}$ and observe that by \eqref{formula con czeta}, \eqref{stima vzeta}, and \eqref{equality on vtilde} we have \begin{gather}\label{stima Dzeta}
    \Lb^2(D_\e^{k,\zeta})\leq \frac{\e}{2}\quad \text{ for every $k\in K^\zeta_\e$},\\
    \label{equality con dtilde}
    \frac{1}{k} \sum_{j\in\Z}|Du^{\zeta}_{\!y^{k,\zeta}_{j}}|(\check{E}^{k,\zeta}_{m_\e}(y_j^{k,\zeta})\cap U^\zeta_{y^{k,\zeta}_{j}}\setminus  J^{\zeta}_{\!y^{k,\zeta}_{j}})\leq \e\quad \text{for every $k\in K^\zeta_\e$ and  $\omega\in U\setminus D^{k,\zeta}_\e$}.
   \end{gather}
   
  Let $U^{k,\zeta}_\e:=U\setminus 
(C^\zeta_\e\cup D_\e^{k,\zeta})$. Combining  \eqref{stima Czeta} and \eqref{stima Dzeta} we obtain that $\Lb^2(U\setminus U^{k,\zeta}_\e)\leq  \e$.
Finally, from  \eqref{new uniform} and \eqref{equality con dtilde} it follows that 
  \eqref{condizione infernale} and \eqref{condizione che sembrava infernale} hold for every sequence $(\omega_k)_{k\in K^\zeta_\e}$ such that  $\omega_k\in U^{k,\zeta}_\e$ for every $k\in K^\zeta_\e$.
This concludes the proof.
\end{proof}

We are finally ready to prove Theorem \ref{vanishing indices}.

\medskip

\noindent {\it Proof of Theorem \ref{vanishing indices}.} Recalling \eqref{def bad} and  \eqref{union of segments} for every $k\in\N$ we have\begin{equation}
   \begin{gathered}\label{aaa}
       \mathcal{B}^k=\big\{(i,j)\in \mathcal{J}^k_2: (x^{k}_{i,j}+\tfrac{1}{k}S)\cap  J\neq \emptyset\}.
   \end{gathered}
   \end{equation}
Since  by \eqref{true def z}  we have \begin{equation}\label{zeta=x}
       z^{k,\zeta}(x^k_{i,j})=x^{k}_{i,j}\end{equation}
 (see also the comments after \eqref{star}), from \eqref{def Akzeta} and \eqref{aaa} we deduce that for every $\zeta\in\{\xi,\eta,\xi+\eta,\xi-\eta\}$ we have 
 \begin{align}\notag \label{def alternatica bad} 
 &\displaystyle \mathcal{B}^{k} =\{(i,j)\in\mathcal{J}^k_2: x^{k}_{i,j}\in E^{k,\zeta}\}.
 \end{align}
For every $k,m\in\N$ and $\zeta\in\{\xi,\eta,\xi+\eta,\xi-\eta\}$  we set 
\begin{equation}
\begin{gathered}\label{def Bkmdiscrete}
\hat{\mathcal{B}}^{k,\zeta}_{m}:=\{(i,j)\in \mathcal{J}^k_2: x^{k}_{i,j}\in \hat E^{k,\zeta}_m\},\\
\check{\mathcal{B}}^{k,\zeta}_{m}:=\{(i,j)\in \mathcal{J}^k_2: x^{k}_{i,j}\in \check{E}^{k,\zeta}_m\},
\end{gathered}
\end{equation}
and observe that $\mathcal{B}^k=\hat{\mathcal{B}}^{k,\zeta}_{m}\cup \check{\mathcal{B}}^{k,\zeta}_{m}$.

Let us fix $0<\e<\Lb^2(U)/2$. For every $n\in\N$ we apply Lemma \ref{fundamental lemma 2}  with $\e$ replaced by $\e/2^n$  and obtain an infinite set $H_n$ and a Borel set $U_{n}$, with $\Lb^2(U\setminus U_n)\leq \e/{2^n}$, such that all conditions of Lemma \ref{fundamental lemma 2} hold with $U_\e$ and $K_\e$ replaced by $U_{n}$ and $H_n$. In the step $n+1$ we can replace $\N$ in Lemma \ref{fundamental lemma 2}  by $H_n$,  obtaining $H_{n+1}\subset H_n$ for every $n\in\N$. Then, we apply Lemma \ref{fundamental lemma}, with $\e$ replaced by $\e/2^n$ and $\N$ replaced by $H_n$, and we obtain $m_n\in\N$, with $m_{n+1}\geq m_{n}$,  an infinite set $K_n\subset H_n$, and for every $k\in K_n$ a Borel set $U^k_n\subset U$, with $\Lb^2(U\setminus U^k_n)\leq \e/2^n$, with the following property:  for every sequence $(\omega_k)_{k\in K_n}$ such that  $\omega_k\in U^k_n$ for every $k\in K_n$  and for every $\zeta\in\{\xi,\eta,\xi+\eta,\xi-\eta\}$ we have
\begin{gather}\label{aa 1}
\frac{1}{k}\sum_{j\in\Z}|Du^{\zeta}_{\!y^{k,\zeta}_{j}}|(\Check{E}^{k,\zeta}_{m_n}(y_{j}^{k,\zeta})\cap U^\zeta_{y^{k,\zeta}_{j}}\setminus  J^{\zeta}_{\!y^{k,\zeta}_{j}})\leq \frac{\e}{2^n}\text{ \,\, for every $k\in K_n$},\\\label{aa}
 \underset{k\in K_n}{\lim_{k\to+\infty}}\frac{1}{k}\sum_{j\in\Z}|Du^{\zeta}_{\!y^{k,\zeta}_{j}}|(\hat{E}^{k,\zeta}_{m_n}(y_{j}^{k,\zeta})\cap U^\zeta_{y^{k,\zeta}_{j}}\setminus  J^{\zeta}_{\!y^{k,\zeta}_{j}})=0.
  \end{gather}
Replacing $U^k_n$ by $U^k_n\cap U_n$ we obtain that $U_n^k\subset U_n$  and $\Lb^2(U_n\setminus  U^k_n)\leq \e/2^n$.   Moreover, repeating the same argument used  when we pass from $n$ to $n+1$, it is not restrictive to assume that $K_{n+1}\subset K_n$ for every $n\in\N$.
 
 By a diagonal argument, we can find an infinite set $K\subset \N$ and a strictly increasing  sequence $(k_n)_n\subset \N$  such that $K\cap [k_n,+\infty)\subset K_n$ for every $n\in\N$. For every $k\in K$ we also define $n_k\in\N$ as the largest integer such that $k_n\leq k$ and observe that $n_k\to+\infty$ as $k\to+\infty$. For $k\in K$ consider the set $U^k:=\bigcap_{n=1}^{n_k}U^k_n$ and observe that $\Lb^2(U^k)\geq \Lb^2(U)-2\e$ for every $k\in K$, hence  $U^
  k\neq \emptyset$. We fix $(\omega_k)_{k\in K}\subset U$ such that $\omega_k\in U^k$ for every $k\in K$. By \eqref{aa} we have  
  \begin{equation}\label{new aa}
       \underset{k\in K}{\lim_{k\to+\infty}}\frac{1}{k}\sum_{j\in\Z}|Du^{\zeta}_{\!y^{k,\zeta}_{j}}|(\hat{E}^{k,\zeta}_{m_n}(y_{j}^{k,\zeta})\cap U^\zeta_{y^{k,\zeta}_{j}}\setminus  J^{\zeta}_{\!y^{k,\zeta}_{j}})=0\quad \text{ for every $n\in\N$.}
  \end{equation}
  
 We now prove \eqref{da scartare 1} for $\zeta=\xi$. 
Let us fix  $(i,j)\in\mathcal{B}^k$ and $t\in I^{k,\xi}_{i,j}$. By \eqref{aaa} we have \begin{equation}\label{intersection final}
    (x_{i,j}^k+\tfrac{1}{k}S)\cap J\neq \emptyset.
\end{equation} Recalling \eqref{def intervals zeta},  we also have  $t^{k,\xi}_{i,j}\leq t < t^{k,\xi}_{i,j}+\frac{1}{k}$, where $t^{k,\xi}_{i,j}$ is defined in \eqref{def:coordinates}. Since $x^{k}_{i,j}\in\{\omega_k+\frac{i}{k}\xi+s\eta:s\in\R\}$, by \eqref{def:coordinates} and \eqref{def intersections coord} we have $t_i^{k,\xi}(y_j^{k,\xi})=t^{k,\xi}_{i,j}$, so that $t^{k,\xi}_{i}(y^{k,\xi}_j)\leq t < t^{k,\xi}_{i}(y^{k,\xi}_j)+\frac{1}{k}$. From \eqref{star} and \eqref{zeta=x} we deduce that 
\begin{equation*}
    z^{k,\xi}(y_j^{k,\xi}+t\xi)=z^{k,\xi}(y_j^{k,\xi}+t^{k,\xi}_{i}(y^{k,\xi}_j)\xi)=z^{k,\xi}(y_j^{k,\xi}+t^{k,\xi}_{i,j}\xi)=z^{k,\xi}(x^k_{i,j})=x^k_{i,j}.
\end{equation*}
In light of \eqref{intersection final}, this implies that 
\begin{equation*}
     (z^{k,\xi}(y_j^{k,\xi}+t\xi)+\tfrac{1}{k}S)\cap J\neq\emptyset,
\end{equation*}
 which, by  \eqref{def Akzeta}, implies that  $y_j^{k,\xi}+t\xi\in E^{k,\xi}$. Recalling \eqref{eq:definition notation slicing}, this is equivalent to $t\in E^{k,\xi}(y^{k,\xi}_j)$. Since $(i,j)\in\mathcal{J}_2^k$, this shows that 
\begin{equation}\label{inclusion intervals}
    I_{i,j}^{k,\xi}\subset  E^{k,\xi}(y^{k,\xi}_j)\cap U^\xi_{y^{k,\xi}_j}.
\end{equation}

If $(i,j)\in \hat{\mathcal{B}}^{k,\xi}_{m}$ for some $m\in\N$, it follows from \eqref{6.23 bis} and \eqref{def Bkmdiscrete} that $\mathcal{N}^{k,\xi}(x_{i,j}^k)\leq m$, where $\mathcal{N}^{k,\xi}$ is defined by \eqref{Ikzy Nkzy}. By \eqref{Nky+tzeta} we have $\mathcal{N}^{k,\xi}(x_{i,j}^k)=\mathcal{N}^{k,\xi}(y_{j}^k+t\xi)$ for every $t\in\R$, so that from  \eqref{inclusion intervals} we also deduce that 
\begin{equation}\label{inclusion intervals 1}
     I_{i,j}^{k,\xi}\subset  \hat{E}_m^{k,\xi}(y^{k,\xi}_j)\cap U^\xi_{y^{k,\xi}_j}.
\end{equation}
In a similar way we prove that if $(i,j)\in \check{\mathcal{B}}^{k,\xi}_{m}$ for some $m\in\N$ then 
\begin{equation}\label{inclusion intervals 2}
     I_{i,j}^{k,\xi}\subset  \check{E}_m^{k,\xi}(y^{k,\xi}_j)\cap U^\xi_{y^{k,\xi}_j}.
\end{equation}

For every $k,m\in\N$ and $j\in\Z$ we set $\mathcal{B}^{k}
(j):=\{i\in\Z:(i,j)\in \B^{k}\}$,  
$\hat{\mathcal{B}}^{k,\xi}_{m}(j):=\{i\in\Z:(i,j)\in \hat{\B}^{k,\xi}_{m}\}$, and $\check{\B}^{k,\xi}_{m}(j):=\{i\in\Z:(i,j)\in \check{\B}^{k,\xi}_{m}\}$.
From \eqref{inclusion intervals 1} and  \eqref{inclusion intervals 2} we deduce  that for every $k,m\in\N$ and  $j\in\Z$ we have
\begin{gather*}
    \bigcup_{i\in \hat{\B}^{k,\xi}_{m}\!(j)}I^{k,\xi}_{i,j}\subset \hat{E}^{k,\xi}_{{m}}(y^{k,\xi}_{j})\cap U^\xi_{y^{k,\xi}_j} \quad \text{ and }\quad 
    \bigcup_{i\in \check{\B}_{m}^{k,\xi}\!(j)}I^{k,\zeta}_{i,j}\subset \check{E}^{k,\xi}_{{m}}(y^{k,\xi}_{j})\cap U^\xi_{y^{k,\xi}_j}.
\end{gather*}
Therefore,  for every $k,m\in\N$ we have
\begin{gather}\label{sottosucc}
  \sum_{i\in \hat{\B}^{k,\xi}_{m}\!(j)}\!\!\!\!|Du^{\xi}_{y^{k,\xi}_{j}}|(I^{k,\xi}_{i,j}\setminus  J^{\xi}_{y^{k,\xi}_{j}})\leq |Du^{\xi}_{y^{k,\xi}_{j}}|(\hat{E}^{k,\xi}_{m}(y^{k,\xi}_{j})\cap U^\xi_{y^{k,\xi}_{j}}\setminus  J^{\xi}_{y^{k,\xi}_{j}}),\\\label{sottosucc 1}
      \sum_{i\in \check{\B}^{k,\xi}_{m}\!
(j)}\!\!\!\!|Du^{\xi}_{y^{k,\xi}_{j}}|(I^{k,\xi}_{i,j}\setminus  J^{\xi}_{y^{k,\xi}_{j}})\leq |Du^{\xi}_{y^{k,\xi}_{j}}|(\check{E}^{k,\xi}_{m}(y^{k,\xi}_{j})\cap U^\xi_{y^{k,\xi}_{j}}\setminus  J^{\xi}_{y^{k,\xi}_{j}}).
\end{gather}
From  \eqref{aa 1} and  \eqref{sottosucc}  for every $n\in\N$ we obtain that 
\begin{gather}\label{aa 4}
    \frac{1}{k}\sum_{j\in\Z}\sum_{i\in\check{ \B}^{k,\xi}_{m_n}\!(j)}|Du^{\xi}_{y^{k,\xi}_{i,j}}|(I^{k,\xi}_{i,j}\setminus  J^{\xi}_{y^{k,\xi}_{i,j}})\leq \frac{\e}{2^{n}}\quad \text{ for every   $k\in K$ with $k\geq k_n$},
\end{gather}
while  \eqref{new aa} and \eqref{sottosucc 1} give
\begin{equation}\label{conc}
    \underset{k\in K}{\lim_{k\to+\infty}}\,\frac{1}{k}\sum_{j\in\Z}\sum_{i\in \hat{\B}^{k,\xi}_{m_n}\!(j)}\!\!\!\!|Du^{\xi}_{y^{k,\xi}_{i,j}}|(I^{k,\xi}_{i,j}\setminus  J^{\xi}_{y^{k,\xi}_{i,j}})=0\quad  \text{for every $n\in\N$.  }
    \end{equation}

Combining \eqref{aa 4} and \eqref{conc}, we deduce that
\begin{equation}\notag \label{aa 5}
    \underset{k\in K}{\limsup_{k\to+\infty}}\,\,\,\,\frac{1}{k}\sum_{(i,j)\in \mathcal{B}^k}\displaystyle|Du^{\xi}_{y^{k,\xi}_{i,j}}|(I^{k,\xi}_{i,j}\setminus  J^{\xi}_{y^{k,\xi}_{i,j}})\leq  \frac{\e}{2^{n}}\quad \text{for every $n\in\N$},
\end{equation}
 which concludes the proof of 
\eqref{da scartare 1}  for $\zeta=\xi.$  The proof of \eqref{da scartare 1} for  $\zeta\in\{\eta,\xi-\eta,\xi+\eta\}$, as well as the proof of \eqref{da scartare 3} and \eqref{da scartare 4}, is analogous.

Equalities \eqref{da scartare nuova}-\eqref{da scartare nuova 2} can then be obtained by difference from \eqref{eq:riemannXi1}. This concludes the proof of Theorem \ref{vanishing indices}.
\qed

\section{The case of dimension \texorpdfstring{$d>2$}{}}\label{section d larger}

We now show that Theorem \ref{prop:mainingredient} can be extended to the general case  $d>2$.  This is done by means of a  Fubini-type argument.  To this aim, we present and prove a short lemma that shows that the measure $\lambda_u$ introduced in \eqref{def lambda u} does not charge Borel sets that are $\sigma$-finite with respect to $\hd$   and do not intersect the jump set. 
\begin{lemma}\label{lemma salvami}
    Let $d\geq 1$, let  $u\in GBD(\Omega)$, and let $B\subset \Omega$ be a Borel  set that is $\sigma$-finite with respect to $\hd$. Then  $\lambda_u(B\setminus J_u)=0.$
\end{lemma}
\begin{proof}
    It is not restrictive to assume that $\hd(B)<+\infty$. By \eqref{def lambda u} to prove the claim it is enough to show that for every $\xi\in\Sd$ the measure $\lambda^\xi_u$ defined by \eqref{eq:defmuxi} satisfies $
        \lambda^\xi_u(B\setminus J_u)=0.$
   Let us fix  $\xi\in\Sd$. Since $J_{u^\xi_y}\subset (J_u)^\xi_y$ for $\hd$-a.e.\ $y\in\Pi^\xi$ by Theorem \ref{fine propr of GBD},  it follows from \eqref{eq:defmuxi} that 
    \begin{equation}\label{inequiality lemma d>2}
      \lambda^\xi_u(B\setminus J_u)\leq \int_{\Pi^\xi}|Du^{\xi}_y|(B^\xi_y\setminus J_{u^\xi_y})\,{\rm d}\hd(y).
    \end{equation}
     Recalling that by assumption $\hd(B)<+\infty$,  Lemma \ref{h0 slices}  implies  that for $\hd$-a.e.\ $y\in \Pi^\xi$ the slice $B^\xi_y$ is a  finite set. By well-known properties of $BV$ functions of one variable (see \cite[Corollary 3.33]{AmbrosioFuscoPallara}), this implies that for $\hd$-a.e.\ $y\in \Pi^\xi$ we have $|Du^{\xi}_y|(B^\xi_y\setminus J_{u^\xi_y})=0$. By \eqref{inequiality lemma d>2} this equality gives  $\lambda^\xi_u(B\setminus J_u)=0$, concluding the proof.
\end{proof}

\begin{theorem}\label{main prop d>2}
    Let $d>2$, let $u\in GBD(\Omega)$, and let $B\subset \Omega$ be a  Borel set. Then the function $\xi\mapsto\sigma^\xi_u(B)$ is quadratic.
\end{theorem}
\begin{proof}
    By Proposition \ref{character quadratic} it is enough to show that $\xi\mapsto\sigma_u^\xi(B)$ is $2$-homogeneous,   satisfies the parallelogram identity, and  is lower bounded  in the sense of (c) of Proposition \ref{character quadratic}. Since by Proposition \ref{prop: 2 homogeneity} the function $\xi\mapsto\sigma_u^\xi(B)$ is $2$-homogeneous and by Remark \ref{coincidenza sigma lambda} it satisfies the correct lower bound,  we only need to prove the parallelogram identity.     
    
    We decompose $\sigma^\xi_u(B)$ as 
\begin{equation}\nonumber \label{eq:decomposing sigma}
\sigma^\xi_u(B)=\sigma^\xi_u(B\setminus J_u)+\sigma^\xi_u(B\cap J_u)
\end{equation}
and observe that by Propositions \ref{lemma:jumps smaller than 1}  the function
 $\xi\mapsto\sigma_u^\xi(B\cap J_u)$ is quadratic.
Thus, to conclude we only need to prove that $\xi\mapsto\sigma_u^{\xi}(B\setminus J_u)$ satisfies the parallelogram identity.

 Let $\xi,\eta\in\Rd$ be two linearly independent vectors and consider the 2-dimensional vector space $V$  generated by $\xi$ and $\eta$. Let $\pi_V\colon \Rd\to V$ be the orthogonal projection  onto $V$. For $z\in \Rd$ and  $E\subset \Rd$ let $E^V_z:=\{y\in V:z+y\in E\}=V\cap (E-z)$ and let $u^V_z\colon \Omega^V_z\to V$ be the function defined for every $y\in\Omega^V_z$ by  $u^V_z(y):=\pi_V(u(z+y))$.  By  (a) of Theorem \ref{theorem restriction} for $\mathcal{H}^{d-2}$-a.e.\ $z\in V^{\perp}$  we have $u^V_z\in GBD(\Omega^V_z)$.

 We observe that for every $\zeta\in \{\xi,\eta,\xi+\eta,\xi-\eta\}\subset V$ and   for every $E\subset \Rd$ we have  
 \begin{equation}\label{uguaglianza traslati}(E^V_z)^\zeta_y= E^\zeta_{z+y} \quad \text{for every $z\in V^\perp$ and $y\in V\cap \Pi^\zeta$}.
 \end{equation}
 Moreover, for every $z\in V^\perp$, $y\in V\cap \Pi^\zeta$, and  $t\in(\Omega^V_z)^\zeta_y= \Omega^\zeta_{z+y}$ we have \begin{equation}\label{uguaglianza funzioni tralsate}
    (u^V_z)^{\zeta}_{y}(t)=u^{\zeta}_{z+y}(t).
    \end{equation} 
This implies that  for $\mathcal{H}^{d-2}$-a.e.\ $z\in\ V^\perp$ and for $\mathcal{H}^{1}$-a.e.\ $y\in V\cap \Pi^\zeta$ we have
    \begin{equation}\label{fundamental larger that 2}
D(u^V_z)_y^\zeta=Du^{\zeta}_{z+y}
\end{equation}
as Borel measures on $(\Omega^V_z)^\zeta_y=\Omega^\zeta_{y+z}$. 

We then apply Theorem \ref{theorem restriction} to deduce that for $\mathcal{H}^{d-2}$-a.e.\ $z\in V^\perp$ there exists a Borel set $N_z\subset V$ such that 
\begin{equation}\label{Nzeta}
\text{$\mathcal{H}^{1}(N_z)=0$\quad \, and \quad  $J_{u^V_z}\subset (J_u)^V_z\cup N_z$}.
\end{equation}  
 By Remark \ref{re:RemarkAlmiTasso} the set $J_u$ is $\sigma$-finite with respect to $\hd$. Therefore, by \cite[Theorem 2.10.25]{Federer}  we have that $\huno((J_u)^V_z)$ is $\sigma$-finite with respect to $\huno$ for $\mathcal{H}^{d-2}$-a.e.\ $z\in V^\perp$. We can then apply Lemma \ref{lemma salvami}, with $d=2$,  $\Omega$ replaced by 
$\Omega^V_z$, $u$ replaced by $u^V_z$, and $B=(J_u)^V_z$, and we obtain that $\lambda_{u^V_z}((J_u)^V_z\setminus J_{u^V_z})=0$
for $\mathcal{H}^{d-2}$-a.e.\ $z\in V^\perp$. By Definition \ref{def:GBD}, applied to $u^V_z\in GBD(\Omega^V_z)$,  and from this equality it follows that 
\begin{equation}\label{zero lambda}
 \int_{V\cap \Pi^\zeta}|D(u^V_z)_y^\zeta|((J_u)^V_z\setminus J_{u^V_z})^\zeta_y)\,{\rm  d}\mathcal{H}^{1}(y)\leq \lambda_{u^V_z}((J_u)^V_z\setminus J_{u^V_z})=0
\end{equation}
for $\mathcal{H}^{d-2}$-a.e.\ $z\in V^\perp$.

Since  for every triple of sets $A_1,A_2,A_3$ we have
$(A_1\setminus A_2)\setminus (A_1\setminus A_3)\subset A_3\setminus A_2$, we obtain that 
\begin{equation}\nonumber 
    (B^V_z\setminus J_{u^V_z})\setminus (B^V_z\setminus (J_u)^V_z)\subset (J_{u})^V_z\setminus J_{u^V_z},
\end{equation}
hence by \eqref{zero lambda} 
\begin{equation}\label{zero setminus}
    |D(u^V_z)^\zeta_y|((B^V_z\setminus J_{u^V_z})^\zeta_y\setminus (B^V_z\setminus (J_u)^V_z)^\zeta_y)=0
\end{equation}
  for  $\mathcal{H}^{d-2}$-a.e.\ $z\in V^\perp$ and for  $\huno$-a.e.\ $y\in V\cap \Pi^\zeta$. 

 The inclusion in \eqref{Nzeta} implies that for  $\mathcal{H}^{d-2}$-a.e.\ $z\in V^\perp$ and for every $y\in V\cap \Pi^\zeta\setminus \pi^\zeta(N_z)$ we have $(J_{u^V_z})^\zeta_y\subset ((J_u)^V_z)^\zeta_y$, and hence $ (B^V_z\setminus (J_u)^V_z)^\zeta_y\subset (B^V_z\setminus J_{u^V_z})^\zeta_y$. Observing that  $\huno(\pi^\zeta(N_z))=0$ by the equality in \eqref{Nzeta},  we deduce from this inclusion and from  \eqref{zero setminus} that
\begin{equation*}
   D(u^V_z)^\zeta_y((B^V_z\setminus J_{u^V_z})^\zeta_y)=  D(u^V_z)^\zeta_y ((B^V_z\setminus (J_u)^V_z)^\zeta_y)
\end{equation*}
  for $\mathcal{H}^{d-2}$-a.e.\ $z\in V^\perp$ and for $\huno$-a.e.\ $y\in V\cap \Pi^\zeta$.
  
Integrating  this equality with respect to $y$ we obtain 
  that
\begin{equation*}
    \int_{V\cap \Pi^\zeta}D(u^V_z)_y^\zeta((B\setminus J_u)^V_z)^\zeta_y)\,{\rm  d}\mathcal{H}^{1}(y)= \int_{V\cap \Pi^\zeta}D(u^V_z)_y^\zeta((B_z^V\setminus J_{u^V_z})^\zeta_y)\, {\rm d}\mathcal{H}^{1}(y)
\end{equation*}
 for $\mathcal{H}^{d-2}$- a.e.\ $z\in V^\perp$,
so that, setting $y'=z+y$, by  \eqref{uguaglianza traslati}-\eqref{fundamental larger that 2} and the Fubini Theorem we have
\begin{align}\notag 
\int_{\Pi^\zeta}Du_{y'}^\zeta((B&\setminus J_u)_{y'}^\zeta)\, {\rm d}\mathcal{H}^{d-1}(y')\\\nonumber \label{conclusion}
&= \int_{V^\perp}\Big(\int_{V\cap \Pi^\zeta}D(u^V_z)_y^\zeta((B^V_z\setminus J_{u^V_z})^\zeta_y)\, {\rm d}\mathcal{H}^{1}(y)\Big){\rm d}\mathcal{H}^{d-2}(z).
\end{align}
 Taking into account the definition of $\sigma^\zeta_u$ (see \eqref{def sigma}), this last equality can be written as 
 \begin{equation}\label{sigmato2dsigma}
     \sigma_{u}^\zeta(B\setminus J_u)=\int_{V^\perp}\sigma_{u^V_z}^\zeta(B^V_z\setminus J_{u^V_z})\,{\rm d}\mathcal{H}^{d-2}(z).
 \end{equation}
 
 We may now  apply Theorem \ref{prop:mainingredient}  to the function $u^V_z\in GBD(\Omega^V_z)$ to obtain
 \begin{equation*}
     \sigma^{\xi+\eta}_{u^V_z}(B^V_z\setminus J_{u^V_z})+\sigma^{\xi-\eta}_{u^V_z}(B^V_z\setminus J_{u^V_z})=2\sigma^{\xi}_{u^V_z}(B^V_z\setminus J_{u^V_z})+2\sigma^{\eta}_{u^V_z}(B^V_z\setminus J_{u^V_z})
 \end{equation*}
 for $\mathcal{H}^{d-2}$-a.e.\ $z\in V^\perp$. Integrating this equality with respect to $z$ and exploiting \eqref{sigmato2dsigma} we deduce that 
\begin{equation*}
    \sigma^{\xi+\eta}_{u}(B\setminus J_u)+\sigma^{\xi-\eta}_{u}(B\setminus J_u)=2\sigma^{\xi}_{u}(B\setminus J_u)+2\sigma^{\eta}_{u}(B\setminus J_u).
\end{equation*}
This shows that the function  $\xi\mapsto\sigma^\xi_u(B\setminus J_u)$ satisfies the parallelogram identity, concluding the proof.
\end{proof}

\section{A matrix-valued measure associated to a \texorpdfstring{$GBD$}{}   function}\label{cantor section}
In this section, for every $u\in GBD(\Omega)$ we introduce a matrix-valued measure $\mu_u$ that generalises the distributional symmetric gradient $Eu$ of $BD(\Omega)$ functions. We then analyse some of its properties and deduce some useful consequences.  

\begin{theorem}\label{corollary}
    Let $d\geq 1$ and $u\in GBD(\Omega)$. Then there exists a measure $\mu_u\in \mathcal{M}_b(\Omega;\Rdsym)$ such that for every $\xi\in \Sd$ we have
    \begin{gather}\label{claim corollary}
\mu_u(B)\xi\cdot\xi=\sigma^\xi_u(B)=\lim_{R\to+\infty}D_\xi(\tau_R(u\cdot\xi))(B\setminus J^1_u)\quad \text{ for every Borel set }B\subset \Omega ,
    \end{gather}
    where  $\sigma^\xi_u$ is the measure defined in \eqref{def sigma} and $\tau_R$ are the truncation functions defined in \eqref{def tauR}.
    Moreover, the variation $|\mu_u|$ with respect to the operator norm in $\Rdsym$ satisfies \begin{gather}\label{mu j1}  |\mu_u|(J^1_u)=0,\\\label{total var corollary}
    |\mu_u|=\lambda_u\mres(\Omega\setminus J^1_u) \quad \text{ as Borel measures on $\Omega$,}
    \end{gather}
    where $\lambda_u$ is the positive measure defined by \eqref{def lambda u}. 
\end{theorem}
\begin{proof}
If $d=1$ these results follow from Remarks \ref{remark BD}, \ref{coincidenza sigma lambda}, and  Proposition \ref{teorema truncation 1}. We may thus assume that $d\geq 2$.  From Theorems  \ref{prop:mainingredient} and \ref{main prop d>2} it follows that the function $\xi\mapsto \sigma^\xi_u(B)$ is quadratic  for every Borel set $B\subset \Omega$. Thus, there exists a set function $\mu_u$ defined on the $\sigma$-algebra of all Borel subsets of $\Omega$ and with values in $\Rdsym$ such that 
    \begin{equation}\label{muxixiemisura}
        \sigma^\xi_u(B)=\mu_u(B)\xi\cdot\xi
    \end{equation}
    for every Borel set $B\subset \Omega$  and $\xi\in\Rd\setminus\{0\}$.
   Observing that $\sigma^\xi_u(B\cap J^1_u)=0$ by \eqref{def sigma},  we may apply Proposition \ref{truncation} to $B\setminus J^1_u$ and we  obtain 
   \begin{equation*}
\sigma^\xi_u(B)=\lim_{R\to+\infty}D_\xi(\tau_R(u\cdot\xi))(B\setminus J^1_u)
   \end{equation*}
   for every Borel set $B\subset \Omega$  and $\xi\in\Sd$.
  As \eqref{mu j1} is an obvious consequence of \eqref{total var corollary}, we are left with proving that $\mu_u\in\mathcal{M}_b(\Omega;\Rdsym)$ and that equality \eqref{total var corollary} holds.
  
   Since $B\mapsto\sigma^\xi_u(B)$ is a bounded scalar-valued Radon measure for every $\xi\in\Rd\setminus \{0\}$, it follows from \eqref{muxixiemisura} that the same property holds for $B\mapsto\mu_u(B)\xi\cdot\xi$. The polarisation identity then implies that $B\mapsto\mu_u(B)\xi\cdot \eta$ belongs to $\mathcal{M}_b(\Omega)$ for every $\xi,\eta\in\Rd$, hence $\mu_u\in\mathcal{M}_b(\Omega;\Rdsym)$.

    To prove \eqref{total var corollary} let us first show that 
    \begin{equation}\label{intermedio corollary}
        |\mu_u|\leq \lambda_u\mres 
        (\Omega\setminus J^1_u)\quad \text{ as Borel measures on $\Omega$.}
    \end{equation} To this aim, we observe that, since  $\mu_u$ takes values in $\Rdsym$,  for every Borel set $B\subset \Omega$  the operator norm $|\mu_u(B)|$ satsfies
    
    \begin{equation}\nonumber 
        |\mu_u(B)|=\sup_{\xi\in\Sd}|(\mu_u(B)\xi \cdot \xi)|=\sup_{\xi\in\Sd}|\sigma^{\xi}_u(B)|,
    \end{equation}
    so that by \eqref{def lambda u} and \eqref{lower bound} we have
    \begin{equation}\nonumber \label{dv6}
          |\mu_u|(B)=\sup\sum|\sigma^{\xi_i}_u(B_i)|\leq \sup\sum\lambda^{\xi_i}_u(B_i\setminus J^1_u)=\lambda_u(B\setminus J^1_u),
    \end{equation}
  where the {\it supremum} is taken over all finite Borel partitions $(B_i)_{i}$ of $B$ and all finite collections of vectors $(\xi_i)_i\subset \Sd$.
This shows \eqref{intermedio corollary}.

To prove the inequality 
\begin{equation}\label{secondinequality}
    |\mu_u|\geq \lambda_u\mres 
        (\Omega\setminus J^1_u)\quad \text{ as Borel measures on $\Omega$,}
\end{equation}
we argue as follows. Consider the measure $\lambda\in\mathcal{M}_b^+(\Omega)$ defined for Borel set $B\subset \Omega$  by
\begin{equation*}
    \lambda(B):=|\mu_u|(B)+\hd(B\cap J^1_u).
\end{equation*}
 Thanks to \eqref{intermedio corollary}, Lemma \ref{h0 slices}, Remark \ref{re:RemarkAlmiTasso}, and Theorem \ref{fine propr of GBD} it follows that $\lambda$ satisfies \eqref{eq:def GBD}. Since $\lambda_u$ is the minimal measure that satisfies \eqref{eq:def GBD} it follows that $\lambda_u(B)\leq \lambda(B)$ for every Borel set $B\subset \Omega$, which implies \eqref{secondinequality}. 
\end{proof}

\begin{remark}\label{derivative BD}
    By Remark \ref{remark BD} it follows immediately from \eqref{muxixiemisura} that if $u\in BD(\Omega)$, then 
   $ \mu_{u}=(Eu)\mres(\Omega\setminus J^1_u)$
    as Borel measures on $\Omega$.
\end{remark}

Given $u\in GBD(\Omega)$,  the Lebesgue Decomposition Theorem allows us to decompose the measure $\mu_u$ as the sum of a measure $\mu^a_u$, which is absolutely continuous with respect to $\Lb^d$,  and a measure $\mu^s_u$, which is singular with respect to $\Lb^d$. In the following definition, we introduce a further decomposition of $\mu_u$, which closely resembles the classical decomposition \eqref{decomposition Eu} of $Eu$  for a function $u\in BD(\Omega)$.
\begin{definition}\label{def: Cantor Measure}
    For $u\in GBD(\Omega)$ we introduce the measures $\mu^c_u,\mu^j_u\in\mathcal{M}_b(\Omega;\Rdsym)$, called the {\it Cantor} part  and the {\it jump} part of $\mu_u,$ defined for every Borel set $B\subset \Omega$  by
    \begin{gather*}
        \mu^c_u(B):=\mu^s_u(B\setminus J_u),\\
         \mu^j_u(B):=\mu^s_u(B\cap  J_u)=\mu_u(B\cap J_u).
    \end{gather*}
    Since $\mu_u=\mu^a_u+\mu^s_u$, we have $\mu_u=\mu^a_u+\mu^c_u+\mu^j_u$
   as Borel measures on $\Omega$.
\end{definition}
\begin{remark}\label{remark Eu mu}
    It follows from Remark \ref{derivative BD} that if $u\in BD(\Omega)$  then the measure $\mu_u^c$ of Definition \ref{def: Cantor Measure} coincides with the Cantor part $E^cu$ (see \cite[Definition 4.1]{AmbrCosciaDalM}) of the symmetrised gradient $Eu$.
\end{remark}
We recall that for every $\xi\in\R^d\setminus\{0\}$ and every $y\in\Pi^\xi$ such that $u^\xi_y\in BV_{\rm loc}(\Omega^\xi_y)$ and $Du^\xi_y\in \mathcal{M}_b(\Omega^\xi_y)$ we can consider the measures
$D^au^\xi_y\in \mathcal{M}_b(\Omega^\xi_y)$ and $D^su^\xi_y\in \mathcal{M}_b(\Omega^\xi_y)$, defined as the absolutely continuous and the singular part of 
$Du^\xi_y$ with respect to the one-dimensional Lebesgue measure, and the measures $D^cu^\xi_y\in \mathcal{M}_b(\Omega^\xi_y)$ and $D^ju^\xi_y\in \mathcal{M}_b(\Omega^\xi_y)$, defined
for every Borel set $B\subset \Omega^\xi_y$ by
\begin{gather*}
  D^cu^\xi_y(B):=D^su^\xi_y(B\setminus J_{u^\xi_y}),\\
         D^ju^\xi_y(B):=D^su^\xi_y(B\cap   J_{u^\xi_y})=Du^\xi_y(B\cap  J_{u^\xi_y}).
    \end{gather*}
  Since $Du^\xi_y=D^au^\xi_y+D^su^\xi_y$, we have $Du^\xi_y=D^au^\xi_y+D^cu^\xi_y+D^ju^\xi_y$ 
  as Borel measures on $\Omega^\xi_y$.

We now show that the measures $\mu^a_u$ and $\mu_u^j$ can be expressed as suitable integrals depending on the approximate symmetric gradient $\E u$ (see Theorem \ref{fine propr of GBD}) and $[u]$, respectively,
and that $\mu^c_u$ an be expressed by means of $D^cu^\xi_y$.
\begin{proposition}
\label{representation}
Let $u\in GBD(\Omega)$. Then
\begin{gather}\label{ac representation}
    \mu^a_u(B)=\int_B\E u\, {\rm d}x,\\\label{representation jump}
    \mu^j_u(B)=\int_{(J_u\setminus J^1_u)\cap B}[u]\odot \nu_u\, {\rm d}\hd,
\end{gather}
for every Borel set $B\subset \Omega$. Moreover, for every $\xi\in\Rd\setminus\{0\}$ and every Borel set $B\subset \Omega$  we have
\begin{gather}\label{representation cantor}
    \mu^c_u(B)\xi\cdot\xi=|\xi|\int_{\Pi^\xi}D^cu^{\xi}_y(B^\xi_y)\,{\rm d}\hd(y).
\end{gather}
\end{proposition}
\begin{proof}
 Let us fix a Borel set $B\subset \Omega$.   By definition of $\mu_u^a$ and of $\mu^j_u$,  it follows from the polarisation identity and from \eqref{muxixiemisura} that
    \begin{gather}\label{polarisation ac}
        \mu^a_u(B)\xi\cdot \eta=\frac{1}{4}((\sigma^{\xi+\eta}_u)^a(B)-(\sigma^{\xi-\eta}_u)^a(B)),\\\label{polarisation jump}
        \mu^j_u(B)\xi\cdot \eta=\frac{1}{4}(\sigma^{\xi+\eta}_u(B\cap J_u)-\sigma^{\xi-\eta}_u(B\cap J_u)),
    \end{gather}
    for every $\xi,\eta\in\Rd\setminus \{0\}$, where for a vector $\zeta\in\Rd\setminus \{0\}$ the measure $(\sigma_u^\zeta)^a$ is the absolutely continuous part of $\sigma^\zeta_u$ with respect to $\Lb^d$. By Lemma \ref{lemma A.4}  for every $\zeta\in\Rd\setminus \{0\}$ we have
    \begin{gather}\label{absolutely slicing}
  (\sigma^\zeta_u)^a(B)=|\zeta|\int_{\Pi^\zeta}D^au^{\zeta}_y(B^\zeta_y)\,{\rm d}\hd(y),
  \\\label{singular slicing}
  (\sigma^\zeta_u)^s(B)=|\zeta|\int_{\Pi^\zeta}D^su^{\zeta}_y((B\setminus J^1_u)^\zeta_y)\,{\rm d}\hd(y),
    \end{gather}
    where $(\sigma_u^\zeta)^s$ is the singular part of $\sigma_u^\zeta$ with respect to $\Ld$. In light of Theorem \ref{fine propr of GBD}, using \eqref{absolutely slicing}  and the Fubini Theorem, we deduce that
    \begin{equation}
        \nonumber (\sigma^{\zeta}_u)^a(B)=\int_\Omega\E u\,\zeta\cdot\zeta\, {\rm d}x
    \end{equation}
    for every $\zeta\in\Rd\setminus \{0\}$. Combining this  equality with \eqref{polarisation ac}, we obtain that
    \begin{equation*}
        \mu^a_u(B)\xi\cdot \eta=\int_{B}\E u\,\xi\cdot \eta\, {\rm d}x =\big(\int_{B}\E u\, {\rm d}x\big)\xi\cdot \eta
    \end{equation*}
    for every $\xi,\eta\in\Rd\setminus \{0\}$, 
    which proves \eqref{ac representation}.

    To prove \eqref{representation jump}, we observe that by Proposition \ref{lemma:jumps smaller than 1} for every $\zeta\in\Rd\setminus \{0\}$ we have
    \begin{equation*}
        \sigma^\zeta_u(B\cap J_u)z
        =\int_{(J_u\setminus J^1_u)\cap B}([u]\odot\nu_u)\zeta\cdot\zeta\,{\rm d}\hd.
    \end{equation*}
    Together with \eqref{polarisation jump}, this equality shows that
    \begin{equation}
        \notag \mu^j_u(B)\xi\cdot \eta=\int_{(J_u\setminus J^1_u)\cap B}([u]\odot\nu_u)\xi\cdot \eta\,{\rm d}\hd =\big(\int_{(J_u\setminus J^1_u)\cap B}([u]\odot\nu_u)\,{\rm d}\hd\big)\xi\cdot \eta
    \end{equation}
    for every $\xi,\eta\in\Rd\setminus \{0\}$. This proves \eqref{representation jump}.

To conclude, we fix $\xi \in\Rd\setminus \{0\}$ and a Borel set $B\subset \Omega$ and show that \eqref{representation cantor} holds. We observe that for every  $y\in\Pi^\xi$  such that \eqref{eq:jump set to sliced jump set} holds  we have $ |D^ju^\xi_y|((B\setminus J_u)^\xi_y)=0$, hence
\begin{equation}\notag
    D^su^\xi_y((B\setminus J_u)^\xi_y)= D^cu^\xi_y((B\setminus J_u)^\xi_y).
\end{equation}
Since \eqref{eq:jump set to sliced jump set}  holds for $\hd$-a.e.\ $y\in \Pi^\xi$, from this equality, \eqref{muxixiemisura}, \eqref{singular slicing}, and Definition \ref{def: Cantor Measure} it follows that 
\begin{equation}\label{AbCdE}
  \mu_u^c(B)\xi\cdot\xi=  (\sigma^\xi_u)^{s}(B\setminus J_u)=|\xi|\int_{\Pi^\xi}D^cu^\xi_y((B\setminus J_u)^\xi_y)\, {\rm d}\hd(y).
\end{equation}
Recalling Lemma \ref{h0 slices}, by Remark \ref{re:RemarkAlmiTasso}  for $\hd$-a.e.\ $y\in \Pi^\xi$ the set  $(J_u)^\xi_y$ is finite or countable. Recalling the properties of the derivatives  of $BV$ functions in dimension one, this implies that
$D^cu^\xi_y((J_u)^\xi_y)=0$  for $\hd$-a.e.\ $y\in \Pi^\xi$. Therefore \eqref{AbCdE} implies \eqref{representation cantor}.
\end{proof}
 The following corollary shows that in the previous results  we can replace $J^1_u$ by $J^r_u$ (see \eqref{def Jr}) for an arbitrary $r>0$ and that the absolutely continuous and the Cantor part of the corresponding  measure $\mu_{u,r}$ do not depend on $r$.
\begin{corollary}\label{main corollary}
      Let $d\geq 1$, $u\in GBD(\Omega)$, and $r>0$. Then there exists a measure $\mu_{u,r}\in \mathcal{M}_b(\Omega;\Rdsym)$ such that for every $\xi\in \Sd$ we have
    \begin{gather}\label{formula corollary}
\mu_{u,r}(B)\xi\cdot\xi=\lim_{R\to+\infty}D_\xi(\tau_R(u\cdot\xi))(B\setminus J^r_u)\quad \text{ for every Borel set }B\subset \Omega ,
    \end{gather}
    where $\tau_R$ are the truncation functions defined in \eqref{def tauR}.
  Moreover, setting $\mu^j_{u,r}:=\mu_{u,r}\mres J_u$, we have 
  \begin{gather}\label{representation jumpr}
      \mu^j_{u,r}(B)=\int_{(J_u\setminus J^r_u)\cap B}[u]\odot \nu_u\, {\rm d}\hd \quad \text{ for every Borel set $B\subset \Omega$.}
  \end{gather}
 Finally, 
  we have $\mu_{u,r}=\mu^a_u+\mu^c_u+\mu^j_{u,r}$  as Borel measures on $\Omega$.
\end{corollary}
\begin{proof}
    Let $v:=u/r$ and $\mu_{u,r}:=r\mu_v$. Using the equalities 
    \begin{equation}\label{Jru}
        \tau_R(v\cdot\xi)=\frac{1}{r}\tau_{(rR)}(u\cdot\xi)\quad \text{ and }\quad J^1_v=J^r_u,
    \end{equation}
   from \eqref{claim corollary} we obtain \eqref{formula corollary}.  Using the equalities $J^1_v=J^r_u$ and  $J_v=J_u$, from  \eqref{representation jump} we deduce \eqref{representation jumpr}. 
   
   Let $\mu^a_{u,r}$ and $\mu^s_{u,r}$ be the absolutely  continuous and the singular part of $\mu_{u,r}$ with respect to $\Lb^d$. By  \eqref{claim corollary} and \eqref{formula corollary},  with $B$ replaced by $B\setminus J_u$, we obtain $\mu_{u,r}(B\setminus J_u)=\mu_u(B\setminus J_u)$. This implies that $\mu^a_{u,r}(B)=\mu^a_u(B)$ and $\mu^s_{u,r}(B\setminus J_u)=\mu^s_{u}(B\setminus J_u)=\mu^c_u(B)$ for every Borel set $B\subset \Omega$. Hence,  $\mu_{u,r}(B)=\mu^a_{u,r}(B)+\mu^s_{u,r}(B\setminus J_u)+\mu_{u,r}^s(B\cap J_u)=\mu_{u}^a(B)+\mu^c_u(B)+\mu_{u,r}^j(B)$ for every Borel set $B\subset \Omega$.
\end{proof}

\begin{remark}\label{derivative BDsecond}
    Using the function $v:=u/r$, it follows from Remark \ref{derivative BD} that, if $u\in BD(\Omega)$, then  
   $ \mu_{u,r}=(Eu)\mres(\Omega\setminus J^r_u)$
    as Borel measures on $\Omega$.
\end{remark}

The following result shows that, in analogy with $E^cu$, the measure $\mu^c_u$ does not charge Borel sets which are $\sigma$-finite with respect to $\hd$. 
\begin{proposition}\nonumber \label{prop sigma finito}
    Let $u\in GBD(\Omega)$ and let $B$ be a Borel set that is $\sigma$-finite with respect to $\hd$. Then $|\mu^c_u|(B)=0$.
\end{proposition}
\begin{proof}
    It is not restrictive to assume that $\hd(B)<+\infty$. Let us fix $\xi\in\Sd$.  Thanks to Lemma \ref{h0 slices}, we have $\hzero(B^\xi_y)<+\infty$ for $\hd$-a.e.\ $y\in \Pi^\xi$. By the properites of one-dimensional $BV$ functions this implies that $D^cu^\xi_y(B^\xi_y)=0$ for $\hd$-a.e.\ $y\in \Pi^\xi$. Hence, by \eqref{representation cantor} we have $\mu^c_u(B)\xi\cdot\xi=0$. Since this is true for every $\xi\in\Sd$, we obtain $\mu^c_u(B)=0$. As this property holds  also for every Borel subset of $B$, we deduce that $|\mu^c_u|(B)=0$.
\end{proof}

The definition and properties of $\mu_u$ allow us to give a new characterisation of the space $GSBD(\Omega)$, originally defined by slicing. We recall that $GSBD(\Omega)$ (see \cite[Definition 4.2]{DalMasoJems}) is the space of all $u\in GBD(\Omega)$ such that for every $\xi\in\Sd$ and for $\hd$-a.e.\ $y\in\Pi^\xi$  we have
\begin{equation}\label{def GSBD}
u^\xi_y\in SBV_{\rm loc}(\Omega^\xi_y).\end{equation}  
\begin{theorem}\label{cantor sbd}
    Let $u\in GBD(\Omega)$. Then $u\in GSBD(\Omega)$ if and only if $\mu_u^c=0$.
\end{theorem}
\begin{proof}
 Assume that $\mu^c_u=0$ as a Borel measure on $\Omega$. Recalling the uniqueness of the disintegration of measures (see   \cite[Theorem 2.28]{AmbrosioFuscoPallara}),  it follows from \eqref{representation cantor}   that 
 for  every $\xi\in\Sd$  and for $\hd$-a.e.\ $y\in\Pi^\zeta $ we have $D^cu^{\xi}_y=0$ as a Borel measure on $\Omega^\xi_y$, i.e., \eqref{def GSBD} holds. By definition this implies that  $u\in GSBD(\Omega)$.

Conversely, if $u\in GSBD(\Omega)$ it follows from \eqref{representation cantor} that $\mu^c_u(B)\xi\cdot\xi=0$ for every $\xi\in\Sd$ and every Borel set $B\subset \Omega$. This implies that $\mu^c_u=0$. 
   \end{proof}

Combining this result with recent results of \cite{ChambolleCrismaleSad}, where new characterisations of the spaces $GBD(\Omega)$ and  $GSBD^p(\Omega)$, for  $p>1$, are obtained, we can give an analogous characterisation for $GSBD(\Omega)$. More precisely, we show that  an $\Lb^d$-measurable function $u\colon \Omega\to\Rd$ belongs to $GSBD(\Omega)$ if and only if \eqref{eq:def GBD} and   \eqref{def GSBD} hold for a suitable finite number of  directions $\xi\in\Sd$. 
\begin{theorem}\label{chambolleCrismale}
    Let $u\colon\Omega\to\Rd$ be an $\Ld$-measurable function. Assume that there exists an  orthonormal basis $\{\xi_i:i=1,...,d\}$ such that for every $\xi\in \Xi:=\{\xi_i:i=1,...,d\}\cup\{\xi_i+\xi_j:1\leq i\leq j\leq d\}$  the two following conditions hold:
    \begin{gather}
  \label{condition gsbd}
    u^{\xi}_y\in SBV_{\rm loc
    }(\Omega^\xi_y) \quad \text{for } \hd\text{-}a.e.\,\, y\in\Pi^{\xi},\\\notag
    \Lambda^{\xi}_u:=  \int_{\Pi^{\xi}}|Du^{\xi}_y|(\Omega^{\xi}_y\setminus J^1_{u^{\xi}_y})+\mathcal{H}^0(J^1_{u^{\xi}_y})\, {\rm d}\mathcal{H}^{d-1}(y)<+\infty.
    \end{gather}
   Then $u\in GSBD(\Omega)$ and, setting  $\Lambda:=\sum_{\xi\in \Xi}\Lambda^{\xi}_u$, there exists a constant $C_d>0$, depending only on $d$, such that
   \begin{equation}\label{eq:bound chacris}
       \lambda_u(\Omega)\leq C_d\Lambda,
   \end{equation}
   where $\lambda_u$ is the measure defined by \eqref{def lambda u}.
\end{theorem}
\begin{proof}
Since the inclusion  $u\in GBD(\Omega)$ and inequality \eqref{eq:bound chacris} follow directly from \cite[Theorem 1, Corollary 1]{ChambolleCrismaleSad},   to conclude we only need to show that $u\in GSBD(\Omega)$.

To prove this, we observe that from \eqref{representation cantor} it follows that given a Borel set $B\subset \Omega$
 we have that  
   \begin{align}\notag 
\mu_u^c(B)\xi\cdot\xi=|\xi|\int_{\Pi^\xi}D^cu^{\xi}_{y}(B^\xi_y)\, {\rm d}\hd(y)=0\label{improvement principal}
   \end{align}
    for  every $\xi\in\Xi$.
By \eqref{condition gsbd} this equality implies that $
     \mu^c_u(B)\xi\cdot\xi=0$ for every $\xi\in\Xi$. From  the polarisation identity  we obtain  $\mu^c_u(B)\xi_i\cdot\xi_j=0$ for every $i,j=1,...,d$ . Recalling that  $\{\xi_i\}_i$ is a basis of $\Rd$, we deduce that $\mu^c_u(B)=0$. Since this property holds for every Borel set $B\subset \Omega$, from Theorem \ref{cantor sbd}   we obtain that $u\in GSBD(\Omega)$, concluding the proof.
\end{proof}

\textbf{Acknowledgements.}
 This paper is based on work supported by the National Research Project PRIN 2022J4FYNJ “Variational methods for stationary and evolution problems
with singularities and interfaces” funded by the Italian Ministry of University and Research. 
The authors are members of GNAMPA of INdAM.

\noindent 
\begin{appendices}
    \section{Auxiliary results}\label{appendix}
    The purpose of this section is to show that the functions  $g^{k,\zeta}_{m}$ and $h^{k,\zeta}_{m,\e}$  defined in \eqref{def funzioni ausiliare in proof lemma} and \eqref{def hkm aux}   are Borel  measurable, and prove some general properties of measures defined by integration.
    Results similar to those we present here are already well-established in the existing literature. However, given the specific form of the functions we  study, it is not easy to  apply them directly to our case. For this reason, we give here a precise statement and a complete proof of the results we need.

    \subsection{Lebesgue decomposition of measures defined by integration} 

In this subsection  we consider measures defined on the slices of a set, depending on  a parameter $\omega\in\R^k$, and the measures that can be obtained by  integrating with respect to the parameters corresponding to the slices. We are  interested in   a formula for the Lebesgue decomposition of these measures.

We begin by a lemma concerning measurability conditions with respect to these parameters. Given $h, k\in\N$, a Borel set $B\subset \R^h\times \R^k$, and $\omega\in\R^k$ we set 
\begin{equation}\label{A.NUM}
B(\omega):=\{x\in\R^h: (x,\omega)\in B\}.
\end{equation}
    \begin{lemma}\label{appendix 2}
        Let $k\in\N$ and let $\zeta\in\R^d\setminus \{0\}$. For every $y\in\Pi^\zeta$ and $\omega\in\R^k$ let $\mu_y^{\omega}\in\mathcal{M}_b(\R)$ be a signed measure. The following three measurability conditions are equivalent:
        \begin{enumerate}
        \item [(a)] for every  $\psi\in C^0_c(\R\times \R^k)$ 
        \begin{equation*}
    \text{ the function }(y,\omega)\mapsto\int_{-\infty}^{+\infty}\psi(t,\omega)\, {\rm d}\mu_y^{\omega}(t) \text{  is Borel measurable on $\Pi^\zeta\times \R^k$;}
         \end{equation*}
         
            \item [(b)] for every  $\varphi\in C^0_c(\R^d\times \R^k)$ 
            \begin{equation*}
\text{ the function }(y,\omega)\mapsto\int_{-\infty}^{+\infty}\varphi(y+t\zeta,\omega)\, {\rm d}\mu_y^{\omega}(t) \text{ is Borel measurable on $\Pi^\zeta\times \R^k$; }
            \end{equation*}

            \item [(c)] for every Borel set $B\subset \R^d\times \R^k$ the function 
           $(y,\omega)\mapsto \mu^{\omega}_{y}(B(\omega)^\zeta_y)$
     is Borel measurable on  $\Pi^\zeta\times \R^k$.
        \end{enumerate}
    Moreover, if the previous conditions are satisfied, then for every Borel set $B\subset \R^d\times \R^k$ 
     the function  $(y,\omega)\mapsto |\mu^{\omega}_y|(B(\omega)^\zeta_y)$
     is Borel measurable on  $\Pi^\zeta\times \R^k$.
    \end{lemma}
    \begin{proof}  For simplicity of notation we consider only the case $\zeta=e_d$, the $d$-th vector of the canonical basis, the proof in the other cases being analogous.  For $x\in\Rd$ we set $x=(x^\prime,x_d)$, where $x^\prime:=(x_1,...,x_{d-1})$, and we  identify $\Pi^{e_d}=\R^{d-1}\times \{0\}$ with $\R^{d-1}$. Therefore the measure $\mu^{\omega}_{(x^\prime,0)}$ is denoted simply by $\mu^{\omega}_{x^{\prime}}$ and for the slicing of sets we use the notation $B_{x^\prime}$ instead of $B^{e_d}_{(x^\prime,0)}$.  
    
        Assume (a). For every pair of functions $w\in C^0_c(\R^{d-1}\times \R^k)$ and  $\psi\in C^0_c(\R\times \R^k)$ it follows from (a) that the function 
        \begin{equation*}
     (x^\prime,\omega)\mapsto \int_{-\infty}^{+\infty}w(x^\prime,\omega)\psi(x_d,\omega)\, {\rm d}\mu_{x^\prime}^{\omega}(x_d) \
        \end{equation*}
        is Borel measurable  on $\R^{d-1}\times \R^k$. 
        
        By an argument based on partition of unities corresponding to coverings by open sets with small diameter we can prove that         the class of linear combinations of functions $\varphi\in C^0_c(\R^d\times \R^k)$ of the form $\varphi(x,\omega)=w(x^\prime,\omega)\psi(x_d,\omega)$ is dense in $C^0_c(\R^d\times \R^k)$ with respect to the uniform convergence. This implies that 
        the function 
        \begin{equation*}
     (x^\prime,\omega)\mapsto \int_{-\infty}^{+\infty}\varphi((x^\prime,x_d),\omega)\, {\rm d}\mu_{x^\prime}^{\omega}(x_d) \
        \end{equation*}
        is Borel measurable  on $\R^{d-1}\times \R^k$ for every $\varphi\in C^0_c(\R^d\times \R^k)$, which is condition (b) in the case $\zeta=e_d$. 
        
       Assume (b) and consider the class $\mathcal{F}$ of bounded Borel functions $f\colon \R^d\times \R^k\to \R$ such that 
        \begin{gather*}
        (x^\prime,\omega)\mapsto\int_{-\infty}^{+\infty}f((x^\prime,x_d),\omega)\, {\rm d}\mu_{x^\prime}^{\omega}(x_d) \text{ is Borel measurable on $\R^{d-1}\times \R^k$}.
        \end{gather*}
        It is easy to check that $\mathcal{F}$ is a monotone class, that is,
        \begin{enumerate} 
            \item [(i)] if $(f_n)_n\subset\mathcal{F}$, with $f_n\leq g$, for some $g\in\mathcal{F}$, and  $f_n\nearrow f$, then $f\in\mathcal{F}$;
            \item [(ii)]if $(f_n)_n\subset\mathcal{F}$, with $f_n\geq g$, for some $g\in\mathcal{F}$, and  $f_n\searrow f$, then $f\in\mathcal{F}$.
        \end{enumerate}
        Moreover, thanks to (b), we have that $C^0_c(\R^d\times \R^k)\subset \mathcal{F}$. Hence, from the Monotone Class Theorem (see \cite[Section 3.14]{Williams})
         we deduce that for every bounded Borel function $f\colon\R^d\times\R^k\to \R$ the function 
        \begin{equation*}
            (x^\prime,\omega)\mapsto\int_{-\infty}^{+\infty}f((x^\prime,x_d),\omega)\, {\rm d}\mu_{x^\prime}^{\omega}(x_d)\end{equation*} is Borel measurable on $\R^{d-1}\times \R^k$. By taking as   $f$ the characteristic function of  a Borel set  $B\subset \R^d\times \R^k$ we obtain that the function
            $$
            (x',\omega)\mapsto \mu_{x^\prime}^{\omega}(B(\omega)_{x'}) 
            $$
  is Borel measurable  on $\R^{d-1}\times \R^k$, which is condition (c) in the case $\zeta=e_d$.

            Assume now (c). Let $E\subset  \R\times\R^k$ be a Borel set and let $B:=\R^{d-1}\times E\subset \R^d\times \R^k$. Clearly, $B(\omega)_{x^\prime}=E(\omega)$   for every $x^\prime\in \R^{d-1}$, so that by (c)  the function $(x^\prime,\omega)\mapsto \mu^{\omega}_{x^\prime}(E(\omega))$ is $\R^{d-1}\times \R^k$ measurable. Hence, 
            \begin{equation}\notag
                (x^\prime,\omega)\mapsto\int_{-\infty}^{+\infty}\chi_{E}(x_d,\omega)\,{\rm d}\mu^\omega_{x^\prime}(x_d)            \end{equation}
                is Borel measurable on $\R^{d-1}\times \R^k$.
           By linearity we obtain that 
           \begin{equation*}
              (x^\prime,\omega)\mapsto\int_{-\infty}^{+\infty} g(x_d,\omega)\,{\rm d}\mu^\omega_{x^\prime}(x_d)            \end{equation*}
              is Borel measurable on $\R^{d-1}\times \R^k$ for every simple function $g\colon\R \times \R^k\to \R$.
            Since every function $\psi\in C^0_c(\R\times \R^k)$ can be approximated by a uniformly bounded sequence of simple functions, an application of the Dominated Convergence Theorem yields (a).

           We now show that, if (a)-(c) hold, then the last part of the statement holds. By the equivalence of (a)-(c) for $|\mu^{\omega}_{x^\prime}|$, to conclude the proof it is enough to show that  for every function $\psi\in C^0_c(\R\times \R^k)$ the function
         \begin{equation} \label{measure app 1}
            (x^\prime,\omega)\mapsto\int_{-\infty}^{+\infty}
            \psi(x_d,\omega)\, {\rm d}|\mu_{x^\prime}^{\omega}|(x_d)
         \end{equation}
           is Borel measurable on $\R^{d-1}\times \R^k$. Assume for a moment that $\psi\geq 0$. By definition of total variation of a measure we have
           \begin{equation*}
               \int_{-\infty}^{+\infty}\psi(x_d,\omega)\, {\rm d}|\mu_{x^\prime}^{\omega}|(x_d)=\underset{|\varphi|\leq \psi}{\sup_{\varphi\in C^0_c(\R\times \R^k)}} \int_{-\infty}^{+\infty}\varphi(x_d,\omega)\, {\rm d}\mu_{x^\prime}^{\omega}(x_d).
           \end{equation*}
           Since the supremum above can be reduced to a countable dense subset of $C^0_c(\R\times \R^k)$, this equality, together with (a) for $\mu^{\omega}_{x^\prime}$, implies that the function in \eqref{measure app 1} is Borel measurable on $\R^{d-1}\times \R^k$ when $\psi\geq 0$. In the general case, one can split $\psi$ into its positive and negative part.
    \end{proof}

Let $\zeta\in\Rd\setminus \{0\}$ and for every $y\in \Pi^\zeta$  let $\mu_y\in \mathcal{M}_b(\Omega^\zeta_y)$. For every $y\in\Pi^\zeta$ let $\mu^a_y$ and $\mu^s_y$  be the absolutely continuous part and the singular part of $\mu_y$ 
with respect to the one dimensional Lebesgue measure. In the following lemma we consider some general conditions on $(\mu_y)_{y\in\Pi^\zeta}$ which allow us to define a  measure $\mu$ on $\Omega$  by integrating $\mu_y$ with respect to $y$.  We then show that the absolutely continuous part $\mu^a$ and its singular part $\mu^s$ of $\mu$ with respect to $\Ld$ can be obtained by integrating $\mu^a_y$ and $\mu^s_y$ with respect to $y$.   
  \begin{lemma}\label{lemma A.4} Let $\zeta\in\Rd\setminus \{0\}$ and for every $y\in \Pi^\zeta$  let $\mu_y\in \mathcal{M}_b(\Omega^\zeta_y)$.
  Assume that for every Borel set $B\subset \Omega$
  \begin{gather}\label{measurability condition}
     \text{  the function  $y\mapsto \mu_y(B^\zeta_y)$ is Borel measurable on $\Pi^\zeta$}
     \end{gather}
     and that there exists $g\in L^1(\Pi^\zeta,\hd)$ such that \begin{gather}\label{g boundante}
     \text{ $|\mu_y|(\Omega^\zeta_y)\leq g(y)$ for $\hd$-a.e.\ $y\in\Pi^\zeta$.} 
 \end{gather}
  Consider the measure defined for every Borel set $B\subset \Omega$ by
  \begin{equation}\label{def mu slice}
      \mu(B):=\int_{\Pi^\zeta}\mu_y(B^\zeta_y)\,{\rm d}\hd(y).
  \end{equation}
  Le $\mu^a$  and $\mu^s$  be the absolutely continuous part and the singular part of $\mu$  with respect to the Lebesgue measure $\Lb^d$. Then for every Borel set $B\subset \Omega$ the functions
  \begin{equation}\label{borel measurablity}
      \text{   $y\mapsto \mu^a_y(B^\zeta_y)$ and $y\mapsto \mu^s_y(B^\zeta_y)$ are Borel measurable  and $\hd$-integrable on $\Pi^\zeta$}
  \end{equation} and we have
  \begin{equation}
\label{absolute continuity}\mu^a(B)=\int_{\Pi^\zeta}\mu^a_y(B^\zeta_y)\,{\rm d}\hd(y) \quad \text{ and }\quad 
\mu^s(B)=\int_{\Pi^\zeta}\mu^s_y(B^\zeta_y)\,{\rm d}\hd(y)
  \end{equation}
      for every Borel set $B\subset \Omega$.
  \end{lemma}  
  \begin{proof}
       As in the previous lemma we consider only the case $\zeta=e_d$ and use the notation of the proof of the previous lemma.  We also drop the hypothesis that $\Omega$ is bounded and assume that $\Omega=\R^d$, as the result for a general $\Omega$ then easily follows.

    We now prove \eqref{borel measurablity}. To this aim let 
       \begin{gather}\label{def Sics}
\hspace{-0.2 cm}S:=\Big\{x\in\R^d: x=(x^\prime,x_d)\,\text{ and }\, \limsup_{\rho\to 0^+}\frac{|\mu_{x^\prime}|((x_d-\rho,x_d+\rho))}{2\rho}=+\infty\Big\}
       \end{gather}
       and,  for every $x^\prime \in\R^{d-1}$, let $S_{x^\prime}$ be the corresponding slice. 
 By the Besicovitch Derivation Theorem (see \cite[Theorem 2.22]{AmbrosioFuscoPallara}) for every $x^\prime\in\R^{d-1}$ we have
       \begin{equation}\label{singular absolute restrictions}
           \mu^s_{x^\prime}(B_{x^\prime})=\mu_{x^\prime}(B_{x^\prime}\cap S_{x^{\prime}})\quad \text{and }\quad \mu^a_{x^\prime }(B_{x^\prime})=\mu_{x^\prime}(B_{x^\prime}\setminus S_{x^\prime})
       \end{equation}
       for every Borel set $B\subset \R^d$.
       
       Therefore, by \eqref{measurability condition} to prove \eqref{borel measurablity} it is enough to show that the set $S$ is Borel measurable. To this end, we note that, as the function $\rho\mapsto|\mu_{x^\prime}|((x_d-\rho,x_d+\rho))$ is left-continuous, in the $\limsup$ in \eqref{def Sics} we can reduce to considering $\rho$ varying in a countable dense set. Hence, to conclude that $S$ is Borel measurable we only need to prove that for every $\rho>0$ the function \begin{equation}\label{def to be measurable}(x^\prime,x_d)\mapsto |\mu_{x^{\prime}}|((x_d-\rho,x_d+\rho))=\int_{\R}\chi_{(-\rho,\rho)}(x_d-t)\, {\rm d}|\mu_{x^\prime}|(t)\end{equation} is Borel measurable on $\R^{d-1}\times \R$, where $\chi_{(-\rho,\rho)}$ denotes the characteristic function of $(-\rho,\rho).$ Let $(\psi_n)_{n}\subset C^0_c(\R)$ be a sequence of functions with $\psi_n\leq \psi_{n+1}$ for every $n\in\N$ and converging pointwise to $\chi_{(-\rho,\rho)}$ as $n\to+\infty$. By the Monotone Convergence Theorem for every $x^{\prime}\in\R^{d-1}$ we have  
 \begin{equation}\label{monotone}
     \int_{\R}\chi_{(-\rho,\rho)}(x_d-t)\,{\rm d}|\mu_{x^\prime}|(t)=\lim_{n\to+\infty}\int_{\R}\psi_n(x_d-t)\,{\rm d}|\mu_{x^\prime}|(t).
 \end{equation}
For every $n\in\N$ the function $$(x^\prime,x_d)\mapsto\int_{\R}\psi_n(x_d-t)\,{\rm d}|\mu_{x^\prime}|(t)$$ is Borel measurable in $x^\prime$ for $x_d$ fixed,  thanks to \eqref{measurability condition} and Lemma \ref{appendix 2}, and continuous in $x_d$ for $x^\prime$ fixed. Thus, it is  Borel measurable in the product space $\R^{d-1}\times \R$. Thanks to \eqref{monotone}, this implies that the function \eqref{def to be measurable} is Borel measurable, which proves that the set $S$ is Borel measurable and concludes the proof of the measurability property in \eqref{borel measurablity}. The integrability follows from  \eqref{g boundante}.

Thanks to  \eqref{borel measurablity}  we can define two bounded Radon measures on $\R^d$ by setting 
\begin{gather*}
\nu_1(B):=\int_{\R^{d-1}}\mu^a_{x^\prime}(B_{x^\prime})\,{\rm d}\hd(x^\prime) \quad \text{ and }\quad 
\nu_2(B):=\int_{\R^{d-1}}\mu^s_{x^\prime}(B_{x^\prime})\,{\rm d}\hd(x^\prime)
\end{gather*} for every Borel set  $B\subset \R^d$. It follows  from the Fubini Theorem that $\nu_1$ is absolutely continuous with respect to $\Lb^d$ and that   $\Lb^d(S)=0$.  By \eqref{singular absolute restrictions}  we also deduce that $\mu^s_{x^\prime}(B_{x^\prime})=\mu^s_{x^\prime}(B_{x^\prime}\cap S_{x^\prime})$, hence $\nu_2(B)=\nu_2(B\cap S)$ 
 for every Borel set $B\subset \Omega$. This shows that $\nu_2$ is singular with respect to the Lebesgue measure. Since $\mu=\nu_1+\nu_2$, the equalities in \eqref{absolute continuity}  follow from the uniqueness of the Lebesgue decomposition. 
  \end{proof}
    
  \subsection{Measurability of the auxiliary functions used in Section \ref{section conclusion}}  
  In this subsection we prove the the measurability of the functions $g^{k,\zeta}_{m}$ and $h^{k,\zeta}_{m,\e}$ defined in \eqref{def funzioni ausiliare in proof lemma} and \eqref{def hkm aux}. 
    
    As in the proof of Theorem \ref{prop:mainingredient}, $\Omega$ is a bounded open set of $\R^2$, $u$ is a function in $GBD(\Omega)$ with $J^1_u=J_u$, $\xi$ and $\eta$ are  two linearly independent vectors in $ \R^2$ and $U$ is the parallelogram defined by \eqref{parallelogramU}. We keep $u$, $\xi,\eta$, and $U$ fixed throughout the rest of the subsection. 
     
    We also recall that $ J\subset U$ is the set  defined in \eqref{def Jhat}, and  that the sets $E^{k,\zeta}$, $\hat E^{k,\zeta}_m$, and $\check E^{k,\zeta}_m$ are  defined in \eqref{def Akzeta} and \eqref{def hatE checkE}. Since it will be important to keep track of the dependence of such sets on $\omega$, in the following we underline their dependence on $\omega$ by writing $E^{k,\zeta,\omega}$, $\hat{E}^{k,\zeta,\omega}_{m}$, and $\check{E}^{k,\zeta,\omega}_{m}$.

We introduce some sets which will play a crucial role in our arguments.  For every $k\in\N$, $m\in\N$, and $\zeta\in\{\xi,\eta,\xi+\eta,\xi-\eta\}$ we set 
    \begin{equation}\label{A3}
    \begin{gathered} 
        {\Eom}^{k,\zeta}:=\{(x,\omega)\in \R^2\times \R^2: x\in E^{k,\zeta,\omega}\}\subset \R^4,\\
        \hat{\Eom}^{k,\zeta}_{m}:=\{(x,\omega)\in \R^2\times \R^2: x\in \hat{E}^{k,\zeta,\omega}_{ m}\}\subset \R^4,\\
         \check{\Eom}^{k,\zeta}_{m}:=\{(x,\omega)\in \R^2\times \R^2: x\in \check{E}^{k,\zeta,\omega}_{m}\}\subset \R^4,\\
         \mathfrak{J}:=\{(x,\omega)\in\R^2\times \R^2:x\in  J\}\subset \R^4.
    \end{gathered}
    \end{equation}
    The following lemma addresses the Borel measurability of these sets. The proof is very similar to that of Lemma \ref{lemma measurable R2}, but for the sake of completeness we give here all details.
    \begin{lemma}\label{appendix 1}
        The sets ${\Eom}^{k,\zeta}$,  $\hat{\Eom}^{k,\zeta}_{m}$,  $\check{\Eom}^{k,\zeta}_{m}$, and $\mathfrak{J}$ are Borel measurable subsets of $\R^4$. 
    \end{lemma}
    \begin{proof}
    The property for $\mathfrak{J}$  is trivial. To prove the result for the other sets we
         consider the map $z^{k,\zeta}\colon \R^2\times\R^2\to\R^2$  defined by \eqref{def zeta} (see Figure \ref{figure parallelogram}), where the dependence on the variable $\omega$ is made clear by \eqref{def xkij}. By elementary geometrical arguments, it follows that $(x,\omega)\mapsto z^{k,\zeta}(x,\omega)$ is Borel measurable.
          
          For every set $B\subset \R^2$ we define
    \begin{gather*}
            {\Eom}^{k,\zeta}_{B}:=\{(x,\omega)\in \R^2\times \R^2:( z^{k,\zeta}(x,\omega)+\tfrac{1}{k}S)\cap B \neq \emptyset\},\\
         F_B:=\{z\in \R^2: (z+\tfrac{1}{k}S)\cap B\neq \emptyset\},
    \end{gather*}
    where $S$ is defined by \eqref{union of segments}.
   For a compact set $K\subset \R^2$ the set $F_K$ is closed, so that, observing that $\Eom_{K}^{k,\zeta}=\{(x,\omega)\in \R^2\times \R^2:z^{k,\zeta}(x,\omega)\in F_K\}$ and recalling that $z^{k,\zeta}$ is Borel measurable, we conclude that  $\Eom^{k,\zeta}_K$ is Borel measurable.
   By definition, ${\Eom}^{k,\zeta}=\Eom^{k,\zeta}_{ J}$. Since  $ J=\bigcup_{n\in\N}K_n$ with $K_n$ compact, this gives that $\Eom^{k,\zeta}=\bigcup_{n\in N}\Eom^{k,\zeta}_{K_n}$. Since the sets $\Eom^{k,\zeta}_{K_n}$ are Borel measurable, so is $\Eom^{k,\zeta}$. 

   To prove that $\hat{\Eom}^{k,\zeta}_{m}$ is Borel measurable, we observe that by \eqref{new Nkzeta} a pair $(x,\omega)$ belongs to $\hat{\Eom}^{k,\zeta}_{ m}$ if and only if the number of indices $i\in\Z$ such that $x\in E^{k,\zeta,\omega}-\tfrac{i}{k}\zeta$ is less than or equal to $m$. Hence, setting $\Eom_i^{k,\zeta}:=\{(x,\omega)\in \R^2\times \R^2:x\in E^{k,\zeta,\omega}-\tfrac{i}{k}\zeta\}$,  we have that 
     \begin{equation*}
     \textstyle \hat\Eom^{k,\zeta}_m=\big\{(x,\omega)\in \R^2\times \R^2: \sum_{i}\chi_{{\Eom}^{k,\zeta}_i}(x,\omega)\leq m\big\},
     \end{equation*}
     where $\chi_{{\Eom}^{k,\zeta}_i}$ is the characteristic function of ${\Eom}_i^{k,\zeta}$.
     Since the sets  $\Eom_i^{k,\zeta}$ are Borel measurable, we deduce that $\hat \Eom^{k,\zeta}_m$ is a Borel set as well. From the equality $\Check{\Eom}^{k,\zeta}_m= \Eom^{k,\zeta}\setminus \hat\Eom^{k,\zeta}_m$, we deduce that $\Check{\Eom}^{k,\zeta}_m$ is Borel measurable too, concluding the proof.
    \end{proof}

    We are now ready to state and prove the final result of this subsection.
    \begin{lemma}\label{appendix principal}
    Let $k,m\in\N$, let  $\zeta\in\{\xi,\eta,\xi+\eta,\xi-\eta\}$, and let  $N_\zeta\subset \Pi^\zeta$ be a Borel set such that $\hd(N_\zeta)=0$ and  $u^\zeta_y\in BV(U^\zeta_y)$ for every $y\in \Pi^\zeta\setminus N_\zeta$. Then the functions $g^{k,\zeta}_{ m}\colon \Pi^\zeta\times \R^2\to [0,+\infty)$ and   $  h^{k,\zeta}_{m}\colon \Pi^\zeta\times \R^2\to [0,+\infty)$ defined by
    \begin{gather} \label{ def g app}g^{k,\zeta}_{ m}(y,\omega):=
            \begin{cases}
                |Du^\zeta_y|(\hat{E}^{k,\zeta,\omega}_{m}(y)\cap U^{\zeta}_y\setminus  J^\zeta_y)&\text{ if $y\in \Pi^\zeta\setminus N_\zeta$},\\
             0 &\text{ if $y\in\ N_\zeta$},
             \end{cases}\\\label{def h app}
     h^{k,\zeta}_{m}(y,\omega):=\begin{cases}
             |Du^\zeta_y|(\Check{E}^{k,\zeta,\omega}_{m}(y)\cap U^{\zeta}_y\setminus J^{\zeta}_y)&\text{ if $y\in \Pi^\zeta\setminus N_\zeta$},\\
             0 &\text{ if $y\in N_\zeta$},
            \end{cases}    
    \end{gather}
  are Borel measurable on $\Pi^\zeta\times \R^2$.
    \end{lemma}
    \begin{proof}
        We begin by observing that by \eqref{A.NUM} and \eqref{A3} for every $y\in\Pi^\zeta$ and $\omega\in \R^2$ we have\begin{gather}\label{fin app}
         \hat{E}^{k,\zeta,\omega}_{m}(y)\cap U^\zeta_y\setminus J^\zeta_y=\hat{\mathfrak{E}}(\omega)^\zeta_y\cap U^\zeta_y\setminus \mathfrak{J}(\omega)^\zeta_y,\\\label{fin app 2}
          \Check{E}^{k,\zeta,\omega}_{m}(y)\cap U^\zeta_y\setminus  J^\zeta_y=\check{\mathfrak{E}}(\omega)^\zeta_y\cap U^\zeta_y\setminus \mathfrak{J}(\omega)^\zeta_y.
        \end{gather}
        By Lemma \ref{appendix 1}, the sets $\hat{\Eom}^{k,\zeta}_{ m}$, $\check{\Eom}^{k,\zeta}_{ m}$, and $\mathfrak{J}$ are Borel measurable subsets of $\R^2\times \R^2$, so that the same property holds for $\hat{\Eom}^{k,\zeta}_{ m}\setminus \mathfrak{J}$ and $\check{\Eom}^{k,\zeta}_{ m}\setminus \mathfrak{J}$.
        
        We want to apply Lemma \ref{appendix 2} with the measure $\mu^\omega_y\in\mathcal{M}_b(\R)$ defined for every Borel set $B\subset \R$ by 
        \begin{equation}\label{def muomega nellappendice}\mu^\omega_y(B):=\begin{cases}Du^\zeta_y(B\cap U^\zeta_y) \quad &\text{ if $y\in \Pi^\zeta\setminus N_\zeta$}, \\
            0 & \text{ $y\in N_\zeta$}.
            \end{cases}
        \end{equation}
        To show that condition (b) of Lemma \ref{appendix 2} holds we fix a function $\varphi\in C^\infty_c(\R^2\times \R^2)$ and observe that 
        \begin{equation*}
            \int_{-\infty}^{+\infty}\varphi(y+t\zeta,\omega){\,\rm d}\mu^{\omega}_y(t)=\int_{U^\zeta_y}\varphi(y+t\zeta,\omega){\,\rm d}Du^{\zeta}_y(t)
        \end{equation*}
        for every $y\in\Pi^\zeta\setminus N_\zeta$.
         Let us consider a bounded sequence $(\rho_n )\subset C^\infty_c(U)$ converging to $1$ on $\Omega$ as $n\to+\infty$ and  set $\varphi_n:=\rho_n\varphi\in C^\infty_c(U\times \R^2)$. For every $y\in\Pi^\zeta\setminus N_\zeta$  by the Dominated Convergence Theorem we have  
         \begin{equation}\label{dominated conv}
          \int_{-\infty}^{+\infty}\varphi(y+t\zeta,\omega){\,\rm d}\mu_y^\omega(t)=\lim_{n\to+\infty}\int_{U^\zeta_y}\varphi_n(y+t\zeta,\omega){\,\rm d}Du^{\zeta}_y(t).
         \end{equation}
        Moreover, integrating by parts, we have
        \begin{equation}\label{measur app 3}
    \int_{U^\zeta_y}\varphi_n(y+t\zeta,\omega)\,{\rm d}Du^{\zeta}_y(t)=-\int_{U^\zeta_y}u^\zeta_y(t)\nabla_x\varphi_n(y+t\zeta,\omega)\cdot\zeta\,{\rm d}t,
        \end{equation}
        where for a given $\omega\in \R^2$, we denote  by $\nabla_x\varphi_n(y+t\zeta,\omega)\in\R^2$ the vector whose two components are the partial derivatives of $\varphi_n$ with respect to $x$  at the point $(y+t\zeta,\omega)$. By the Fubini Theorem, it follows from \eqref{measur app 3}  that 
        \begin{equation*}
(y,\omega)\mapsto\int_{U^\zeta_y}\varphi_n(y+t\zeta,\omega)\,{\rm d}Du^{\zeta}_y(t) 
        \end{equation*} 
     is Borel measurable on $(\Pi^\zeta\setminus N_\zeta)\times \R^2$.
        It follows from \eqref{dominated conv} that the same property holds for
        \begin{equation}\label{sing}
            (y,\omega)\mapsto \int_{-\infty}^{+\infty}\varphi(y+t\zeta,\omega){\,\rm d}\mu_y^\omega(t).
        \end{equation}
        Since by definition $\mu^\omega_y(B)=0$ for $y\in N_\zeta$, it follows that the function \eqref{sing} is Borel measurable on $\Pi^\zeta\times \R^2$.
        As every function  in $ C^0_c(\R^2\times \R^2)$ can be approximated uniformly by a sequence of functions in $C^\infty_c( \R^2\times \R^2)$, the function \eqref{sing} is Borel measurable on $\Pi^\zeta\times \R^2$ for every $\varphi\in C^0_c(\R^2\times \R^2)$. Hence, $\mu^\omega_y$ satisfies condition (b) of Lemma \ref{appendix 2}.
        
        By Lemma \ref{appendix 2} applied to the measure $\mu^\omega_y$ defined by \eqref{def muomega nellappendice} and  with $B=\hat{\Eom}^{k,\zeta}_{ m}\setminus \mathfrak{J}$ and $B=\check{\Eom}^{k,\zeta}_{ m}\setminus \mathfrak{J}$ we obtain that the functions  $(y,\omega)\mapsto |Du^\zeta_y|(\hat{\mathfrak{E}}(\omega)^\zeta_y\cap U^\zeta_y\setminus \mathfrak{J}(\omega)^\zeta_y)$ and  $(y,\omega)\mapsto |Du^\zeta_y|(\check{\mathfrak{E}}(\omega)^\zeta_y\cap U^\zeta_y\setminus \mathfrak{J}(\omega)^\zeta_y)$ are Borel measurable on $(\Pi^\zeta\setminus N_\zeta)\times \R^2$. 
        By \eqref{ def g app}-\eqref{fin app 2} this implies that the functions $g^{k,\zeta}_{m}$ and $h^{k,\zeta}_m$ are Borel measurable on $\Pi^\zeta\times \R^2.$
    \end{proof}

\end{appendices}

\frenchspacing
\bibliographystyle{siam}
\bibliography{Sources}
\end{document}